\documentclass{iopart}
\pdfoutput=1
\usepackage{amssymb}
\usepackage{graphicx}
\usepackage{algorithmic}
\usepackage{algorithm}
\usepackage{url}
\usepackage{color}
\newcommand{\mbf}[1]{\mathbf{#1}}
\newcommand{\mcl}[1]{\mathcal{#1}}

\newcommand{\argmin}{\mathrm{arg}\!\min}
\newcommand{\ts}[1]{{\textstyle#1}}

\newcommand{\bq}{\begin{eqnarray}}
\newcommand{\eq}{\end{eqnarray}}

\newtheorem{theorem}{Theorem}[section]
\newtheorem{lemma}[theorem]{Lemma}

\newenvironment{proof}[1][Proof]{\begin{trivlist}
\item[\hskip \labelsep {\bfseries #1}]}{\end{trivlist}}

\newenvironment{remark}[1][Remark]{\begin{trivlist}
\item[\hskip \labelsep {\bfseries #1}]}{\end{trivlist}}

\newcommand{\qed}{\nobreak \ifvmode \relax \else
      \ifdim\lastskip<1.5em \hskip-\lastskip
      \hskip1.5em plus0em minus0.5em \fi \nobreak
      \vrule height0.75em width0.5em depth0.25em\fi}

\begin{document}

\title[A penalty method for inverse problems]{A penalty method for PDE-constrained optimization in inverse problems}
\author{T. van Leeuwen$^1$ and F.J. Herrmann$^2$}
\address{$^1$Mathematical Institute, Utrecht University, Utrecht, the Netherlands.\\
$^2$ Dept. of Earth, Ocean and Atmospheric Sciences, University of British Columbia, Vancouver (BC), Canada.}
\ead{T.vanLeeuwen@uu.nl}

\begin{abstract}
Many inverse and parameter estimation problems can be written as PDE-constrained optimization problems. The goal is to infer the parameters, typically coefficients of the PDE, from partial measurements of the solutions of the PDE for several right-hand-sides. Such PDE-constrained problems can be solved by finding a stationary point of the Lagrangian, which entails simultaneously updating the parameters and the (adjoint) state variables. For large-scale problems, such an \emph{all-at-once} approach is not feasible as it requires storing all the state variables. In this case one usually resorts to a \emph{reduced} approach where the constraints are explicitly eliminated (at each iteration) by solving the PDEs. These two approaches, and variations thereof, are the main workhorses for solving PDE-constrained optimization problems arising from inverse problems. In this paper, we present an alternative method that aims to combine the advantages of both approaches. Our method is based on a quadratic penalty formulation of the constrained optimization problem. By eliminating the state variable, we develop an efficient algorithm that has roughly the same computational complexity as the conventional \emph{reduced} approach while exploiting a larger search space. Numerical results show that this method indeed reduces some of the non-linearity of the problem and is less sensitive to the initial iterate.
\end{abstract}

\maketitle

\section{Introduction}
In inverse problems, the goal is to infer physical parameters (e.g., density, soundspeed or conductivity) from indirect observations. When the underlying model is described by a partial differential equation (PDE) (e.g., the wave-equation or Maxwell's equations), the observed data are typically partial measurements of the solutions of the PDE for multiple right-hand-sides. The parameters typically appear as coefficients in the PDE. These problems arise in many applications such as geophysics \cite{tarantola82,Haber2004,Epanomeritakis08,vanLeeuwen20133DFWI}, medical imaging \cite{Abdoulaev2005,Wang2015} and non-destructive testing.

For linear PDEs, the inverse problem can be formulated (after discretization) as a constrained optimization problem of the form
\bq
\label{eq:constr}
\min_{\mbf{m},\mbf{u}} \ts{\frac{1}{2}}||P\mathbf{u} - \mathbf{d}||_2^2  \quad 
\mbox{s.t.} \quad A(\mathbf{m})\mathbf{u} = \mathbf{q},
\eq
where $\mathbf{m}\in\mathbb{R}^{M}$ represents the (gridded) parameter of interest, $A(\mathbf{m})\in\mathbb{C}^{N\times N}$ and $\mathbf{q}\in\mathbb{C}^{N}$ represent the discretized PDE and source term, $\mathbf{u}\in\mathbb{C}^{N}$ is the state variable and $\mathbf{d}\in \mathbb{C}^{L}$ are the observed data. The measurement process is modelled by sampling the state with $P\in \mathbb{R}^{L\times N}$. Throughout the paper $^T$ denotes the (complex-conjugate) transpose.  

Typically, measurements are made from multiple, say $K$, independent experiments,
in which case $\mathbf{u} \in \mathbb{C}^{KN}$ is a block vector containing the state variables for all the experiments. Likewise, $\mathbf{q}\in \mathbb{C}^{KN}$ and $\mathbf{d}\in \mathbb{C}^{KL}$ are block vectors containing the right-hand-sides and observations for all experiments. The matrices $A(\mathbf{m})$ and $P$ will be block-diagonal matrices in this case. Typical sizes for $M,N,K,L$ for seismic inverse problems are listed in table \ref{table:sizes}.

In practice, one usually includes a regularization term in the formulation (\ref{eq:constr}) to mitigate the ill-posedness of the problem. To simplify the discussion, however, we ignore such terms with the understanding that appropriate regularization terms can be added when required.

\subsection{All-at-once and reduced methods}
In applications arising from inverse problems, the constrained problem (\ref{eq:constr}) is typically solved using the method of Lagrange multipliers \cite{Haber2000,Grote2014} or sequential quadratic programming (SQP) \cite{Dennis1998,Heinkenschloss2008}. This entails optimizing over both the parameters, states and the Lagrange multipliers (or adjoint-state variables) simultaneously. While such \emph{all-at-once} approaches are often very attractive from an optimization point-of-view, they are typically not feasible for large-scale problems since we cannot afford to store the state variables for all $K$ experiments simultaneously. Instead, the so-called \emph{reduced} approach is based on a (block) elimination of the constraints to formulate an unconstrained optimization problem over the parameters:
\bq
\label{eq:red}
\min_{\mbf{m}} \ts{\frac{1}{2}}||PA(\mathbf{m})^{-1}\mathbf{q} - \mathbf{d}||_2^2.
\eq
While this eliminates the need to store the full state variables for all $K$ experiments, evaluation of the objective and its gradient requires PDE-solves. Moreover, by eliminating the constraints we have dramatically reduced the search-space, thus arguably making it more difficult to find an appropriate minimizer. Note also that the dependency of the objective on $\mathbf{m}$ is now through $A(\mathbf{m})^{-1}\mathbf{q}$ in stead of through $A(\mathbf{m})\mathbf{u}$. For linear PDEs, the latter can often be made to depend linearly on $\mathbf{m}$ while the dependency of $A(\mathbf{m})^{-1}\mathbf{q}$ on $\mathbf{m}$ is typically non-linear.

\subsection{Motivation}
The main motivation for this work is the observation that the stated inverse problem would be much easier to solve if we had a \emph{complete} measurement of the state (i.e., $P$ is invertible). In this case, we could reconstruct the state from the data as $\mathbf{u} = P^{-1}\mathbf{d}$ and subsequently recover the parameter by solving
\bq
\label{eq:eqnerror}
\min_{\mathbf{m}} \textstyle{\frac{1}{2}}\|A(\mathbf{m})\mathbf{u} - \mathbf{q}\|_2^2,
\eq
which for many linear PDEs would lead to a linear least-squares problem. This approach is known in the literature as the \emph{equation-error approach} \cite{Richter1981,Banerjee2013}. When we do not have complete measurements, this method does not apply directly since we cannot invert $P$. We can, however, aim to recover the state by solving the following (inconsistent) overdetermined system
\bq
\label{eq:apuqd}
\left(\begin{array}{c}
A(\mathbf{m})\\
P
\end{array}\right)\mathbf{u}\approx
\left(\begin{array}{c}
\mathbf{q}\\
\mathbf{d}
\end{array}\right),
\eq
which combines the physics and the data. We can subsequently use the obtained estimate of the state to estimate $\mathbf{m}$ by solving (\ref{eq:eqnerror}). These steps can be repeated in an alternating fashion as needed. In a previous paper \cite{vanLeeuwen2013Penalty1}, we proposed this methodology for seismic inversion and coined it Wavefield Reconstruction Inversion (WRI). We showed -- via numerical experiments -- that this approach can mitigate some of the (notorious) non-linearity of the seismic inverse problem. In the current paper we seek to analyse this approach in more detail and broaden the scope of its application to inverse problems involving PDEs.

\subsection{Contributions and outline}
We give the above sketched method a sound theoretical basis by showing that it can be derived from a \emph{penalty} formulation of the constrained problem:
\bq
\label{eq:pen}
\min_{\mbf{m},\mbf{u}} \ts{\frac{1}{2}}||P\mathbf{u} - \mathbf{d}||_2^2 + \frac{\lambda}{2}\|A(\mathbf{m})\mathbf{u} - \mathbf{q}\|_2^2,
\eq
the solution of which (theoretically) satisfies the optimality conditions of the constrained problem (\ref{eq:constr}) as $\lambda\rightarrow\infty$.
Such reformulations of the constrained problem are well-known but are, to our best knowledge, not being applied to inverse problems. 

The main contribution of this paper is the development of an efficient algorithm based on the penalty formulation for inverse problems and the insight that we can approximate the solution of the constrained problem (\ref{eq:constr}) up to arbitrary finite precision with a finite $\lambda$. In (ill-posed) inverse problems, we often do not require very high precision because we can only hope to resolve certain components of $\mathbf{m}$ anyway. We can understand this by realizing that not all constraints are equally important; those that constrain the null-space components of $\mathbf{m}$ need not be enforced as strictly. Thus, the parameter $\lambda$ need only be large enough to enforce the dominant constraints. The numerical experiments suggest that a single fixed value of $\lambda$ is typically sufficient.

Our approach is based on the elimination of the state variable, $\mbf{u}$, from (\ref{eq:pen}). This reduces the dimensionality of the optimization problem in a similar fashion as the \emph{reduced} approach (\ref{eq:red}) does for the constrained formulation (\ref{eq:constr}) by solving $K$ systems of equations. This elimination leads to a cost function $\phi_{\lambda}(\mathbf{m})$ whose gradient and Hessian can be readily computed. The main difference is that the state $\mathbf{u}$ in this case is not defined by solving the PDE, but instead is solved from an overdetermined system that involves both the PDE \emph{and} the data (\ref{eq:apuqd}). Due to the special block-structure of the problems under consideration, this elimination can be done efficiently, leading to a tractable algorithm. Contrary to the conventional \emph{reduced} approach, the resulting algorithm does \emph{not} enforce the constraints at each iteration and arguably leads to a less non-linear problem in $\mbf{m}$. It is outside the scope of the current paper to give a rigorous proof of this statement, but we present some numerical evidence to support this conjecture.

The outline of the paper is as follows. First, we give a brief overview of the constrained and penalty formulations in sections \ref{lagrange} and \ref{penalty}. The main theoretical results are presented in section \ref{varpro} while a detailed description of the proposed algorithm is given in section \ref{algorithm}. Here, we also compare the penalty approach to both the all-at-once and the reduced approaches in terms of algorithmic complexity. Numerical examples on a 1D DC-resistivity and 2D ultrasound and seismic tomography problems are given in section \ref{examples}. Possible extensions and open problems are discussed in section \ref{discussion} and section \ref{conclusion} gives the conclusions.

\section{All-at-once and reduced methods}
\label{lagrange}

A popular approach to solving constrained problems of the form (\ref{eq:constr}) is based on the corresponding Lagrangian:
\bq
\label{eq:Lagrangian}
\mcl{L}(\mbf{m},\mbf{u},\mbf{v})=  \ts{\frac{1}{2}}||P\mathbf{u} - \mathbf{d}||_2^2 
+ \mbf{v}^T\!\left(A(\mathbf{m})\mathbf{u} - \mathbf{q}\right),
\eq
where $\mbf{v}\in\mathbb{C}^{KN}$ is the Lagrange multiplier or adjoint-state variable \cite{Nocedal,Haber2000}. A necessary condition for a solution $(\mathbf{m}^*,\mathbf{u}^*,\mbf{v}^*)$ of the constrained problem (\ref{eq:constr}) is that it is a stationary point of the Lagrangian, i.e. $\nabla\mcl{L}(\mathbf{m}^*,\mathbf{u}^*,\mbf{v}^*) = 0$. The gradient and Hessian of the Lagrangian are given by 
\bq
\nabla\mcl{L}(\mbf{m},\mbf{u},\mbf{v}) &=& 
\left(
\begin{array}{c}
\mcl{L}_{\mbf{m}}\\
\mcl{L}_{\mbf{u}}\\
\mcl{L}_{\mbf{v}}\\
\end{array}
\right)
=\left(
\begin{array}{c}
G(\mbf{m},\mbf{u})^T\mathbf{v},\\
A(\mathbf{m})^T\mathbf{v} + P^T\!(P\mathbf{u} - \mathbf{d}),\\
A(\mathbf{m})\mathbf{u} - \mathbf{q},
\end{array}
\right),
\eq
and
\bq
\nabla^2\mcl{L}(\mbf{m},\mbf{u},\mbf{v}) &=& 
\left(
\begin{array}{ccc}
R(\mbf{m},\mbf{u},\mbf{v})&K(\mbf{m},\mbf{v})^T&G(\mbf{m},\mbf{u})^T\\
K(\mbf{m},\mbf{v})&P^T\!P&A(\mbf{m})^T\\
G(\mbf{m},\mbf{u})&A(\mbf{m})&0\\
\end{array}
\right),
\eq
where
\bq
G(\mbf{m},\mbf{u}) =\frac{\partial A(\mbf{m})\mbf{u}}{\partial \mbf{m}},\,\,
K(\mbf{m},\mbf{v}) = \frac{\partial A(\mbf{m})^T\mbf{v}}{\partial \mbf{m}},\nonumber\\
R(\mbf{m},\mbf{u},\mbf{v}) = \frac{\partial G(\mbf{m},\mbf{u})^T\mbf{v}}{\partial \mbf{m}}\nonumber.
\eq
These Jacobian matrices are typically sparse when $A$ is sparse and can be computed analytically.


\subsection{All-at-once approach}
So-called all-at-once approaches find such a stationary point by applying a Newton-like method to the Lagrangian \cite{Haber2000,Haber2001}. A basic algorithm for finding a stationary point of the Lagrangian up to a given tolerance $\epsilon$ is given in Algorithm \ref{alg:allatonce}. If the algorithm converges it returns iterates $(\mathbf{m}^*,\mathbf{u}^*,\mbf{v}^*)$ such that $\|\nabla \mcl{L}(\mathbf{m}^*,\mathbf{u}^*,\mbf{v}^*)\|_2\leq \epsilon$.
\begin{algorithm}
\caption{Basic Newton algorithm for finding a stationary point of the Lagrangian via the all-at-once method}
\label{alg:allatonce}
\begin{algorithmic}
\REQUIRE{initial guess $\mbf{m}^0, \mbf{u}^0, \mbf{v}^0$, tolerance $\epsilon$}
\STATE{$k=0$}
\WHILE{$\|\nabla\mcl{L}(\mbf{m}^k,\mbf{u}^k,\mbf{v}^k)\|_2 \geq \epsilon$}
\vspace{1mm}
\STATE{$
\left(
\begin{array}{c}
\delta\mbf{m}^k\\
\delta\mbf{u}^k\\
\delta\mbf{v}^k\\
\end{array}
\right)
= -
\left(\nabla^{2}\mcl{L}(\mbf{m}^k,\mbf{u}^k,\mbf{v}^k)\right)^{-1}\nabla\mcl{L}(\mbf{m}^k,\mbf{u}^k,\mbf{v}^k)$}
\vspace{1mm}
\STATE{determine steplength $\alpha^k \in (0,1]$}
\vspace{1mm}
\STATE{$\mathbf{m}^{k+1} = \mathbf{m}^k + \alpha^k\delta\mbf{m}^k$}
\vspace{1mm}
\STATE{$\mathbf{u}^{k+1} = \mathbf{u}^k + \alpha^k\delta\mbf{u}^k$}
\vspace{1mm}
\STATE{$\mathbf{v}^{k+1} = \mathbf{v}^k + \alpha^k\delta\mbf{v}^k$}
\STATE{$k = k + 1$}
\ENDWHILE
\end{algorithmic}
\end{algorithm}
Many variants of Algorithm \ref{alg:allatonce} exist and may include preconditioning, inexact solves of the KKT system $\left(\nabla^2\mathcal{L}\right)^{-1}\nabla\mathcal{L}$ and a linesearch to ensure global convergence \cite{Haber2001,Biros,Grote2014}. For an extensive overview we refer to \cite{Herzog2010}.

An advantage of such an all-at-once approach is that it eliminates the need to
solve the PDEs explicitly; the constraints are only (approximately) satisfied upon convergence. However, such an approach is not feasible for the applications we have in mind because it involves simultaneously updating (and hence storing) all the variables.

\subsection{Reduced approach}
In inverse problems one usually considers a \emph{reduced} formulation that is obtained by eliminating the constraints from (\ref{eq:constr}). This results in an unconstrained optimization problem:
\bq
\min_{\mbf{m}} \left\{\phi(\mbf{m}) = \ts{\frac{1}{2}}||P\mbf{u}_{\mathrm{red}}(\mbf{m}) - \mbf{d}||_2^2\right\},
\label{eq:redL}
\eq
where $\mbf{u}_{\mathrm{red}}(\mbf{m}) = A(\mathbf{m})^{-1}\mathbf{q}$. The resulting optimization problem has a much smaller dimension and can be solved using black-box non-linear optimization methods. In contrast to the \emph{all-at-once} method, the constraints are satisfied at each iteration.

The gradient and the Hessian of $\phi$ are given by
\bq
\nabla\phi(\mbf{m}) &=& G(\mbf{m},\mbf{u}_{\mathrm{red}})^T\mbf{v}_{\mathrm{red}},\\
\nabla^2\phi(\mbf{m}) &=& G(\mbf{m},\mbf{u}_{\mathrm{red}})^TA(\mbf{m})^{-T}P^T\!PA(\mbf{m})^{-1}G(\mbf{m},\mbf{u}_{\mathrm{red}})\nonumber\\
&& - K(\mbf{m},\mbf{v}_{\mathrm{red}})^TA(\mbf{m})^{-1}G(\mbf{m},\mbf{u}_{\mathrm{red}})
\nonumber\\ 
&&  - G(\mbf{m},\mbf{u}_{\mathrm{red}})^TA(\mbf{m})^{-T}K(\mbf{m},\mbf{v}_{\mathrm{red}})
\nonumber\\
&& + R(\mbf{m},\mbf{u}_{\mathrm{red}},\mbf{v}_{\mathrm{red}}),
\eq
where $\mbf{v}_{\mathrm{red}} = A(\mbf{m})^{-T}P\left(\mbf{d} - P\mbf{u}_{\mathrm{red}}\right)$.

A basic (Gauss-Newton) algorithm for minimizing $\phi(\mathbf{m})$ is given in 
Algorithm \ref{alg:reduced}. Note that this corresponds to a block-elimination of the KKT system and the iterates automatically satisfy $\mcl{L}_{\mathbf{u}}(\mbf{m}^k,\mbf{u}^k_{\mathrm{red}},\mbf{v}^k_{\mathrm{red}}) = \mcl{L}_{\mathbf{v}}(\mbf{m}^k,\mbf{u}^k_{\mathrm{red}},\mbf{v}^k_{\mathrm{red}}) = 0$. If the algorithm terminates successfully, the final iterates $(\mbf{m}^*,\mbf{u}^*_{\mathrm{red}},\mbf{v}^*_{\mathrm{red}})$ additionally satisfy $\|\mcl{L}_{\mathbf{m}}(\mbf{m}^*,\mbf{u}^*_{\mathrm{red}},\mbf{v}^*_{\mathrm{red}})\|_2 \leq \epsilon$, so that 
$\|\nabla\mcl{L}(\mbf{m}^*,\mbf{u}^*_{\mathrm{red}},\mbf{v}^*_{\mathrm{red}})\|_2\leq \epsilon$.

\begin{algorithm}
\caption{Basic Gauss-Newton algorithm for find a stationary point of the Lagrangian via the reduced method}
\label{alg:reduced}
\begin{algorithmic}
\REQUIRE{initial guess $\mbf{m}^0$, tolerance $\epsilon$}
\STATE{$k=0$}
\STATE{$\mathbf{u}^{0}_{\mathrm{red}}  = A(\mathbf{m}^{0})^{-1}\mathbf{q}$}
\vspace{1mm}
\STATE{$\mathbf{v}^{0}_{\mathrm{red}}  = A(\mathbf{m}^{0})^{-T}P(\mathbf{d} - P\mathbf{u}^{0}_{\mathrm{red}})$}
\vspace{1mm}
\WHILE{$\|\mcl{L}_{\mathbf{m}}(\mbf{m}^k,\mbf{u}^k_{\mathrm{red}},\mbf{v}^k_{\mathrm{red}})\|_2 \geq \epsilon$}
\vspace{1mm}
\STATE{$\mathbf{g}^k_{\mathrm{red}}    = G(\mathbf{m}^k,\mathbf{u}^k_{\mathrm{red}})^T\mathbf{v}^k_{\mathrm{red}}$}
\vspace{2mm}
\STATE{$H^k_{\mathrm{red}} = G(\mbf{m}^k,\mbf{u}^k_{\mathrm{red}})^TA(\mbf{m}^k)^{-T}P^T\!PA(\mbf{m}^k)^{-1}G(\mbf{m}^k,\mbf{u}^k_{\mathrm{red}})$}
\vspace{1mm}
\STATE{determine steplength $\alpha^k \in (0,1]$}
\vspace{1mm}
\STATE{$\mathbf{m}^{k+1} =\mathbf{m}^k - \alpha^k \left(H_{\mathrm{red}}^{k}\right)^{-1}\mathbf{g}^k_{\mathrm{red}}$}
\vspace{1mm}
\STATE{$\mathbf{u}^{k+1}_{\mathrm{red}}  = A(\mathbf{m}^{k+1})^{-1}\mathbf{q}$}
\vspace{2mm}
\STATE{$\mathbf{v}^{k+1}_{\mathrm{red}}  = A(\mathbf{m}^{k+1})^{-T}P(\mathbf{d} - P\mathbf{u}^{k+1}_{\mathrm{red}})$}
\vspace{1mm}
\STATE{$k = k + 1$}
\ENDWHILE
\end{algorithmic}
\end{algorithm}

A disadvantage of this approach is that it requires the solution of the PDEs at each update, making it computationally expensive.  It also strictly enforces the constraint at each iteration, possibly leading to a very non-linear problem in $\mbf{m}$. 

\section{Penalty and augmented Lagrangian methods}
\label{penalty}
It is impossible to do justice to the wealth of research that has been done on penalty and augmented Lagrangian methods in one section. Instead, we give a brief overview high-lighting the main characteristics of a few basic approaches and their limitations when applied to inverse problems. 

A constrained optimization problem of the form (\ref{eq:constr}) can be recast as an unconstrained problem by introducing a non-negative penalty function $\pi: \mathbb{C}^{N}\rightarrow\mathbb{R}_{\geq 0}$ and a penalty parameter $\lambda>0$ as follows
\bq
\label{eq:penalty}
\min_{\mbf{m},\mbf{u}} \ts{\frac{1}{2}}||P\mathbf{u} - \mathbf{d}||_2^2 + \lambda\pi(\mbf{A}(\mbf{m})\mbf{u} - \mbf{q}).
\eq
The idea is that any departure from the constraint is penalized so that the solution of this unconstrained problem will coincide with that of the constrained problem when $\lambda$ is large enough.

A quadratic penalty function $\pi(\cdot) = \ts{\frac{1}{2}}||\cdot||_2^2$ leads to a differentiable unconstrained optimization problem (\ref{eq:penalty}) whose minimizer coincides with the solution of the constrained optimization problem (\ref{eq:constr}) when $\lambda \rightarrow \infty$ \cite[Thm. 17.1]{Nocedal}. Practical algorithms rely on repeatedly solving the unconstrained problem for increasing values of $\lambda$.

A common concern with this approach is that the Hessian may become increasingly ill-conditioned as $\lambda\rightarrow\infty$ when there are fewer constraints than variables. For PDE-constrained optimization in inverse problems, there are enough constraints ($A(\mbf{m})$ is invertible) to prevent this. We discuss this limiting case in more detail in section \ref{algorithm}.

For certain non-smooth penalty functions, such as $\pi(\cdot) = ||\cdot||_1$, the minimizer of (\ref{eq:penalty}) is a solution of the constrained problem for \emph{any} $\lambda \geq \lambda^*$ for some $\lambda^*$ \cite[Thm. 17.3]{Nocedal}. In practice, a continuation strategy is used to find a suitable value for $\lambda$. An advantage of this approach is that $\lambda$ does not become arbitrarily large, thus avoiding the ill-conditioning problems mentioned above. A disadvantage is that the resulting unconstrained problem is no longer differentiable. With large-scale applications in mind, we therefore do not consider exact penalty methods any further in this paper.

Another approach that avoids having to increase $\lambda$ to infinity is the \emph{augmented Lagrangian} approach (cf. \cite{Nocedal}). In this approach, a quadratic penalty $\lambda\|\mbf{A}(\mbf{m})\mbf{u} - \mbf{q}\|_2^2$ is added to the Lagrangian (\ref{eq:Lagrangian}). A standard approach to solve the constrained problem based on the augmented Lagrangian is the \emph{Alternating direction method of multipliers} (ADMM). In its most basic form it relies on minimizing the augmented Lagrangian w.r.t. $(\mbf{m}, \mbf{u})$ and subsequently updating the multiplier $\mbf{v}$ and the penalty parameter $\lambda$ \cite{Eckstein2012a,Curtis2014}. This would require us to store the multipliers, which is not feasible for the problems we have in mind.

In the next two sections, we discuss a computationally efficient algorithm for solving the constrained optimization problem (\ref{eq:constr}) based on a quadratic penalty formulation. This formulation is attractive because it leads to a differentiable, unconstrained, optimization problem. Moreover, the optimization in $\mbf{u}$ has a closed-form solution which can be computed efficiently, making it an ideal candidate for the type of problems we have in mind.

\section{A reduced penalty method}
\label{varpro}

Using a quadratic penalty function, the constrained problem (\ref{eq:constr}) is reformulated as
\bq
\label{eq:rpenalty}
\min_{\mbf{m},\mbf{u}} \left\{\mcl{P}(\mbf{m},\mbf{u}) = \ts{\frac{1}{2}}||P\mathbf{u} - \mathbf{d}||_2^2 + \ts{\frac{1}{2}}\lambda||A(\mbf{m})\mbf{u} - \mbf{q}||_2^2 \right\}.
\eq
The gradient and Hessian of $\mcl{P}$ are given by
\bq
\nabla\mcl{P} = \left(\begin{array}{c}
\mcl{P}_{\mbf{m}}\\
\mcl{P}_{\mbf{u}}\\
\end{array}
\right)
= 
\left(\begin{array}{c}
\lambda G(\mbf{m},\mbf{u})^T\!\left(A(\mbf{m})\mbf{u} - \mbf{q}\right)\\
P^T\!(P\mbf{u} - \mbf{d}) + \lambda A(\mbf{m})^T\!(A(\mbf{m})\mbf{u} - \mbf{q})\\
\end{array}
\right),
\eq
and
\bq
\nabla^2\mcl{P} &=&
\left(
\begin{array}{cc}
\mcl{P}_{\mbf{m},\mbf{m}}&\mcl{P}_{\mbf{m},\mbf{u}}\\
\mcl{P}_{\mbf{u},\mbf{m}}&\mcl{P}_{\mbf{u},\mbf{u}}\\
\end{array}
\right),
\eq
where
\bq
\mcl{P}_{\mbf{m},\mbf{m}} &=& \lambda (G(\mbf{m},\mbf{u})^TG(\mbf{m},\mbf{u}) + R(\mbf{m},\mbf{u},A(\mbf{m})\mbf{u}-\mbf{q})),\\
\mcl{P}_{\mbf{u},\mbf{u}} &=&P^T\!P + \lambda A(\mbf{m})^T\!A(\mbf{m}),\\
\label{eq:Hl}
\mcl{P}_{\mbf{m},\mbf{u}} &=&\lambda (K(\mbf{m},A(\mbf{m})\mbf{u}-\mbf{q}) + A(\mbf{m})^TG(\mbf{m},\mbf{u})).
\eq
Of course, optimization in the full $(\mbf{m},\mbf{u})$-space is not feasible for large-scale problems, so we  eliminate $\mbf{u}$ by introducing 
\bq
\label{eq:ul}
\mbf{u}_{\lambda}(\mbf{m}) = \argmin_{\mbf{u}} \mcl{P}(\mbf{m},\mbf{u}),
\eq
and define a reduced objective:
\bq
\label{eq:redpenalty}
\phi_{\lambda}(\mbf{m}) = \mcl{P}(\mbf{m},\mbf{u}_{\lambda}(\mbf{m})).
\eq
The optimization problem for the state (\ref{eq:ul}) has a closed-form solution:
\[
\mbf{u}_{\lambda} = \left(A(\mathbf{m})^T\!A(\mathbf{m}) + \lambda^{-1}P^T\!P\right)^{-1}\left(A(\mathbf{m})^T\mbf{q} + \lambda^{-1}P\mbf{d}\right).
\]
The modified system $A^T\!A + \lambda^{-1}P^T\!P$ is a low-rank modification of the original PDE and incorporates the measurements in the PDE solve. This is the main difference with the conventional reduced approach (cf. Algorithm \ref{alg:reduced}); the estimate of the state is not only based on the physics and the current model, but also on the data.

Following \cite[Thm. 1]{Aravkin2012c}, it is readily verified that the gradient and Hessian of $\phi_{\lambda}$ are given by 
\bq
\label{eq:gradpen}
\nabla\phi_{\lambda}(\mbf{m}) &=& \mcl{P}_{\mbf{m}}(\mbf{m},{\mbf{u}}_{\lambda}),\\
\label{eq:hesspen}
\nabla^2\phi_{\lambda}(\mbf{m}) &=& \mcl{P}_{\mbf{m},\mbf{m}}(\mbf{m},{\mbf{u}}_{\lambda}) \nonumber\\
&&-\mcl{P}_{\mbf{m},\mbf{u}}(\mbf{m},{\mbf{u}}_{\lambda})\left(\mcl{P}_{\mbf{u},\mbf{u}}(\mbf{m},{\mbf{u}}_{\lambda})\right)^{-1}\mcl{P}_{\mbf{u},\mbf{m}}(\mbf{m},{\mbf{u}}_{\lambda}).
\eq
Note that $\nabla^2\phi_{\lambda}$ is the Schur complement of $\nabla^2\mcl{P}$.

A basic Gauss-Newton algorithm for minimizing $\phi_{\lambda}$ is shown in Algorithm \ref{alg:penalty}. Note that the computation of the adjoint-state $\mathbf{v}_{\lambda}$ does \emph{not} require an additional PDE-solve in this algorithm. Instead, the forward and adjoint solve are done simultaneously via the normal equations.

\begin{algorithm}
\caption{Basic Gauss-Newton algorithm for find a stationary point of the Lagrangian via the penalty method}
\label{alg:penalty}
\begin{algorithmic}
\REQUIRE{initial guess $\mbf{m}^0$, penalty parameter $\lambda$, tolerance $\epsilon$}
\STATE{$k=0$}
\STATE{$\mathbf{u}^{0}_{\lambda}  = \left(A(\mathbf{m}^{0})^TA(\mathbf{m}^{0}) + \lambda^{-1}P^T\!P\right)^{-1}\left(A(\mathbf{m}^{0})^T\mbf{q} + \lambda^{-1}P\mbf{d}\right)$}
\STATE{$\mathbf{v}^{0}_{\lambda}  = \lambda(A(\mathbf{m}^{0})\mathbf{u}^{0}_{\lambda} - \mathbf{q})$}

\WHILE{$\|\mcl{L}_{\mathbf{m}}(\mbf{m}^k,\mbf{u}^k_{\lambda},\mbf{v}^k_{\lambda})\|_2 \geq \epsilon$}
\vspace{1mm}
\STATE{$\mathbf{g}^k_{\lambda}  = G(\mathbf{m}^k,\mathbf{u}^k_{\lambda})^T\mathbf{v}^k_{\lambda}$}
\vspace{2mm}
\STATE{$H^k_{\lambda}           = \lambda G^T\left(I - A\left(A^T\!A + \lambda^{-1}P^T\!P\right)^{-1}A^T \right)G$}
\vspace{1mm}
\STATE{determine steplength $\alpha^k \in (0,1]$}
\vspace{1mm}
\STATE{$\mathbf{m}^{k+1}        = \mathbf{m}^k - \alpha^k \left(H_{\lambda}^{k}\right)^{-1}\mathbf{g}^k_{\lambda}$}
\vspace{1mm}
\STATE{$\mathbf{u}^{k+1}_{\lambda}  = \left(A(\mathbf{m}^{k+1})^TA(\mathbf{m}^{k+1}) + \lambda^{-1}P^T\!P\right)^{-1}\left(A(\mathbf{m}^{k+1})^T\mbf{q} + \lambda^{-1}P\mbf{d}\right)$}
\vspace{2mm}
\STATE{$\mathbf{v}^{k+1}_{\lambda}  = \lambda(A(\mathbf{m}^{k+1})\mathbf{u}^{k+1}_{\lambda} - \mathbf{q})$}
\vspace{1mm}
\STATE{$k = k + 1$}
\ENDWHILE
\end{algorithmic}
\end{algorithm}

Next, we show how the states, $\mathbf{u}^k_{\lambda}$ and $\mathbf{v}^k_{\lambda}$, generated by this algorithm  relate to the states generated by the reduced approach and subsequently that if the algorithm successfully terminates the iterates $(\mbf{m}^*, \mathbf{u}^*_{\lambda}, \mathbf{v}^*_{\lambda})$ satisfy 
\[
\|\nabla\mcl{L}(\mbf{m}^*, \mathbf{u}^*_{\lambda}, \mathbf{v}^*_{\lambda})\|_2 \leq \epsilon + \mcl{O}(\lambda^{-1}). 
\]

\begin{lemma}
\label{lemma}
For a fixed $\mbf{m}$, the states $\mathbf{u}_{\lambda}$ and $\mathbf{v}_{\lambda}$ used 
in the reduced penalty approach (Algorithm \ref{alg:penalty}) are related to the states $\mathbf{u}_{\mathrm{red}}$ and 
$\mathbf{v}_{\mathrm{red}}$ used in the reduced approach  (Algorithm \ref{alg:reduced})
as follows
\bq
\mathbf{u}_{\lambda} = \mathbf{u}_{\mathrm{red}} + \mcl{O}(\lambda^{-1}),\\
\mathbf{v}_{\lambda} = \mathbf{v}_{\mathrm{red}} + \mcl{O}(\lambda^{-1}).
\eq
\end{lemma}
\begin{proof}
The state variables used in the penalty approach are given by
\[
\mbf{u}_{\lambda} = \left(A^T\!A + \lambda^{-1}P^T\!P\right)^{-1}\left(A^T\mbf{q} + \lambda^{-1}P^T\!\mbf{d}\right),
\]
and
\[
\mbf{v}_{\lambda} = \lambda(A\mbf{u}_{\lambda} - \mathbf{q}).
\]
The former can be re-written as
\[
\mbf{u}_{\lambda} = A^{-1}\left(I + \lambda^{-1}A^{-T}P^T\!PA^{-1}\right)^{-1}\left(\mbf{q} + \lambda^{-1}A^{-T}P^T\mbf{d}\right).
\]
For $\lambda>\|PA^{-1}\|_2^2$ we may expand the inverse as $(I + \lambda^{-1}B)^{-1} \approx I - \lambda^{-1}B + \lambda^{-2}B^2 + \ldots$,
and find that
\bq
\mbf{u}_{\lambda} &=& A^{-1}\mbf{q}\nonumber\\
&&+ \lambda^{-1}A^{-1}\!A^{-T}P^T\!\left(\mbf{d} - PA^{-1}\mbf{q}\right)\nonumber\\
&&- \lambda^{-2}A^{-1}\left(PA^{-1}\right)^{T}\!\!\left(PA^{-1}\right)A^{-T}P^T\!\left(\mbf{d} - PA^{-1}\mbf{q}\right) + \mcl{O}(\lambda^{-3})\nonumber\\
&=& \mbf{u}_{\mathrm{red}} + \lambda^{-1}A^{-1}\left(I - \lambda^{-1}\left(PA^{-1}\right)^{T}\!\!\left(PA^{-1}\right)\right)\mbf{v}_{\mathrm{red}}  + \mcl{O}(\lambda^{-3}).
\eq
We immediately find
\bq
\mbf{v}_{\lambda} = \mbf{v}_{\mathrm{red}} - \lambda^{-1}\left(PA^{-1}\right)^{T}\!\!\left(PA^{-1}\right)\mbf{v}_{\mathrm{red}} + \mcl{O}(\lambda^{-2}).
\eq
\qed
\end{proof}

\begin{remark}
\label{remark}
Lemma \ref{lemma} suggests a natural scaling for the penalty parameter; $\lambda > \|PA^{-1}\|_2^2$ can be considered large, while $\lambda < \|PA^{-1}\|_2^2$ can be considered small. 
\end{remark}

\begin{theorem}
\label{theorem1}
At each iteration of algorithm \ref{alg:penalty}, the iterates satisfy $\|\mcl{L}_{\mbf{u}}(\mbf{m}^k,\mbf{u}^k_{\lambda},\mbf{v}^k_{\lambda})\|_2 = 0$
and
$\|\mcl{L}_{\mbf{v}}(\mbf{m}^k,\mbf{u}^k_{\lambda},\mbf{v}^k_{\lambda})\|_2 = \mcl{O}(\lambda^{-1})$.
Moreover, if algorithm \ref{alg:penalty} terminates successfully
at $\mbf{m}^*$ for which $\|\mcl{L}_{\mbf{m}}(\mbf{m}^*,\mbf{u}^*_{\lambda},\mbf{v}^*_{\lambda})\|_2 \leq \epsilon$,
we have $\|\nabla\mcl{L}(\mbf{m}^*,\mbf{u}^*_{\lambda},\mbf{v}^*_{\lambda})\|_2\leq \epsilon + \mcl{O}(\lambda^{-1}).$
\end{theorem}
\begin{proof}
Using the definitions of $\mbf{u}_{\lambda}$ and $\mbf{v}_{\lambda}$ we find
for any $\mbf{m}$
\bq
\mcl{L}_{\mathbf{u}}(\mbf{m},\mbf{u}_{\lambda},\mbf{v}_{\lambda}) &=& A(\mbf{m})^T\mbf{v}_{\lambda} + P(P\mbf{u}_{\lambda} - \mbf{d})\nonumber\\
&=& \lambda A^T(A\mbf{u}_{\lambda} - \mbf{q}) + P(P\mbf{u}_{\lambda} - \mbf{d}) = 0.
\eq
Using the approximations for $\mbf{u}_{\lambda}$ and $\mbf{v}_{\lambda}$ for $\lambda>\|PA^{-1}\|_2^2$ presented in Lemma \ref{lemma}, we find
\bq
\mcl{L}_{\mathbf{v}}(\mbf{m},\mbf{u}_{\lambda},\mbf{v}_{\lambda}) &=& A(\mbf{m})\mbf{u}_{\lambda} - \mbf{q}\nonumber\\
&=& \lambda^{-1}\mbf{v}_{\mathrm{red}} + \mcl{O}(\lambda^{-2}).
\eq
Thus we find
\bq
\|\mcl{L}_{\mathbf{v}}(\mbf{m}^*,\mbf{u}^*_{\lambda},\mbf{v}^*_{\lambda})\|_2 \leq \mcl{O}(\lambda^{-1}).
\eq
At a point $\mbf{m}^*$ for which $\|\mcl{L}_{\mathbf{m}}(\mbf{m}^*,\mbf{u}^*_{\lambda},\mbf{v}^*_{\lambda})\|_2 \leq \epsilon$
we immediately find that
\bq
\|\nabla\mcl{L}(\mbf{m}^*,\mbf{u}^*_{\lambda},\mbf{v}^*_{\lambda})\|_2^2 = \nonumber\\
\|\mcl{L}_{\mathbf{m}}(\mbf{m}^*,\mbf{u}^*_{\lambda},\mbf{v}^*_{\lambda})\|_2^2 +
\|\mcl{L}_{\mathbf{u}}(\mbf{m}^*,\mbf{u}^*_{\lambda},\mbf{v}^*_{\lambda})\|_2^2 +
\|\mcl{L}_{\mathbf{v}}(\mbf{m}^*,\mbf{u}^*_{\lambda},\mbf{v}^*_{\lambda})\|_2^2 \nonumber\\
\leq \epsilon^2 + \mcl{O}(\lambda^{-2}),
\eq
and hence that 
\bq
\|\nabla\mcl{L}(\mbf{m}^*,\mbf{u}^*_{\lambda},\mbf{v}^*_{\lambda})\|_2 \leq \epsilon + \mcl{O}(\lambda^{-1}).
\eq
\qed
\end{proof}

This means we can use algorithm \ref{alg:penalty} to find a stationary point of the Lagrangian within finite accuracy (in terms of $\|\nabla\mcl{L}\|_2$) with a finite $\lambda$. In order to reach a given tolerance, we need $\lambda^{-1}\|\mcl{L}_{\mbf{v}}\|_2$ to be small compared to $\|\mcl{L}_{\mbf{m}}\|_2$. This condition can easily be verified for a given $\mathbf{m}$ and thus serve as the basis for a selection criterion for $\lambda$.

In an ill-posed inverse problem, we can only hope to reconstruct a few components of $\mathbf{m}^*$, roughly corresponding the dominant eigenvectors of the reduced Hessian. Thus, the tolerance $\epsilon$ can be quite large, even for an acceptable reconstruction. Driving down the norm of the gradient any further would only refine the reconstruction in the eigenmodes corresponding to increasingly small eigenvalues and would hardly affect the final reconstruction and datafit. 

Next, we derive an expression for the distance of the final iterate of algorithm \ref{alg:penalty}, $\mbf{m}^*_{\lambda}$, to a stationary point of the Lagrangian, $\mbf{m}^*$ in terms of the chosen tolerance, $\epsilon$, and the data-error $\eta = \|\mathbf{d} - PA(\mbf{m}^*_{\lambda})^{-1}\mbf{q}\|_2$. 

\begin{theorem}
\label{theorem2}
Given a stationary point of the Lagrangian $(\mbf{m}^*, \mbf{u}^*, \mbf{v}^*)$ and the final iterate of algorithm \ref{alg:penalty} $\mbf{m}^*_{\lambda}$ such that $\|\mcl{L}_{\mathbf{m}}(\mbf{m}^*_{\lambda},\mbf{u}^*_{\lambda},\mbf{v}^*_{\lambda})\|_2 \leq \epsilon$, we have $\|\mbf{m}^*_{\lambda} - \mbf{m}^*\|_2 \leq \kappa(H)\left( \widetilde{\epsilon} + \widetilde{\lambda}^{-1}\widetilde{\eta}\right)$ where $H = J^T\!J$ is the reduced Gauss-Newton Hessian with condition number $\kappa(H)$, $\widetilde{\epsilon} = \epsilon/\|H\|_2$ is the scaled tolerance, $\widetilde{\eta} = \|\mathbf{d} - PA(\mbf{m}^*_{\lambda})^{-1}\mbf{q}\|_2/\|J\|_2$ is the scaled data-error and $\widetilde{\lambda} = \lambda / \|PA^{-1}\|_2^2$ is the scaled penalty parameter.
\end{theorem}
\begin{proof}
We expand the gradient of the Lagrangian at the stationary point as
\bq
\left(
\begin{array}{c}
\mcl{L}_{\mbf{m}}(\mbf{m}^*_{\lambda}, \mbf{u}^*_{\lambda}, \mbf{v}^*_{\lambda})\\
\mcl{L}_{\mbf{u}}(\mbf{m}^*_{\lambda}, \mbf{u}^*_{\lambda}, \mbf{v}^*_{\lambda})\\
\mcl{L}_{\mbf{v}}(\mbf{m}^*_{\lambda}, \mbf{u}^*_{\lambda}, \mbf{v}^*_{\lambda})\\
\end{array}
\right)
\approx\nonumber\\
\left(
\begin{array}{ccc}
R(\mbf{m}^*,\mbf{u}^*,\mbf{v}^*)&K(\mbf{m}^*,\mbf{v}^*)^T&G(\mbf{m}^*,\mbf{u}^*)^T\\
K(\mbf{m}^*,\mbf{v}^*)&P^T\!P&A(\mbf{m}^*)^T\\
G(\mbf{m}^*,\mbf{u}^*)&A(\mbf{m}^*)&0\\
\end{array}
\right)
\left(
\begin{array}{c}
\mbf{m}^*_{\lambda} - \mbf{m}^*\\
\mbf{u}^*_{\lambda} - \mbf{u}^*\\
\mbf{v}^*_{\lambda} - \mbf{v}^*\\
\end{array}
\right)
\eq
Using the expressions for the gradient of the Lagrangian at $(\mbf{m}^*_{\lambda},\mbf{u}^*_{\lambda},\mbf{v}^*_{\lambda})$ obtained in Theorem \ref{theorem1} we can obtain an expression for $\mbf{m}^*_{\lambda} - \mbf{m}^*$ etc. by solving
\bq
\left(
\begin{array}{ccc}
R&K^T&G^T\\
K&P^T\!P&A^T\\
G&A&0\\
\end{array}
\right)
\left(
\begin{array}{c}
\mbf{m}^*_{\lambda} - \mbf{m}^*\\
\mbf{u}^*_{\lambda} - \mbf{u}^*\\
\mbf{v}^*_{\lambda} - \mbf{v}^*\\
\end{array}
\right)
=
\left(
\begin{array}{c}
\mcl{L}_{\mbf{m}}(\mbf{m}^*_{\lambda}, \mbf{u}^*_{\lambda}, \mbf{v}^*_{\lambda})\\
0\\
\lambda^{-1}\mathbf{v}^* + \mcl{O}(\lambda^{-2})\\
\end{array}
\right)
\eq
Eliminating the bottom two rows we find
\bq
\label{eq:dm}
H(\mbf{m}^*_{\lambda} - \mbf{m}^*) = \mcl{L}_{\mbf{m}}(\mbf{m}^*_{\lambda}, \mbf{u}^*_{\lambda}, \mbf{v}^*_{\lambda}) + \lambda^{-1}F\mbf{v}^* + \mcl{O}(\lambda^{-2}),
\eq
where 
\[
H = R - K^TA^{-1}G-G^TA^{-T}K + G^TA^{-T}P^TPA^{-1}G
\]
is the reduced Hessian and
\[
F = G^T\!A^{-T}\!P^T\!PA^{-1} - K^TA^{-1}.
\]
Expressing $\mbf{v}^*$ as
\[
\mbf{v} = A^{-T}\!P^T\!\left(\mathbf{d} - PA(\mbf{m}^*_{\lambda})^{-1}\mbf{q}\right),
\]
and ignoring the second order terms in $H$ and $F$, i.e, 
\[
H \approx G^T\!A^{-T}\!P^T\!PA^{-1}G,
\]
and
\[
F \approx G^T\!A^{-T}\!P^T\!PA^{-1}
\]
we find
\[
\|\mbf{m}^*_{\lambda} - \mbf{m}^*\|_2 \leq ||H^{-1}||_2\left(\epsilon +  \widetilde{\lambda}^{-1} \|G^T\!A^{-T}\!P^T\|_2\eta\right),
\]
where $\eta = \|\mathbf{d} - PA(\mbf{m}^*_{\lambda})^{-1}\mbf{q}\|_2$ denotes the data-error.

Introducing the scaled tolerance and data-error we express this as
\[
\|\mbf{m}^*_{\lambda} - \mbf{m}^*\|_2 \leq \kappa(H)\left(\widetilde{\epsilon} +  \widetilde{\lambda}^{-1} \widetilde{\eta}\right).
\]
\end{proof}
As noted before, we can only hope to reconstruct a few components of $\mathbf{m}^*$, given by the dominant eigenvectors
of $H$. Thus, even for an acceptable reconstruction, the error $\|\mbf{m}^*_{\lambda} - \mbf{m}^*\|_2$ can be relatively large. On the other hand, the data-error for the corresponding $\mathbf{m}_{\lambda}^*$ can be quite small, even if $\|\mbf{m}^*_{\lambda} - \mbf{m}^*\|_2$ is large. 
We expect the penalty method
to yield an acceptable reconstruction when $\widetilde{\lambda}^{-1} \widetilde{\eta}$ is small compared to $\widetilde{\epsilon}$.

\section{Algorithm}
\label{algorithm}
In this section, we discuss some practicalities of the implementation of algorithm \ref{alg:penalty}. We slightly elaborate the notation to explicitly 
reveal the multi-experiment structure of the problem. In this case, the data are acquired in a series of $K$ independent experiments and $\mbf{d} = [\mbf{d}_1; \ldots; \mbf{d}_{K}]$ is a block-vector. We partition the states and sources in a similar manner. Since the experiments are independent, the system matrix $A$ is block-diagonal matrix with $K$ blocks $A_i(\mbf{m})$ of size $N\times N$. Similarly, the matrix $P$ consists of blocks $P_i$. Recall that we collect $L$ independent measurements for each experiment, so the matrices $P_i\in\mathbb{R}^{N\times L}$ have full rank.

\subsection{Solving the augmented PDE}
\label{solving}
Due to the block structure of the problem, the linear systems can be solved independently.
We can obtain the state $\mathbf{u}_i$ by solving the following inconsistent overdetermined system
\bq
\label{eq:u_pen}
\left(
\begin{array}{c}
A_i(\mbf{m})\\
\lambda^{-1/2}P_i
\end{array}
\right)
\mbf{u}_{i} \approx
\left(
\begin{array}{c}
\mbf{q}_{i}\\
\lambda^{-1/2}\mbf{d}_{i}
\end{array}
\right),
\eq
in a least-squares sense. Assuming that all the blocks $A_i$ and $P_i$ are identical, we will drop the subscript $i$ for the remainder of this subsection. Next, we will discuss various approaches to solving the overdetermined system (\ref{eq:u_pen}).

\textbf{Factorization:} If both $A$ and $P$ are sparse, we can efficiently solve the system via a QR factorization or via a Cholesky factorization of the corresponding Normal equations. In many applications, $P^T\!P$ is a (nearly) diagonal matrix and thus the augmented system $A^T\!A + \lambda^{-1}P^T\!P$ has a similar sparsity pattern as the original system. Thus, the fill-in will not be worse than when factorizing the original system.

\textbf{Iterative methods:} While we can make use of factorization techniques for small-scale applications, industry-scale applications will typically require (preconditioned) iterative methods.  Obviously, we can apply any preconditioned iterative method that is suitable for solving least-squares problems, such as LSQR, LSMR or CGLS \cite{Paige1982,Fong2011,Bru2014}. Another promising candidate is a generic accelerated row-projected method described by \cite{Bjorck1979,Gordon2013} which proved useful for solving PDEs and can be easily extended to deal with overdetermined systems \cite{Censor1983}. 

To get an idea of how such iterative methods will perform we explore some of the properties of the augmented system. The augmented system $A^T\!A + \lambda^{-1} P^T\!P$ is a rank $L$ modification of the original system $A^T\!A$. It follows from \cite[Thm 8.1.8]{Golub1996} that the eigenvalues are related as
\bq
\mu_n(A^T\!A + \lambda^{-1} P^T\!P) = \mu_n(A^T\!A) + a_n \lambda^{-1}, n = 1, 2, \ldots, N
\eq
where $\mu_1(B) > \mu_2(B) > \ldots > \mu_{N}(B)$ denote the eigenvalues of $B$ and the coefficients $a_n$ satisfy $\sum_{n=1}^{N} a_n = L$. This means that at worst, 1 eigenvalue is shifted by $L\lambda^{-1}$ while at best, all the eigenvalues are shifted by $LN^{-1}\lambda^{-1}$. For the condition numbers $\kappa(B) = \mu_1(B)/\mu_N(b)$ we find
\bq
C_N^{-1}\kappa(A^T\!A)\leq \kappa(A^T\!A + \lambda^{-1} P^T\!P) \leq C_1\kappa(A^T\!A),
\eq
where $C_i = \left(1 + \frac{L}{\lambda \mu_i(A^T\!A)}\right)$.

To illustrate this, we show a few examples for a 1D (time-harmonic) parabolic PDE $\left(\imath\omega - \partial_x^2\right)u = 0$ and the 1D Helmholtz equation $\left(\omega^2 + \partial_x^2\right)u = 0$, both with Neumann boundary conditions. Both are discretized using first-order finite-differences on $x \in [0,1]$ with $N=51$ points for $\omega = 10\pi$. The sampling matrix $P$ consists of $L$ rows of the identity matrix (regularly sampled). The ratio of the condition numbers of $A^T\!A$ and $A^T\!A + \lambda^{-1}P^T\!P$ for the parabolic and Helmholtz equation are shown in tables \ref{table:example2} and \ref{table:example3}. For these examples, the condition number of the augmented system is actually lower than that of the original system. The eigenvalues 
are shown in figures \ref{fig:example2} and \ref{fig:example3}. These show that the actual eigenvalues distributions do not change significantly. We expect
that iterative methods will perform similarly on the augmented system as they would on the original system. \emph{How} to effectively precondition the augmented system given a good preconditioner for the original system is a different matter which is outside the scope of this paper.

\textbf{Direct methods:} When the matrix has additional structure we might actually prefer a direct method over an iterative method. An example is explicit time-stepping, where the system matrix $A$ exhibits a lower-triangular block-structure. In this case the action of $A^{-1}$ can be computed efficiently via forward substitution, requiring storage of only a few time-slices of the state. The adjoint system $A^{-T}$ can be solved by backward substitution, however, the full time-history of the state variable is needed to compute the gradient \cite{Rothauge2015}.
For the penalty method, the augmented system $A^T\!A + \lambda^{-1}P^T\!P$ will have a banded structure and the system can be solved using a block version of the Thomas algorithm, which again would require storage of the full time-history. So even in this setting it seems possible to apply the penalty method at roughly the same per-iteration complexity as the reduced method.

\subsection{Gradient and Hessian computation}
Given these solutions $\mbf{u}_k$ of (\ref{eq:u_pen}), the gradient, $\mbf{g}_{\lambda}$ and Gauss-Newton Hessian $H_{\lambda}$ of $\phi_{\lambda}$ are given by (cf eqs. \ref{eq:gradpen}-\ref{eq:hesspen})
\bq
\mbf{g}_{\lambda} &=& \lambda\sum_{k=1}^K G_k^T\!\left(A_k\mathbf{u}_{k} - \mathbf{q}_{k}\right),\\
H_{\lambda} &=& \lambda\sum_{k=1}^K G_k^T\!\left(I - A_k\left(A_k^T\!A_k + \lambda^{-1}P_k^T\!P_k\right)^{-1}A_k^T \right)G_k,
\eq
where $G_k = G(\mbf{m},\mbf{u}_k)$. We can compute the inverse of $\left(A_k^TA_k + \lambda^{-1}P_kP_k^T\right)$ in the same way as used when solving for the states. In practice, we would solve for one state at a time and aggregate the results on the fly. Moreover, the Gauss-Newton Hessian admits a natural sparse approximation $H_{\lambda} \approx \lambda\sum_{k=1}^K G_k^TG_k$ which has proved to work well in practice \cite{esser_automatic_2015}.

\subsection{Choosing $\lambda$}
An important aspect of the proposed method is the choice of the penalty parameter, $\lambda$. Theorem \ref{theorem1} essentially states that we can expect to find a stationary point of the Lagrangian within finite accuracy with finite $\lambda$ as long as $\lambda^{-1}\|\mcl{L}_{\mbf{v}}\|_2$ is small compared to $\|\mcl{L}_{\mbf{m}}\|_2$.
We can easily keep track of this quantity -- at no significant additional computational cost -- during the iterations and increase $\lambda$ as needed. Initialization can be done by directly enforcing $\lambda^{-1}\|\mcl{L}_{\mbf{v}}\|_2$ to be some fraction of $\|\mcl{L}_{\mbf{m}}\|_2$ at the initial iterate. A natural scaling for $\lambda$ is suggested by Lemma \ref{lemma}; $\lambda = \widetilde{\lambda}\|PA^{-1}\|_2^2$, where $\widetilde{\lambda} > 1$ is considered large and $\widetilde{\lambda} < 1$ is considered small.
 
\subsection{Complexity estimates}
A summary of the leading order computational costs \emph{per iteration} of the penalty, reduced and all-at-once approaches is given in table \ref{table:costs}. 
The storage and computation required for the reduced and penalty methods are of the same order in terms of $K,N$ and $M$. The PDE-solves in both the penalty and reduced approaches can be done independently and in parallel. This makes these approaches more attractive for large-scale problems from a computational point-of-view.

We have argued in section \ref{solving} that it is plausible that the augmented system can be solved as efficiently as the original PDE. However,
it is not clear how the penalty and reduced methods will compare in the required number of iterations, though we expect that for small $\lambda$ the optimization problem is less non-linear and hence easier to solve. To reach a given tolerance with the penalty method, however, we need a continuation strategy in $\lambda$, adding to the cost of the penalty method. In the next section we compare the actual computational costs on a few test-cases.

\section{Case studies}
\label{examples}
The following experiments are done in Matlab, using direct factorization to solve the PDEs (with Matlab \texttt{slash}). We consider both a Gauss-Newton (GN) and a Quasi-Newton (QN) variant of the algorithms and use a weak Wolfe linesearch to determine the steplength. In the GN method the Hessian is inverted using conjugate gradients (\texttt{pcg}) up to a relative tolerance of $\delta$. The matrix-vector products are computed on the fly. For the QN method we use the L-BFGS inverse Hessian with a history size of 5 \cite{Nocedal}. We measure the cost of the inversion by counting the number of PDE solves as outlined in table \ref{table:costs}. In all experiments, we set $\lambda$ relative to the largest eigenvalue of $A^{-T}P^T\!PA^{-1}$ at the initial iterate. This scaling is justified by lemma \ref{lemma}. We compare the results for various fixed values of $\lambda$ and additionally show the results for an ad-hoc continuation strategy; performing a few iterations for increasing values of $\lambda$ using the result for each $\lambda$ as initial guess for the next.

To avoid the inverse crime, we compute the data for the ground truth model on a finer grid than used for the inversion. 

In these experiments, we illustrate that the penalty method: 
\begin{itemize}
\item converges to a stationary point of the Lagrangian within the predicted tolerance of $\mcl{O}(\lambda^{-1})$;
\item can give practically the same or better results as the reduced method at a lower computational cost;
\item is not overly sensitive to noise;
\item is less sensitive to the initial guess than the conventional approach.
\end{itemize}
The Matlab code used to perform the experiments is available from \url{https://github.com/tleeuwen/Penalty-Method}.

\subsection{1D DC resistivity}
We consider the PDE
\bq
\partial_t u(t,x) = \partial_x\left(m(x)\partial_x\right) u(t,x),
\eq
on the domain $x \in [0,1]$ with Neumann boundary conditions. 
A finite-difference discretization in the temporal Fourier domain gives
\bq
A(\mbf{m}) = \imath\omega\mathsf{diag}(\mbf{w}) + D^T\mathsf{diag}(\mbf{m})D,
\eq
where $\omega$ is the angular frequency, $\mbf{w} = [\frac{1}{2}; 1, \ldots, 1; \frac{1}{2}]$, $\mbf{m}$ represents 
the medium parameter in the cell-centres and $D$ is the $N-1\times N$ finite-difference matrix
\[
D=\frac{1}{h}\left(\begin{array}{cccccc} 
-1& 1&   &  &      &   \\
  &-1& 1 &  &      &   \\
  &  &\ddots&\ddots&   \\
  &  &      & -1   & 1 \\
\end{array}\right),
\]
with $h=1/(N-1)$. The Jacobian is given by 
\[
G(\mbf{m},\mbf{u}) = D^T\mathsf{diag}(D\mbf{u}).
\]
The ground-truth model is $m(x) = 1+e^{-10(x-1/2)^2}$ and we locate two sources and receivers on either end of the domain. The data are generated on a grid with $N=201$ points and we have $K=L=2$.

For the inversion we use $N=101$ points.
We use a GN method with $\epsilon=10^{-9}$, $\delta=10^{-3}$ and include a regularization term $\frac{\alpha}{2} \|D\mbf{m}\|_2^2$ with $\alpha = 10^{-6}$.
The initial parameters are $\mbf{m}^0 = \mbf{1}$.

The results are shown in figure \ref{fig:1D_exp1}. The convergence plot, figure \ref{fig:1D_exp1} (a), shows the predicted behaviour of the penalty method; the norm
of the gradient of the Laplacian stalls at $\mcl{O}(\lambda^{-1})$. The convergence of the continuation strategy shows that it is possible to reach the desired tolerance by gradually increasing $\lambda$ (using $\lambda = \{0.1, 1, 10, 100\}$ with a few iterations each). The resulting parameter estimates are very similar as can be seen 
in figure \ref{fig:1D_exp1} (b). The actual costs of the inversion are listed in table \ref{table:1D_exp1}. The computational cost for the various approaches are of the same order of magnitude, except for $\lambda=10$, where more than twice as many iterations are required. 

\subsection{2D Acoustic tomography}
Consider the 2D scalar wave-equation
\bq
m(x)\partial_t^2u(t,x) = \nabla^2u(t,x),
\eq
on $x \in \Omega \subseteq \mathbb{R}^2$  with radiation boundary conditions $\sqrt{m}(x)\partial_tu(t,x) - n(x)\cdot\nabla u(t,x) = 0$ on $\partial\Omega$
where $n(x)$ is the outward normal vector. 

Discretization in the temporal Fourier domain leads to a scalar Helmholtz equation
\bq
A(\mbf{m}) = \mathsf{diag}(\mbf{s}) - D^T\!D,
\eq
where $D = [I_2\otimes D_1; D_2\otimes I_1]$ with $D_i$ the $(N_i-1)\times N_i$ finite-difference matrix, $I_i$ the $N_i\times N_i$ identity matrix
and $s_i = \omega^2 m_i$ in the interior and $s_i = \omega^2 m_i/2 + \imath\omega\sqrt{m_i}/h$ on the boundary.
The Jacobian is given by
\bq
G(\mbf{m},\mbf{u}) = \mathsf{diag}(\mbf{s}')\mathsf{diag}(\mbf{u}),
\eq
where $s'_i = \omega^2$ in the interior and $s'_i = (\omega^2 + \imath\omega/\sqrt{m_i})/2$ on the boundary.

The observation matrix $P$ samples the solution at the receiver locations using 2D linear interpolation while the point sources are defined using adjoint 2D linear interpolation.

\subsubsection{Ultrasound tomography}

The domain $\Omega = [0,1]\times [0,1]$ m is discretized using $N_1\times N_2$ points. 
The ground-truth $\mbf{m}^*$ as well as the source and receiver locations are shown in figure  \ref{fig:2D_model}. We use a single frequency of $5$ kHz (i.e., $\omega = 10^4\pi$). The data for the ground-truth model are generated using $N_1=N_2=101$ while the following experiments are done with $N_1=N_2=51$.

\textbf{Non-linearity:}
First, we investigate the sensitivity of the misfit functions $\phi$ and $\phi_{\lambda}$ by plotting $\phi(\mbf{m}^* + a_1 \mbf{v}_1 + a_2\mbf{v}_2)$ and $\phi_{\lambda}(\mbf{m}^* + a_1 \mbf{v}_1 + a_2\mbf{v}_2)$ as a function of $(a_1,a_2)$. 
We take $\mbf{v}_1, \mbf{v}_2$ to be slowly oscillatory modes as shown in figure \ref{fig:2D_exp0a}. The misfit as a function of $(a_1,a_2)$ is shown in figure \ref{fig:2D_exp0b}. We see a radically different behaviour for the reduced and penalty methods.

The first exhibits strong non-linearity and some spurious stationary points while for small $\lambda$ the penalty misfit is much better behaved. For larger values $\lambda$ the penalty misfit starts to behave more like the reduced misfit as expected.

\textbf{Inversion:}
For the inversion, we include a regularization term $\frac{\alpha}{2} \|D\mbf{m}\|_2^2$ with $\alpha = 2$ and compare  the GN method ($\epsilon=10^{-6}$, $\delta=10^{-1}$) to the QN method ($\epsilon=10^{-6}$). The initial parameter $\mathbf{m}_0$ is constant at $\frac{1}{4}$ $s^2/m^2$.

The results for the GN method are shown in figure \ref{fig:2D_exp1}.
The convergence history, figure \ref{fig:2D_exp1} (top, left), shows the predicted behaviour of the penalty method; the norm of the gradient of the Lagrangian stalls at $\mcl{O}(\lambda^{-1})$ when using the penalty method. The convergence history of the continuation strategy shows that is possible to reach the desired tolerance by gradually increasing $\lambda$ (using $\lambda = \{0.1, 1, 10, 100, 1000\}$ with a few iterations each). Figure \ref{fig:2D_exp1} (top, right) shows that the methods perform similarly in terms of reconstruction error. The resulting parameter estimates are very similar as can be seen 
in figure \ref{fig:2D_exp1} (bottom). The actual costs of the inversion are listed in table \ref{table:2D_exp1}. The penalty method converges to the same error $\|\mathbf{m}^k - \mathbf{m}^*\|$ in less iterations and uses less PDE solves. Note that all methods start overfitting after a few iterations. This can be countered by including more appropriate regularization or stopping the iterations early. The point here is to show that the penalty method gives similar results as the reduced method.

The results for the QN method are shown in figure \ref{fig:2D_exp2}. The convergence history shows the same behaviour as the previous experiment. The costs of the inversion, shown in table \ref{table:2D_exp2}, are slightly less than those of the GN method. As with the GN method, the penalty method converges in less iterations and uses less PDE solves per iterations. 

\textbf{Sensitivity to noise:}
Results for the QN method on data with 10\% Gaussian noise are shown in figure  \ref{fig:2D_exp3}. Figure \ref{fig:2D_exp4} shows the results on data with 20\% Gaussian noise. These results show that the penalty approach is not overly sensitive to noise and gives very similar --even slightly better-- results compared to the reduced approach. 

\subsubsection{Seismic tomography}
Here, the domain $\Omega = [0,5]\times [0,20]$ km is discretized using $N_1\times N_2$ points. The ground-truth $\mbf{m}^*$ as well as the source and receiver locations are shown in figure \ref{fig:overthrust_model}. We use a frequency of $2$ Hz (i.e., $\omega = 4\pi$). The data for the ground-truth are generated using $N_1=101, N_2=401$ while the following experiments are done with $N_1=51, N_2=201$. 

\textbf{Sensitivity to the initial guess}
For the inversion, we include a regularization term $\frac{\alpha}{2} \|D\mbf{m}\|_2^2$ with $\alpha = 5$ and use the QN method ($\epsilon=10^{-6}$). We will use two different initial guesses, I and II, depicted in figure \ref{fig:overthrust_model}, and see whether the methods converge to the same final iterate. Initial iterate I is much closer to the ground truth than the initial iterate II. This can also be observed when looking at the data produced by these iterates. The first initial iterate produces data that differs only slightly from the observed data and inversion is considered to be easy. The second initial iterate produces data that is shifted significantly with respect to the observed data and inversion is considered to be difficult. It should be noted, however, that a significant source of the error in the initial model is the region near the surface ($z=0$). In practical applications, such large errors in this region of the model might not occur.

The results for initial guess I (figure \ref{fig:overthrust_model}, middle) are shown in figure \ref{fig:2D_overthrust1}. We see that both the reduced and penalty methods converge to roughly the same final iterate and are able to fit the data equally well. Starting from initial guess II (figure \ref{fig:overthrust_model}, bottom), however, we see that the reduced and penalty methods converge to different final iterates. For small $\lambda$, however, the penalty method converges to roughly the same iterate as when starting from a better initial guess. Looking at the data-fit, we observe that the penalty method for small $\lambda$ is still able to fit the data perfectly while the reduced method is not. In seismology, this phenomenon is called \emph{cycle-skipping}.

Figure \ref{fig:2D_overthrust3} shows the convergence of the methods in terms of the data misfit $\|P\mathbf{u} - \mathbf{d}\|_2$ and the distance to the constraint $\|A(\mathbf{m})\mathbf{u} - \mathbf{q}\|_2$. We observe that, when starting from initial guess I, both the penalty and reduced methods converge to approximately the same point. For initial guess II the penalty method for $\lambda=0.1$ and $\lambda=1$ needs a few more iterations, but still converges to the same point as for initial guess I. For $\lambda=10$ and the reduced method, however, the iterations stall at a relatively high data misfit.

These experiments suggests that the penalty method indeed mitigates some of the non-linearity of the problem, allowing the optimization to converge to the same final iterate, even when the initial guess is further away from the ground truth. 
\section{Discussion}
\label{discussion}
This paper lays out the basics of an efficient implementation of the penalty method for PDE-constrained optimization problems arising in inverse problems. While the initial results are promising, some aspects of the proposed method warrant further investigation.

Even though the theoretical results suggest that the penalty approach can find a stationary point of the Lagrangian with finite precision with a finite $\lambda$, it is not clear how to choose a suitable value for $\lambda$ a priori. Our analysis and results suggest that choosing $\lambda$ to be a small fraction of $\|PA^{-1}\|_2^2$ at the initial iterate yields good results. A continuation strategy for $\lambda$ is needed if we want to guarantee finding a stationary point of the Lagrangian with preset tolerance. A natural way to do this seems to be detecting when the penalty method stalls and subsequently reducing $\lambda$. The numerical results suggest that such an approach is viable, but further study is needed in order to develop a robust continuation strategy.

The penalty formulation essentially relaxes the constraints and therefore allows for errors in the physics as well as the data. As a result, the penalty formulation leads to reduced sensitivity of the final reconstruction to the initial guess. Further investigation is needed to characterize this robustness.  

Finally, the Hessian of the penalty objective exhibits additional structure that could potentially be exploited. In particular, the penalty-method GN Hessian is full rank and allows for a natural sparse approximation $H_{\lambda} \approx \lambda G^T\!G$ (cf. equation \ref{eq:Hl}). The reduced GN Hessian, on the other hand, has rank of at most $ML$ and does not permit such a natural sparse approximation.

\section{Conclusions}
\label{conclusion}
We have presented a penalty method for PDE-constrained optimization with linear PDEs with applications to inverse problems. The method is based on a quadratic penalty formulation of the constrained problem. This reformulation results in a an unconstrained optimization problem in both the parameters and the state variables. To avoid having to store and update the state variables as part of the optimization, we explicitly eliminate
the state variables by solving an overdetermined linear system. The proposed method combines features from both the \emph{all-at-once} approach, in which the states and parameters are updated simultaneously, and the conventional \emph{reduced} approach, in which the PDE-constraints are eliminated explicitly. While having a similar computational complexity as the conventional reduced approach, the penalty approach explores a larger search space by not satisfying the PDE-constraints exactly. 

We show that we can (theoretically) find a stationary point of the Lagrangian of the constrained problem within a given tolerance as long as the penalty parameter, $\lambda$, is chosen large enough. While theoretically we need $\lambda \uparrow \infty$, we can suffice with solving the problem for a finite $\lambda$ to reach the stationary point within finite precision. 

The main algorithmic difference with the conventional reduced approach is the way the states are eliminated from the problem. Instead of solving the PDEs, we formulate an overdetermined system of equations that consists of the discretized PDE and the measurements. We discuss the properties of this augmented system and show with a few numerical examples that both the structure of the system as well as the eigenvalues are not altered dramatically as compared the original PDE. Thus, it is plausible that the augmented system can be solved as efficiently using the same approach as is used for the original PDE.

The numerical examples show that very good results can be obtained by using even a single, relatively small, value of $\lambda$. An ad-hoc continuation strategy further shows that it is viable to gradually increase $\lambda$ in order to reach the desired tolerance.

The numerical examples further show that when using the penalty formulation, the optimization problem may actually be less non-linear and that in some cases a better parameter reconstruction is obtained as compared to the conventional reduced approach. In particular, the results show that the penalty method is not overly sensitive to noise and less sensitive to the initial model than the conventional reduced approach. 

Thus, the proposed approach is a viable alternative to the conventional reduced approach for solving inverse problems with PDE-constraints.

\section*{Acknowledgments}
This work was in part financially supported by the Natural Sciences and Engineering Research Council of Canada via the Collaborative Research and Development Grant DNOISEII (375142--08). This research was carried out as part of the SINBAD II project which is supported by the following organizations: BG Group, BGP, CGG, Chevron, ConocoPhillips, DownUnder GeoSolutions, Hess, Petrobras, PGS, Schlumberger, Sub Salt Solutions and Woodside.

\clearpage

\begin{table}
\centering
\begin{tabular}{c|c|c|c}
 	& small 2D 	& large 2D 	& industrial 3D 	\\
\hline
$K$	&   $10^2$	& $10^3$	&    $10^6$     	\\
$L$	&   $10^2$	& $10^3$	&	$10^6$			\\
$M$	&   $10^6$	& $10^9$	&	$10^{12}$		\\
$N$	&   $10^3$  & $10^6$	&	$10^9$			\\
\end{tabular}
\caption{Typical size of seismic inverse problem in terms $K$: of the number of experiments, $L$: the number of measurements per experiment, $N$: the number of discretization points and $M$: the number of parameters.}
\label{table:sizes}
\end{table}

\begin{table}
\centering
\begin{tabular}{cccc}
& $\lambda = 0.1$ & $\lambda = 1$ & $\lambda = 10$ \\
\hline
L = 1 & 9.83e-01 & 9.84e-01 & 9.89e-01 \\
L = 10 & 9.68e-02 & 5.07e-01 & 9.11e-01 \\
L = 20 & 9.19e-02 & 5.01e-01 & 9.09e-01 \\
\hline
\end{tabular}

\caption{Ratio of the condition numbers of $A^T\!A + \lambda P_L^T\!P_L$ and $A^T\!A$ for various $\lambda$ and $L$, where $A$ is a finite-difference discretization of $\imath(10\pi) - \partial_x^2$ on $x\in [0,1]$ and $P_L$ is a restricted identity matrix of rank $L$. }
\label{table:example2}
\end{table}

\begin{table}
\centering
\begin{tabular}{cccc}
& $\lambda = 0.1$ & $\lambda = 1$ & $\lambda = 10$ \\
\hline
L = 1 & 1.80e-01 & 5.44e-01 & 9.17e-01 \\
L = 10 & 1.03e-01 & 5.06e-01 & 9.10e-01 \\
L = 20 & 9.10e-02 & 5.00e-01 & 9.09e-01 \\
\hline
\end{tabular}

\caption{Ratio of the condition numbers of $A^T\!A + \lambda P_L^T\!P_L$ and $A^T\!A$ for various $\lambda$ and $L$, where $A$ is a finite-difference discretization of $(10\pi)^2 m + \partial_x^2$ on $x\in [0,1]$ and $P_L$ is a restricted identity matrix of rank $L$. }
\label{table:example3}
\end{table}

\begin{table}
\begin{tabular}{c|c|c|p{5cm}}
         & \# PDE's & Storage & Gauss-Newton update \\
\hline
penalty  &  $K$  &    $N + M$     & solve matrix-free linear system in $M$ unknowns, requires $K$ (overdetermined) PDE solves per mat-vec \\
\hline
reduced  &  $2K $  &    $2N + M$     & solve matrix-free linear system in $M$ unknowns, requires $2K$ PDE solves per mat-vec                   \\
\hline
all-at-once&   0   &    $2KN + M$   &  solve sparse symmetric, possibly indefinite system in $(2KN + M) \times (2KN + M)$ unknowns \\ 
\end{tabular}
\caption{Leading order computational and storage costs per iteration of different methods; $K$ denotes the number of experiments and $N$ denotes the number of gridpoints and $M$ denotes the number of parameters.}
\label{table:costs}
\end{table}

\begin{table}
\centering
\begin{tabular}{cccccc}
& reduced & $\lambda = 0.1$ & $\lambda = 1$ & $\lambda = 10$ & increasing $\lambda$ \\
\hline
iterations & 7 & 5 & 6 & 7 & 6 \\
PDE solves & 496 & 222 & 223 & 280 & 292 \\
\hline
\end{tabular}

\caption{Costs of the 1D DC resistivity inversion.}
\label{table:1D_exp1}
\end{table}

\begin{table}
\centering
\begin{tabular}{cccccc}
& reduced & $\lambda = 0.1$ & $\lambda = 1$ & $\lambda = 10$ & increasing $\lambda$ \\
\hline
iterations & 6 & 4 & 5 & 6 & 7 \\
PDE solves & 172 & 38 & 56 & 82 & 99 \\
\hline
\end{tabular}

\caption{Costs of the 2D ultrasound inversion with a GN method.}
\label{table:2D_exp1}
\end{table}

\begin{table}
\centering
\begin{tabular}{ccccc}
& reduced & $\lambda = 0.1$ & $\lambda = 1$ & $\lambda = 10$ \\
\hline
iterations & 36 & 18 & 29 & 34 \\
PDE solves & 76 & 21 & 31 & 35 \\
\hline
\end{tabular}

\caption{Costs of the 2D ultrasound inversion with a QN method.}
\label{table:2D_exp2}
\end{table}
\clearpage

\begin{figure}
\centering
\begin{tabular}{ccc}
\includegraphics[scale=.2]{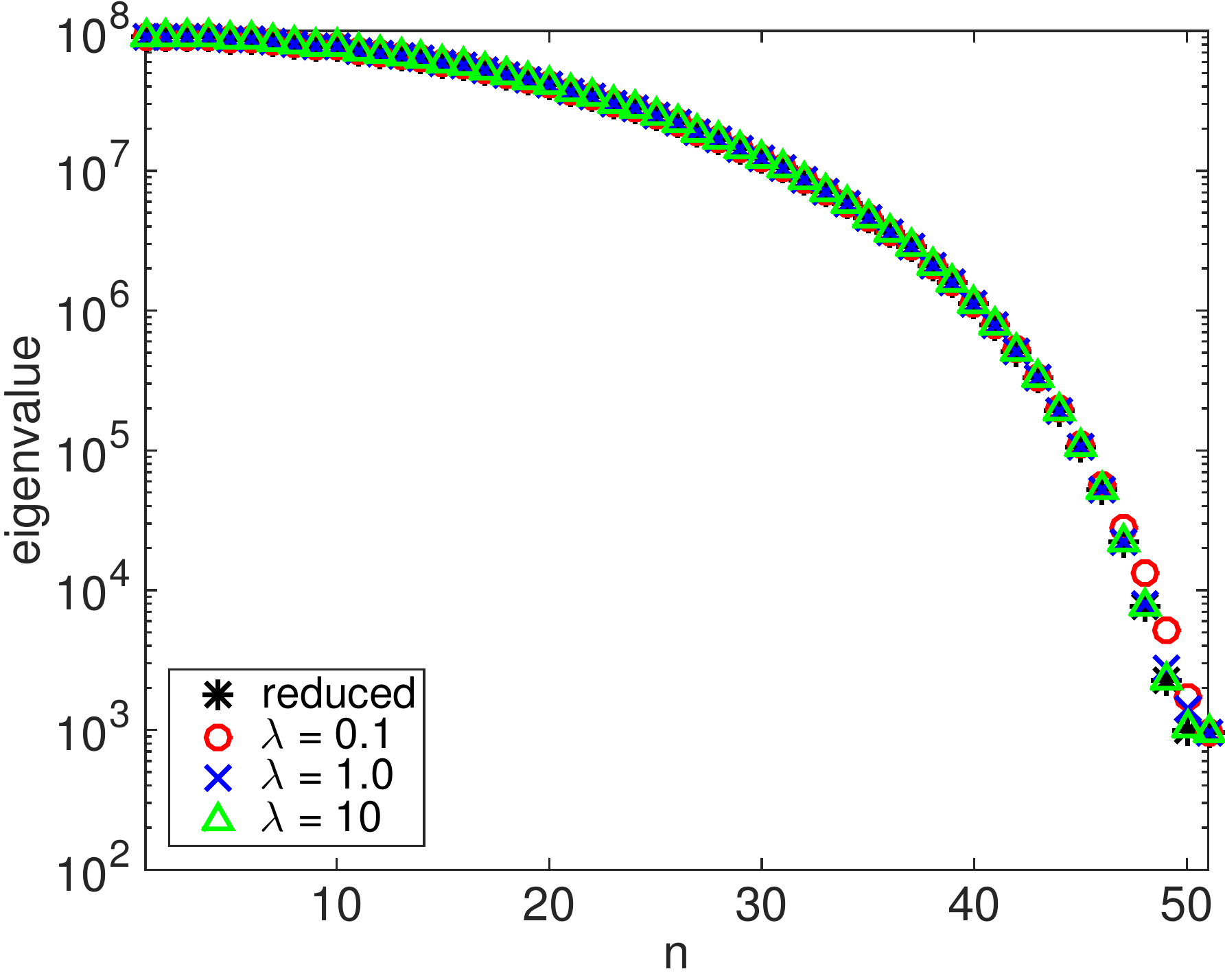}&
\includegraphics[scale=.2]{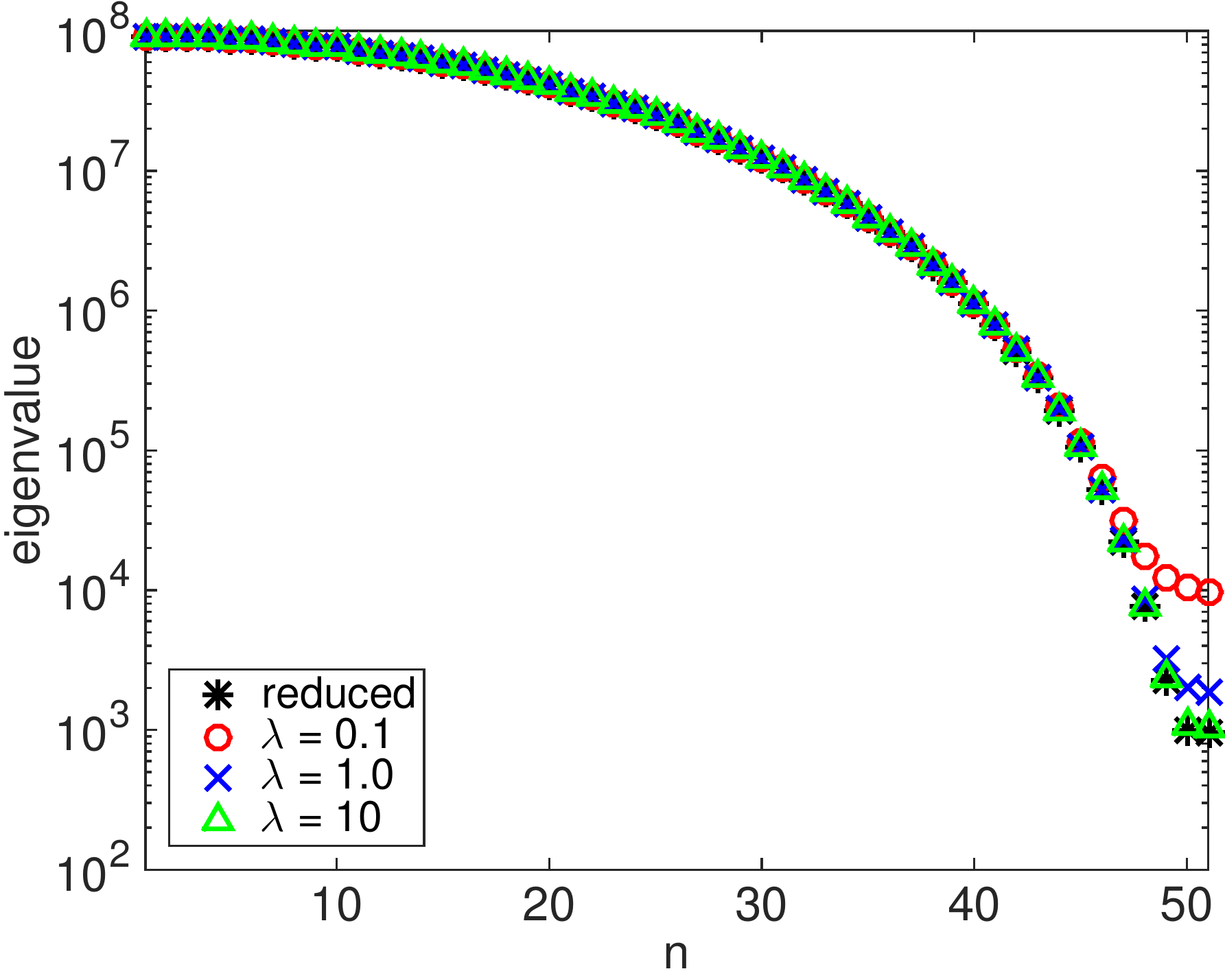}&
\includegraphics[scale=.2]{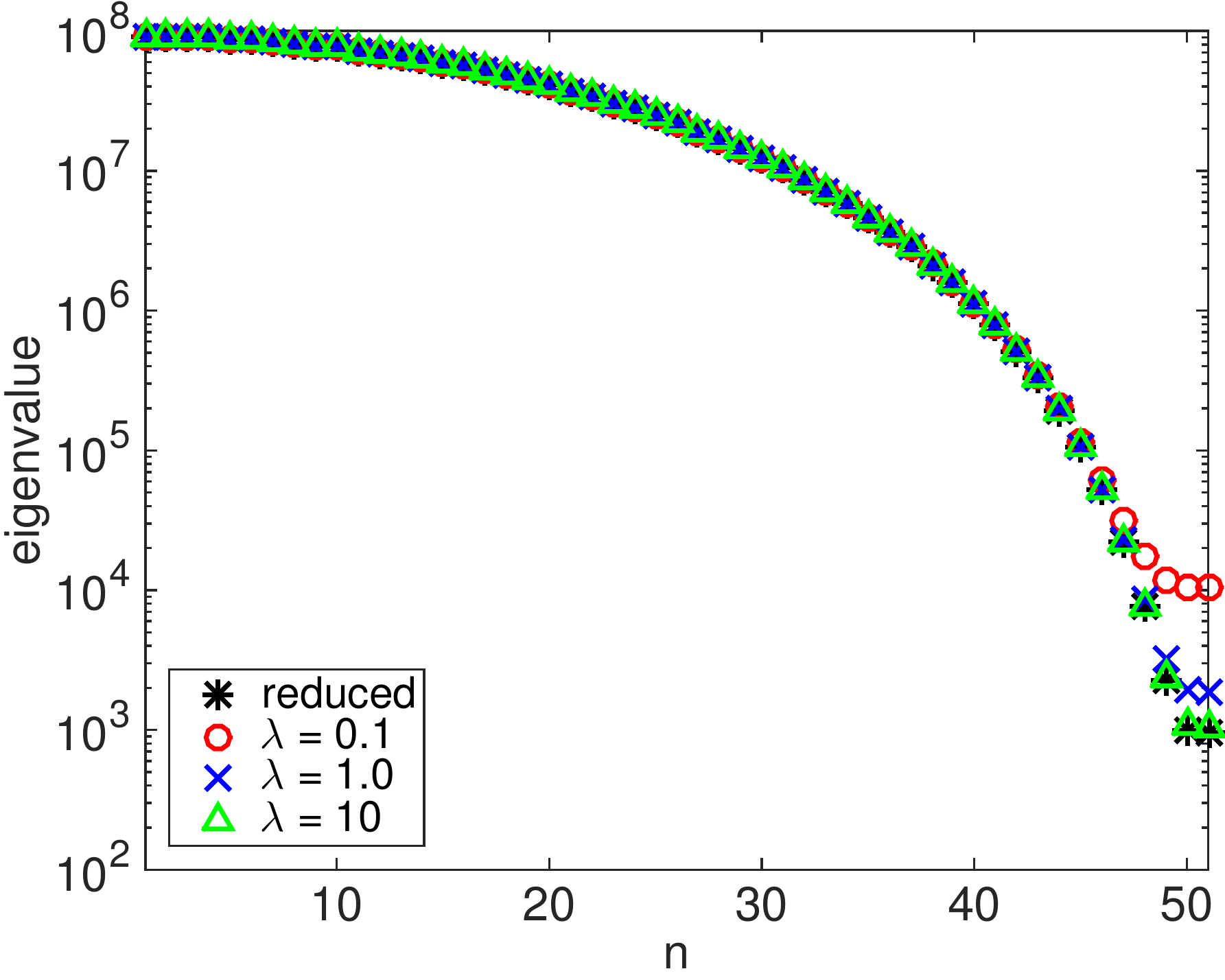}\\
{\small $L = 1$}&{\small $L = 10$}&{\small $L = 20$}\\
\end{tabular}
\caption{Eigenvalues of the augmented system,  $A^T\!A + \lambda P_L^T\!P_L$, for various $\lambda$ and $L$, where $A$ is a finite-difference discretization of $\imath(10\pi) - \partial_x^2$ on $x\in [0,1]$
and $P_L$ is a restricted identity matrix of rank $L$. For comparison, the eigenvalues of the original system $A^T\!A$ are also shown.}
\label{fig:example2}
\end{figure}

\begin{figure}
\centering
\begin{tabular}{ccc}
\includegraphics[scale=.2]{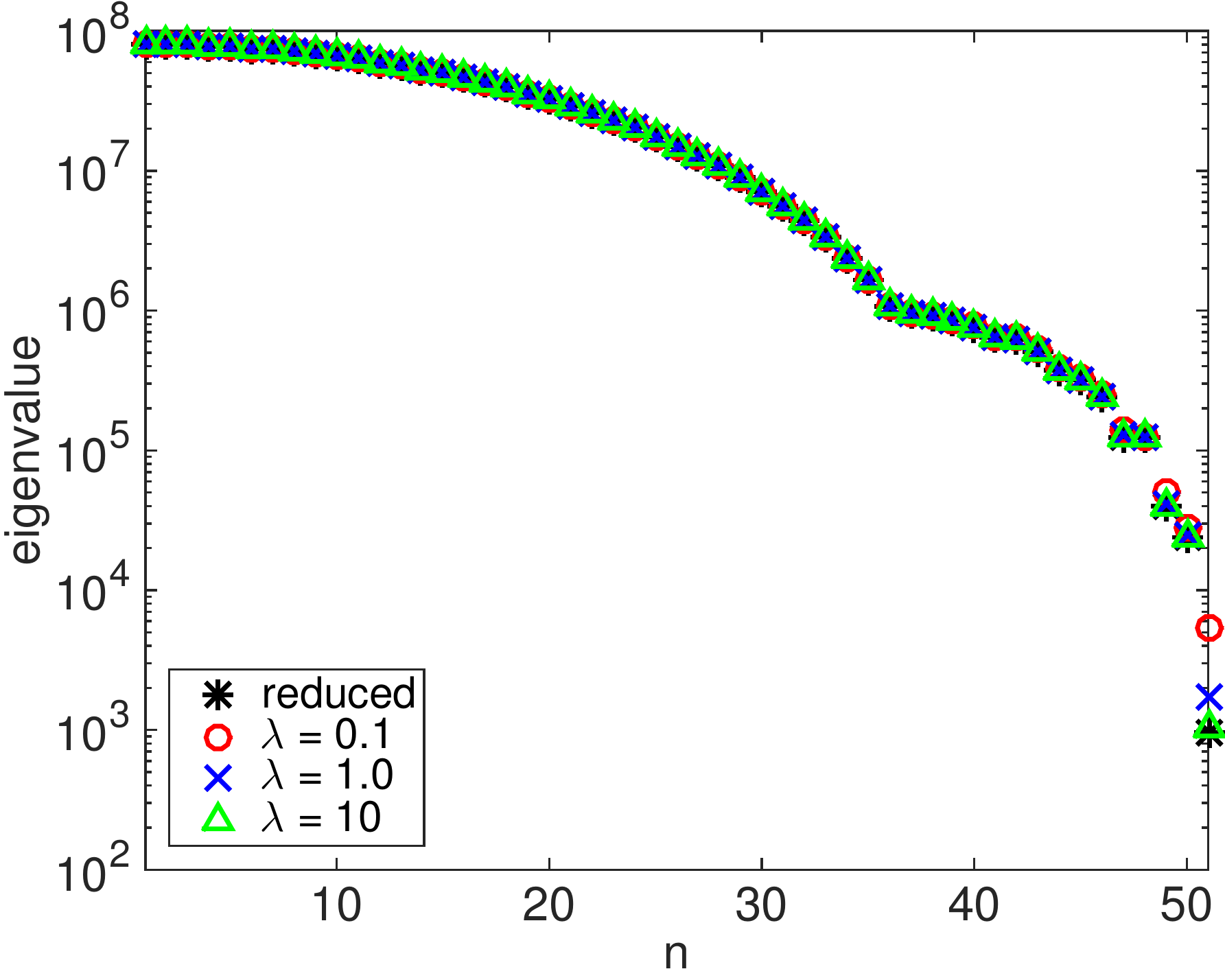}&
\includegraphics[scale=.2]{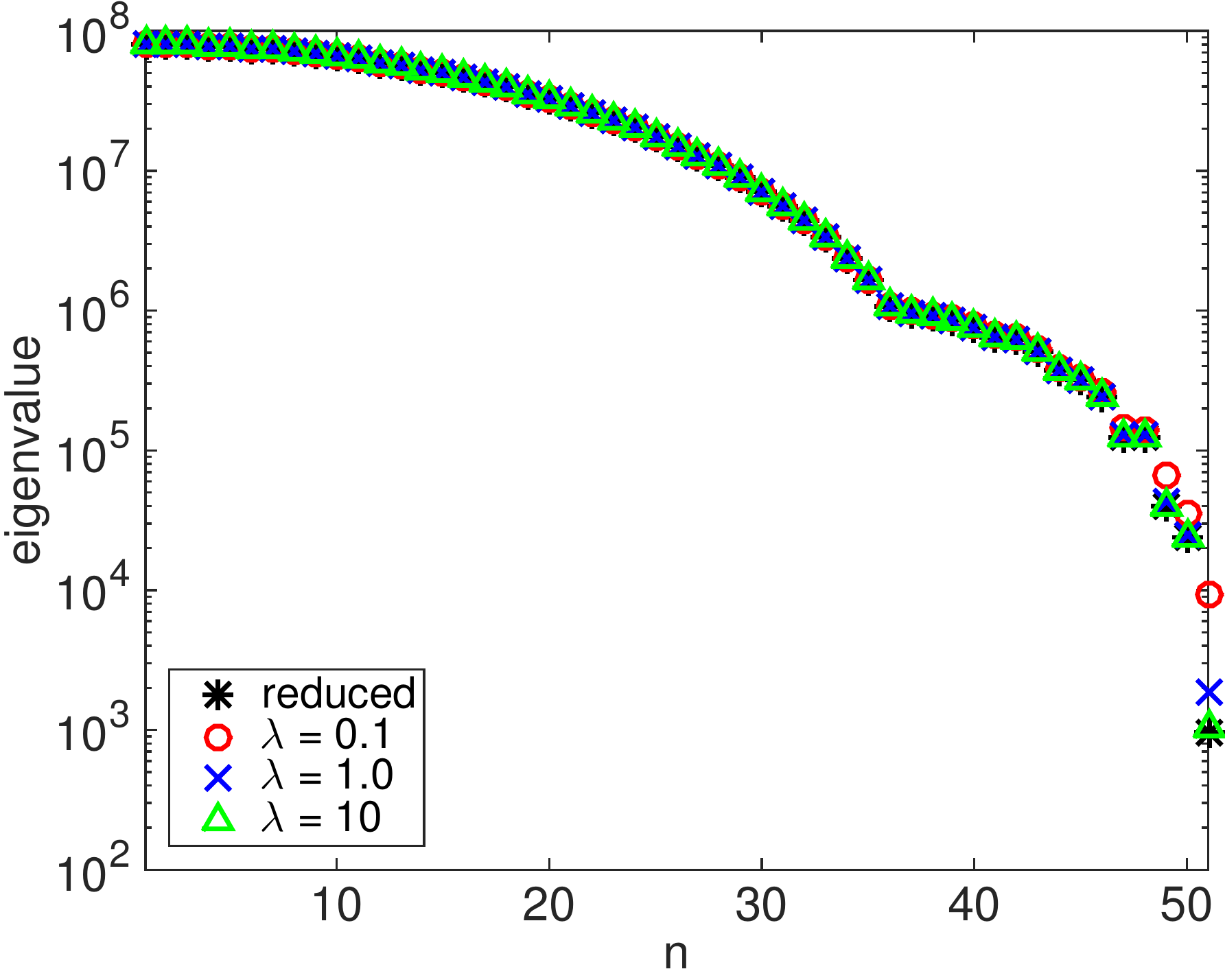}&
\includegraphics[scale=.2]{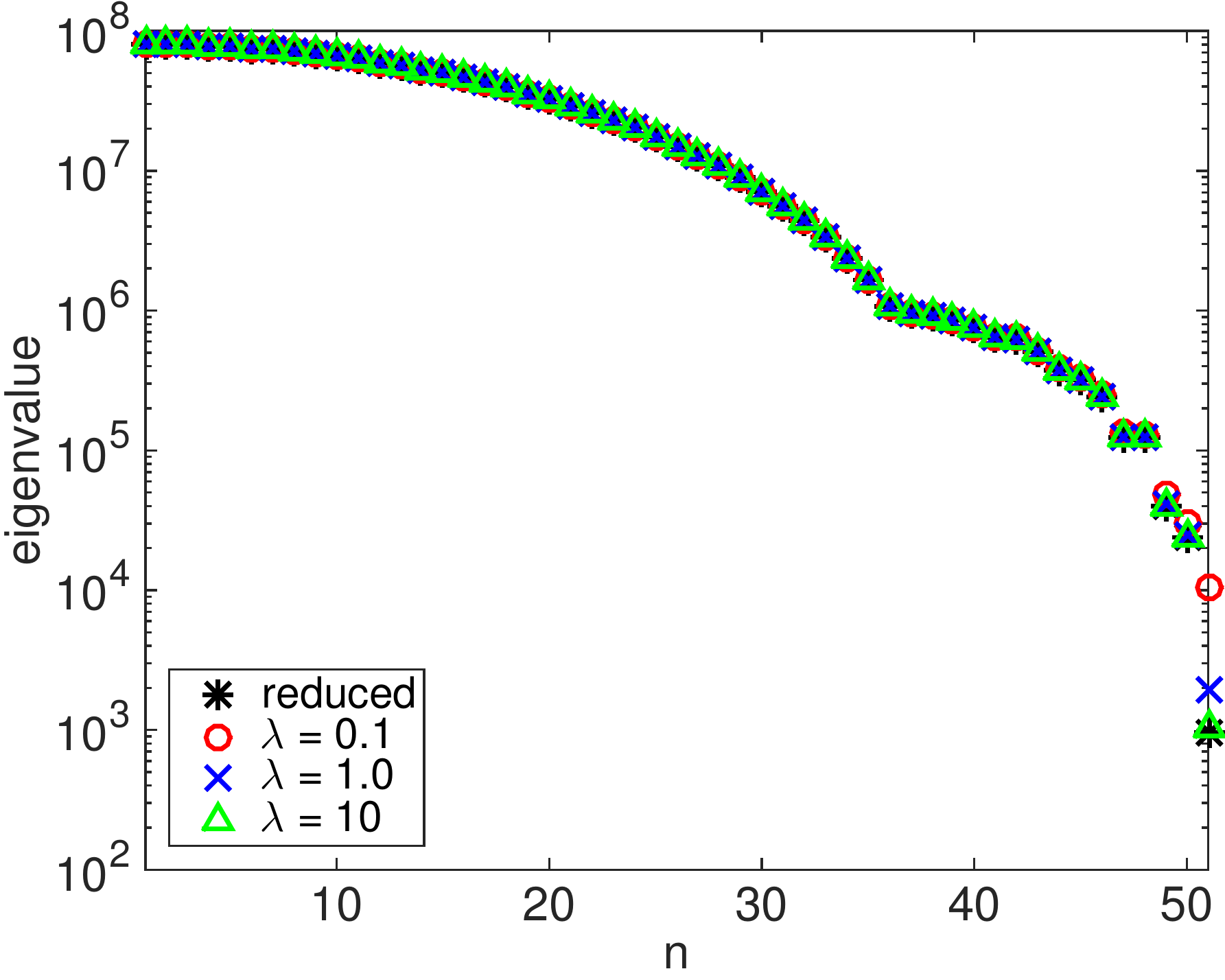}\\
{\small $L = 1$}&{\small $L = 10$}&{\small $L = 20$}\\
\end{tabular}
\caption{Eigenvalues of the augmented system,  $A^T\!A + \lambda P_L^T\!P_L$, for various $\lambda$ and $L$, where $A$ is a finite-difference discretization of $(10\pi)^2 m + \partial_x^2$ on $x\in [0,1]$
and $P_L$ is a restricted identity matrix of rank $L$. For comparison, the eigenvalues of the original system $A^T\!A$ are also shown.}
\label{fig:example3}
\end{figure}

\begin{figure}
\centering
\begin{tabular}{cc}
\includegraphics[scale=.4]{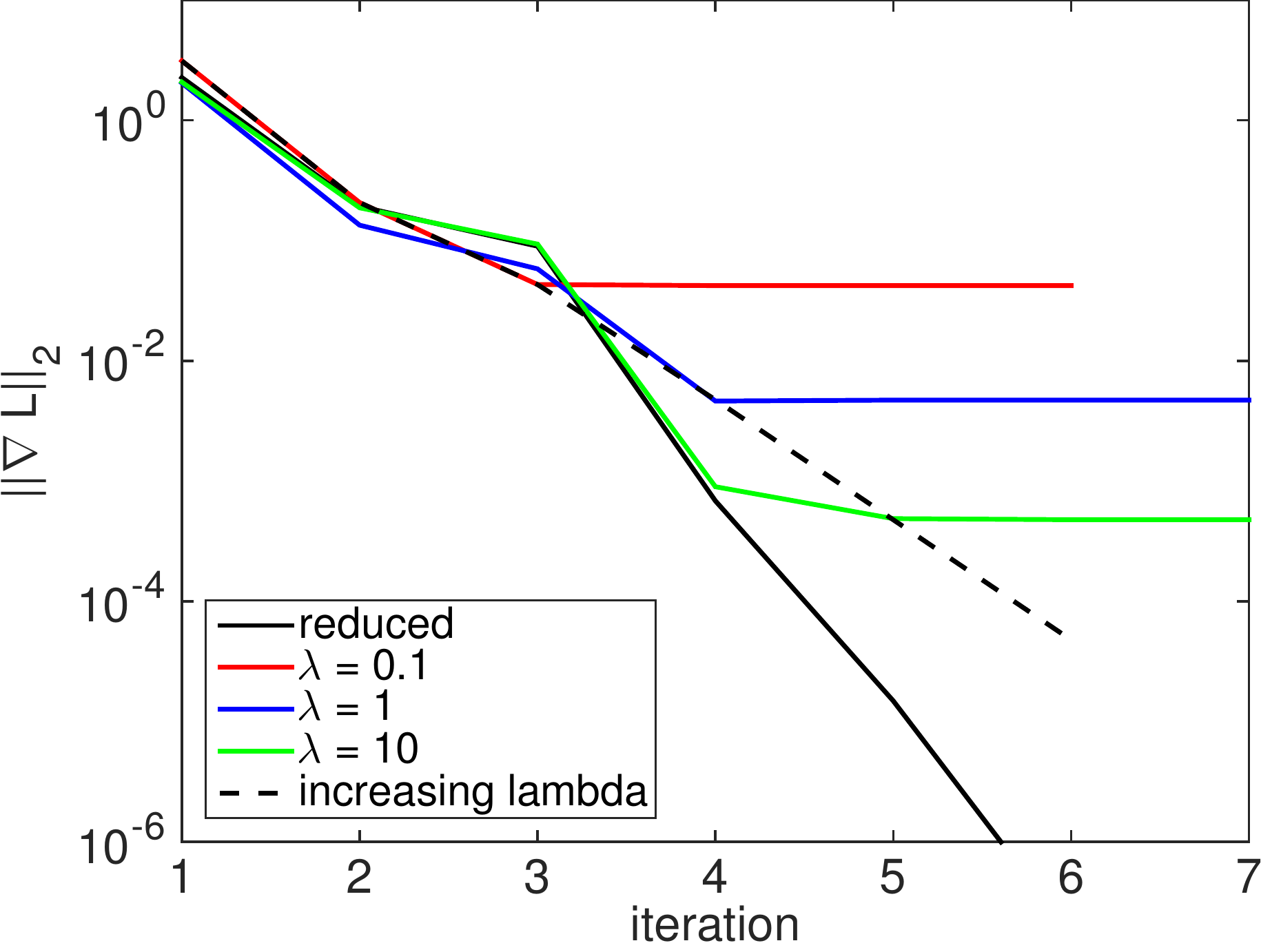}&
\includegraphics[scale=.4]{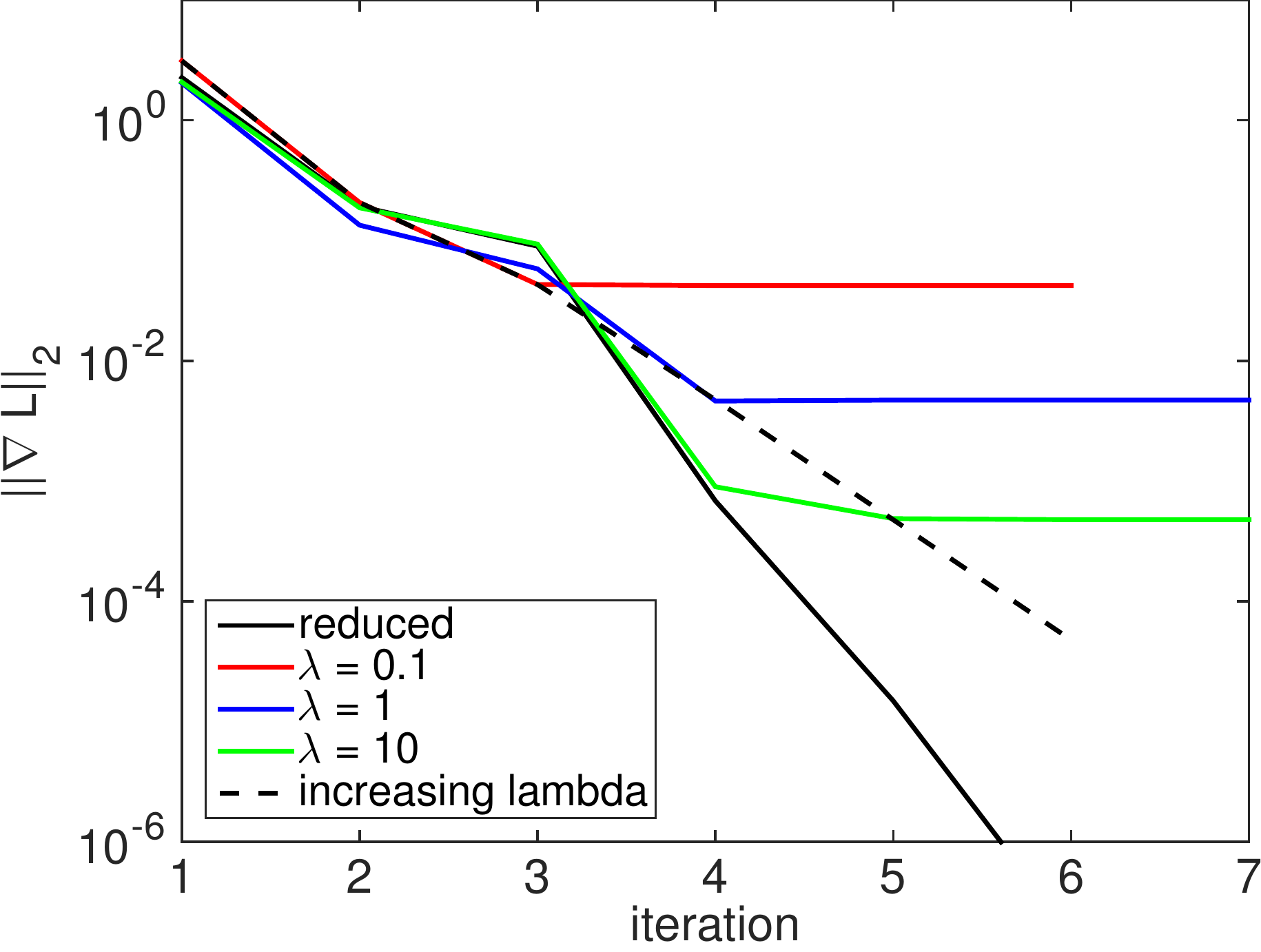}\\
{\small (a)}&{\small (b)}\\
\end{tabular}
\caption{Convergence history and reconstructions for 1D resistivity problem. Even though the penalty method does not converge to same tolerance as the reduced method in terms of the gradient of the Lagrangian, the resulting parameter estimates are almost the same.}
\label{fig:1D_exp1}
\end{figure}

\begin{figure}
\centering
\includegraphics[scale=.4]{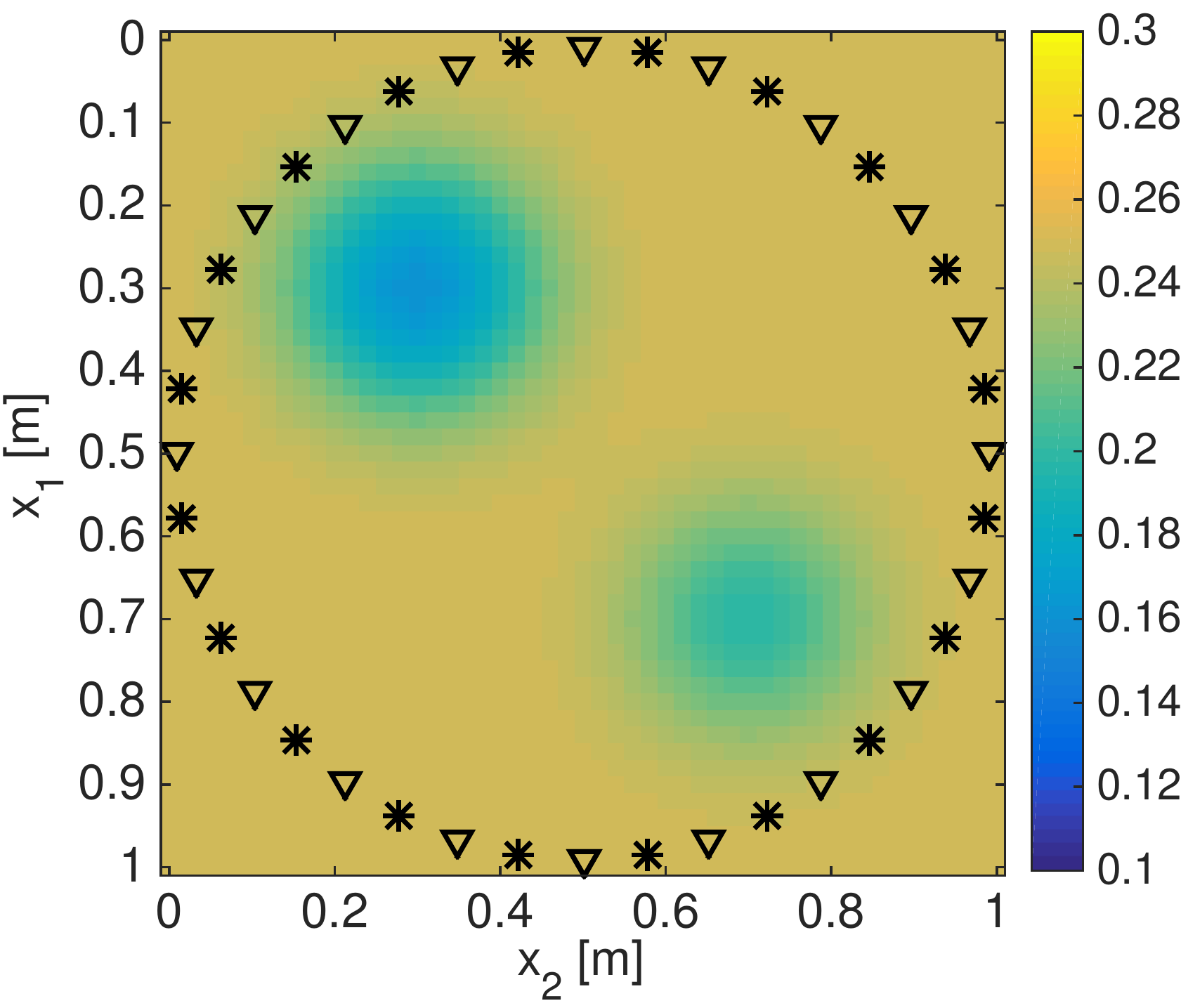}\\
\caption{Ground truth model ($s^2/km^2$) and locations of the sources ($*$) and receivers ($\bigtriangledown$)}
\label{fig:2D_model}
\end{figure}

\begin{figure}
\centering
\begin{tabular}{cc}
\includegraphics[scale=.2]{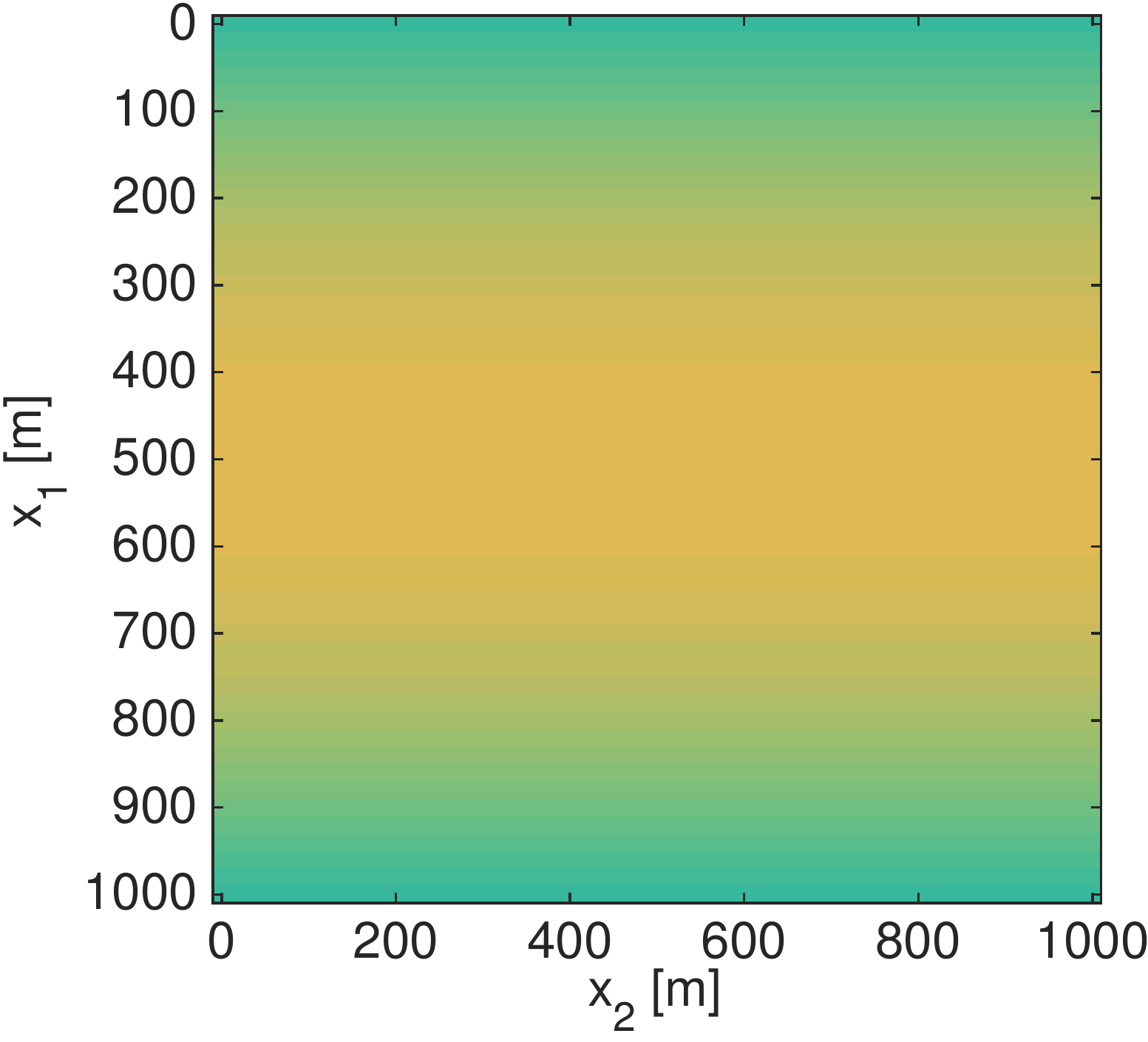}&
\includegraphics[scale=.2]{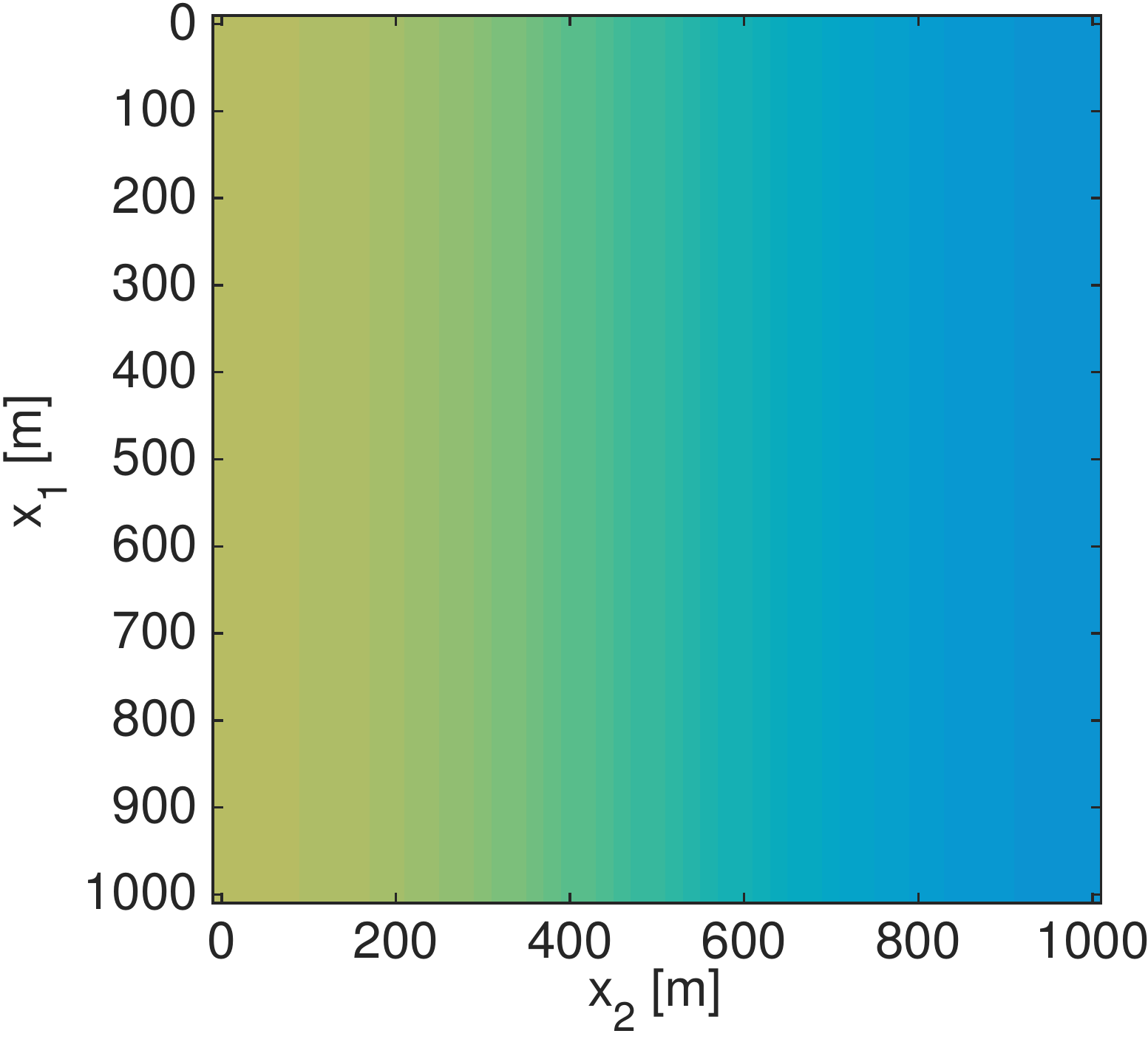}\\
\end{tabular}
\caption{Perturbations $\mathbf{v}_1$ and $\mathbf{v}_2$ used to plot the misfit $\phi(\mathbf{m}^* + a1\mathbf{v}_1 + a_2\mathbf{v}_2)$ and $\phi_{\lambda}(\mathbf{m}^* + a_1\mathbf{v}_1 + a_2\mathbf{v}_2)$ in figure \ref{fig:2D_exp0b}.}
\label{fig:2D_exp0a}
\end{figure}

\begin{figure}
\centering
\begin{tabular}{cccc}
\includegraphics[scale=.2]{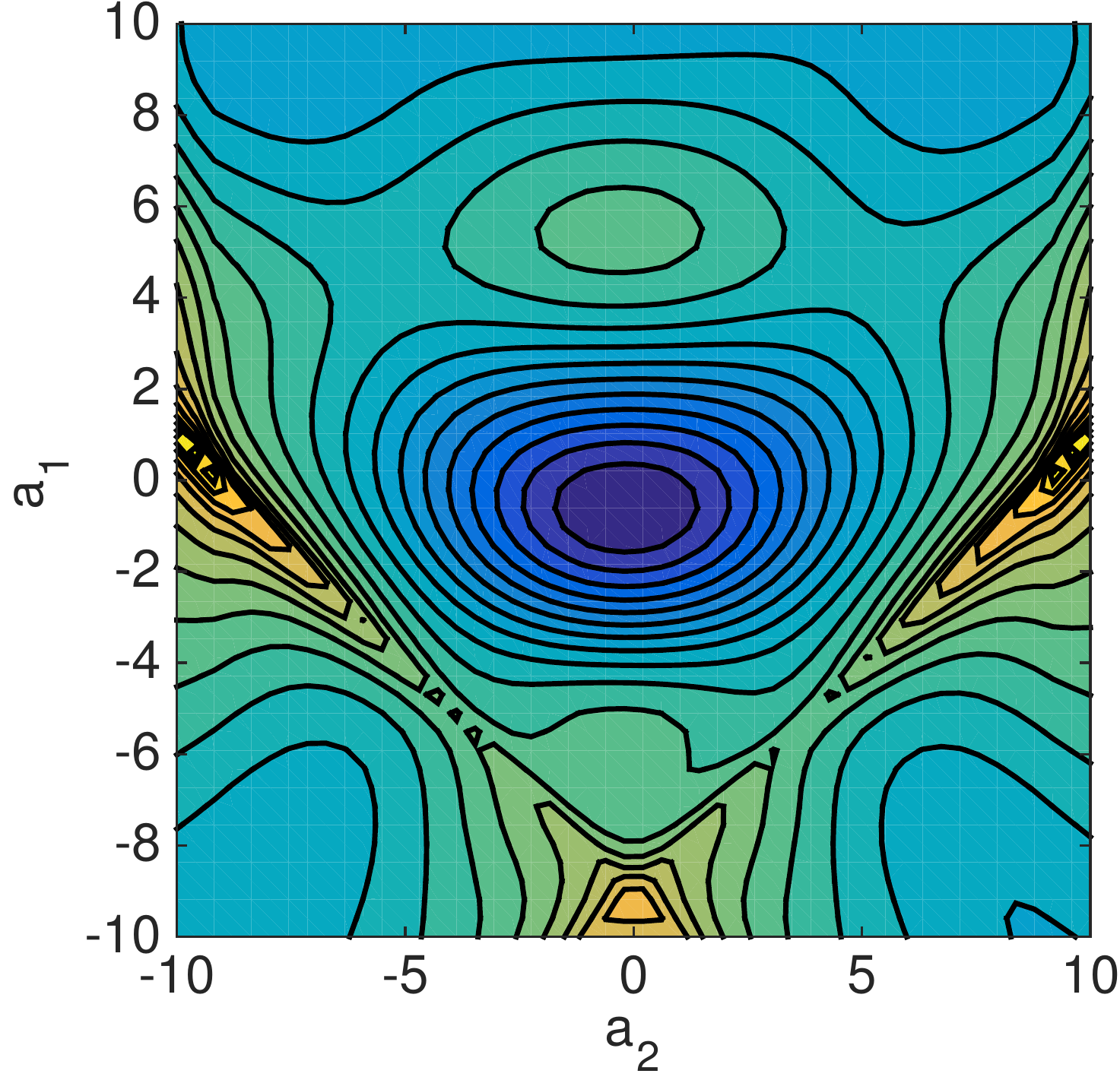}&
\includegraphics[scale=.2]{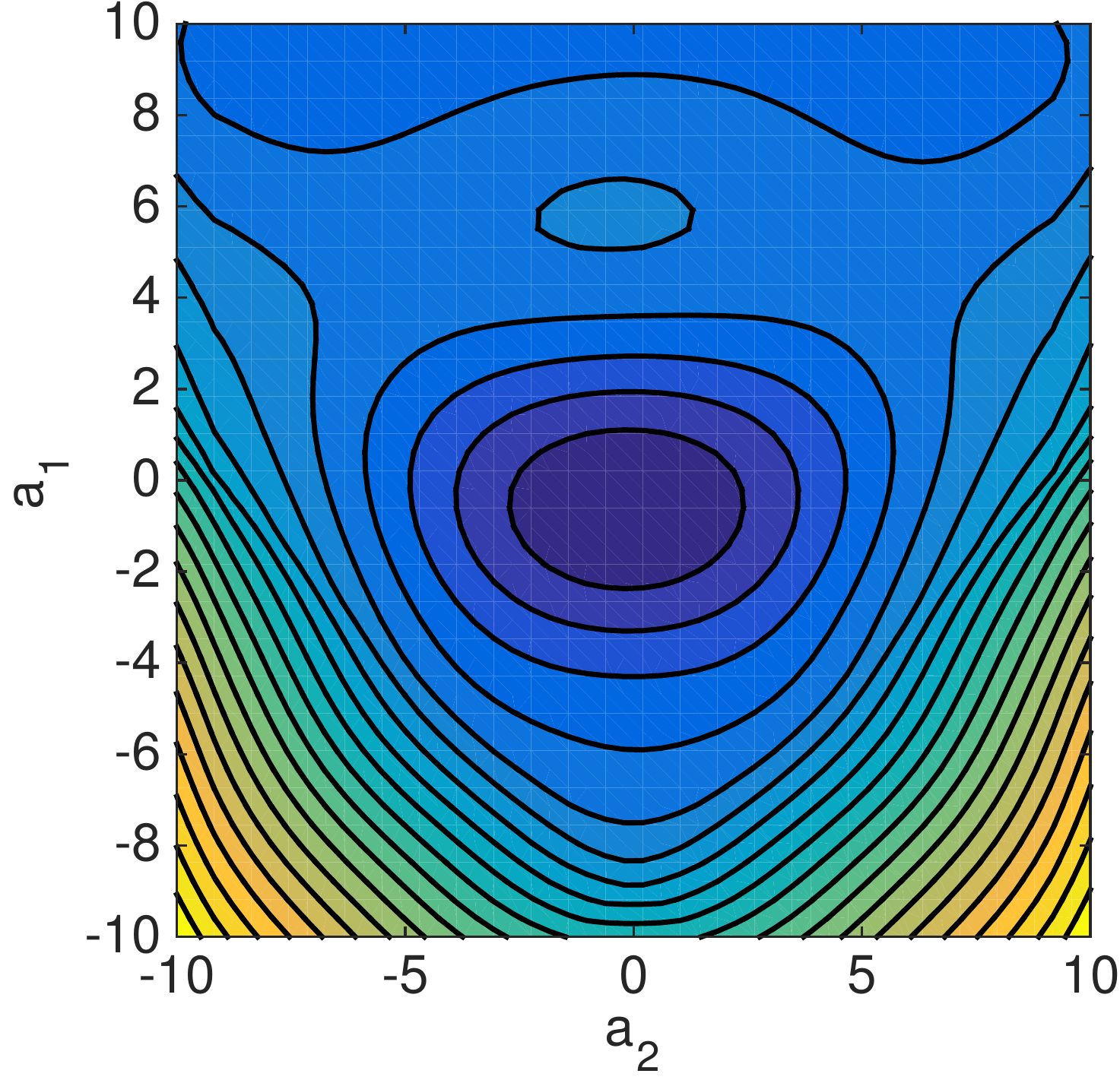}&
\includegraphics[scale=.2]{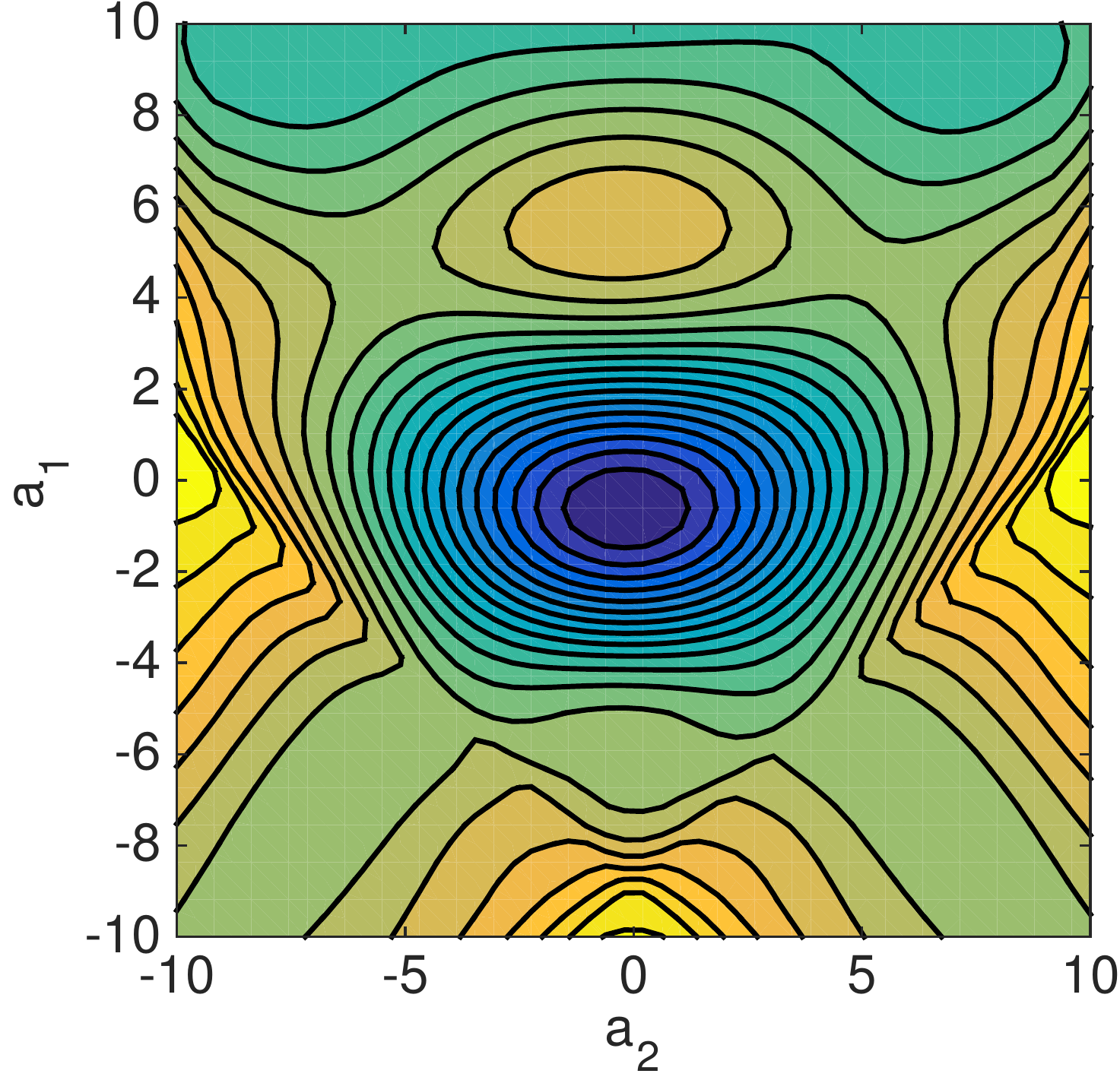}&
\includegraphics[scale=.2]{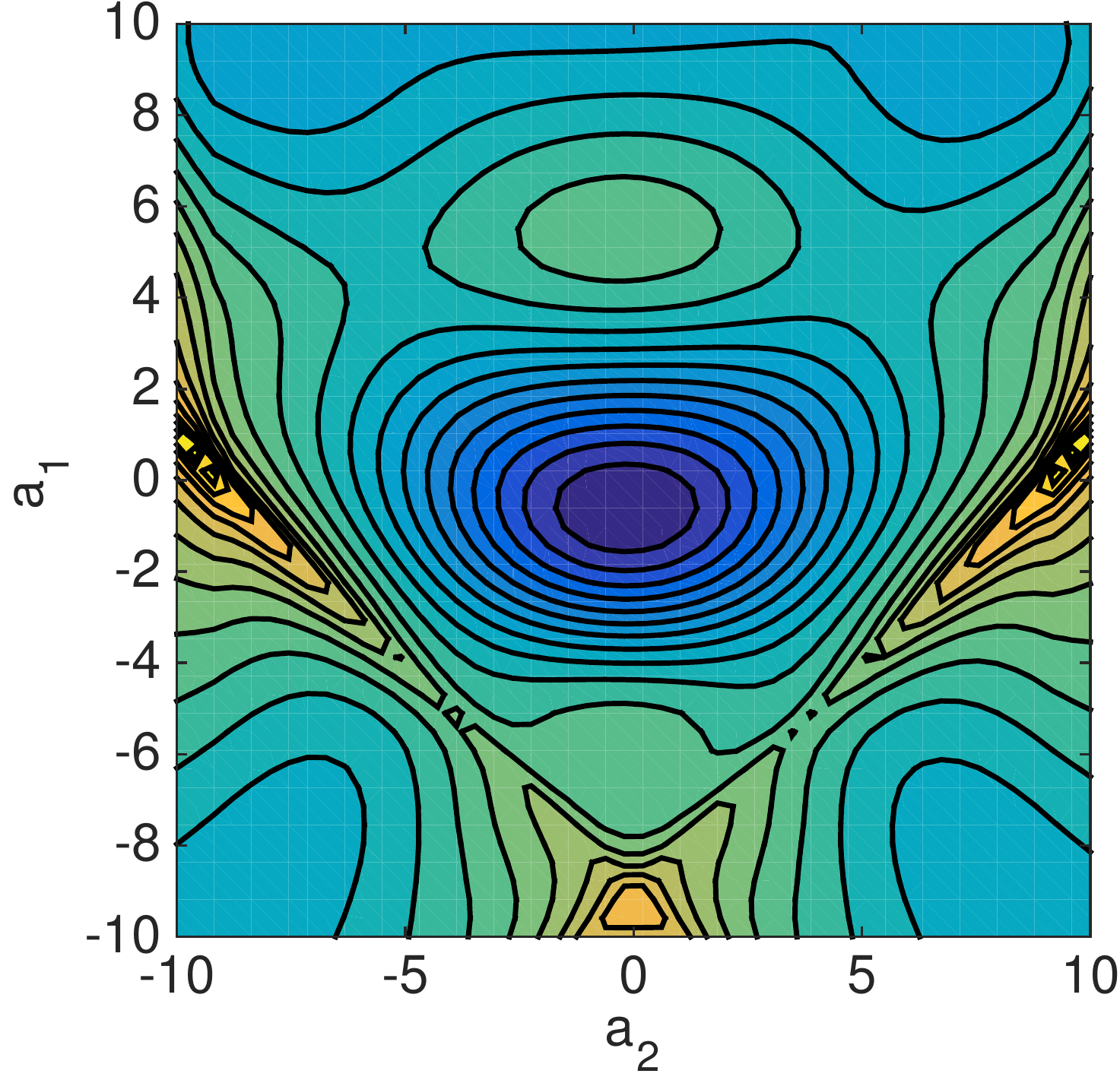}\\
{\small reduced}&{\small $\lambda=0.1$}&{\small $\lambda=1$}&{\small $\lambda=10$}\\
\end{tabular}
\caption{Misfit in the direction of the perturbations shown in figure \ref{fig:2D_exp0a}. For small $\lambda$, the reduced penalty objective $\phi_{\lambda}$ is less non-linear than the reduced objective $\phi$.}
\label{fig:2D_exp0b}
\end{figure}

\begin{figure}
\centering
\begin{tabular}{cc}
\includegraphics[scale=.3]{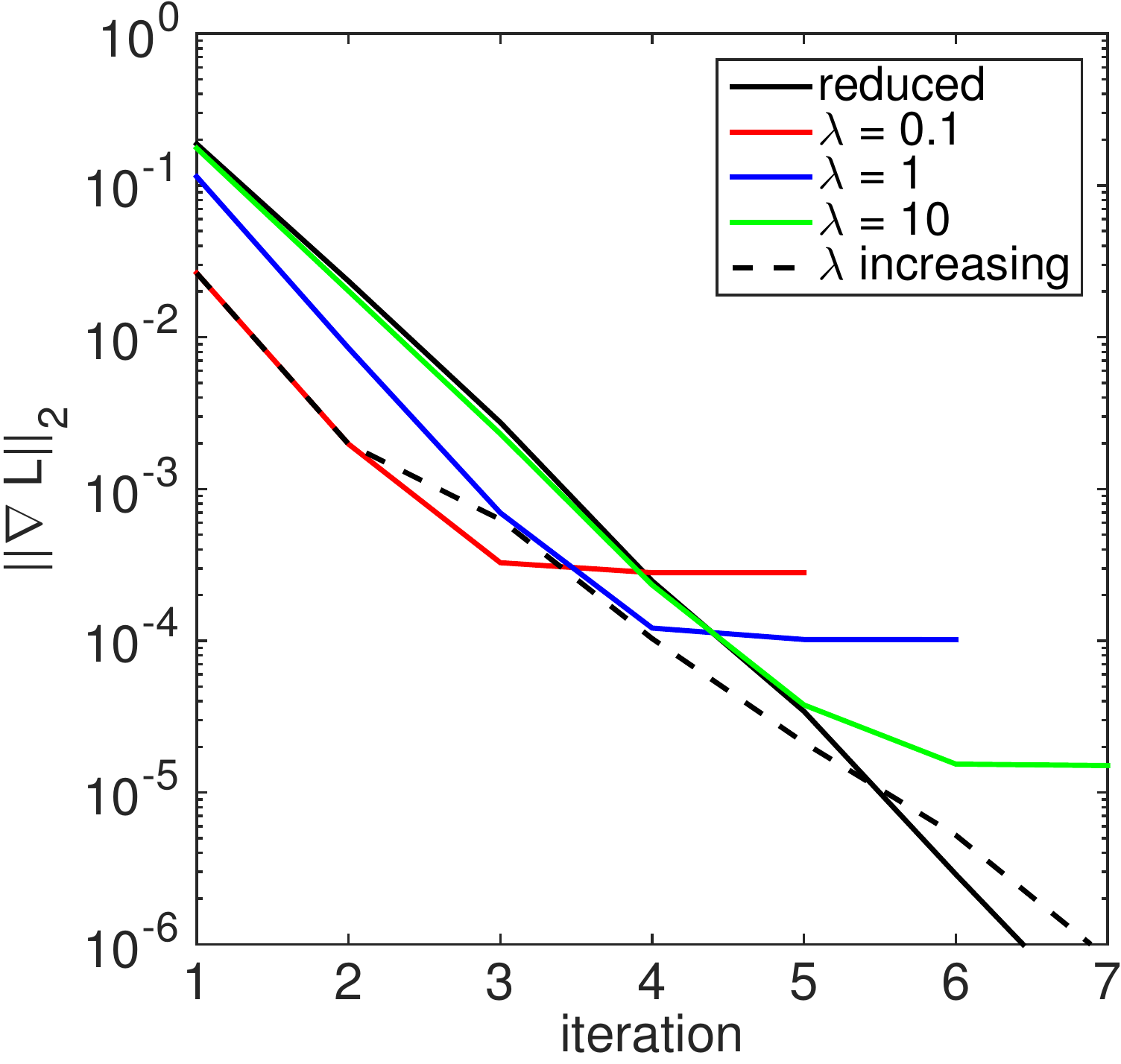}&
\includegraphics[scale=.3]{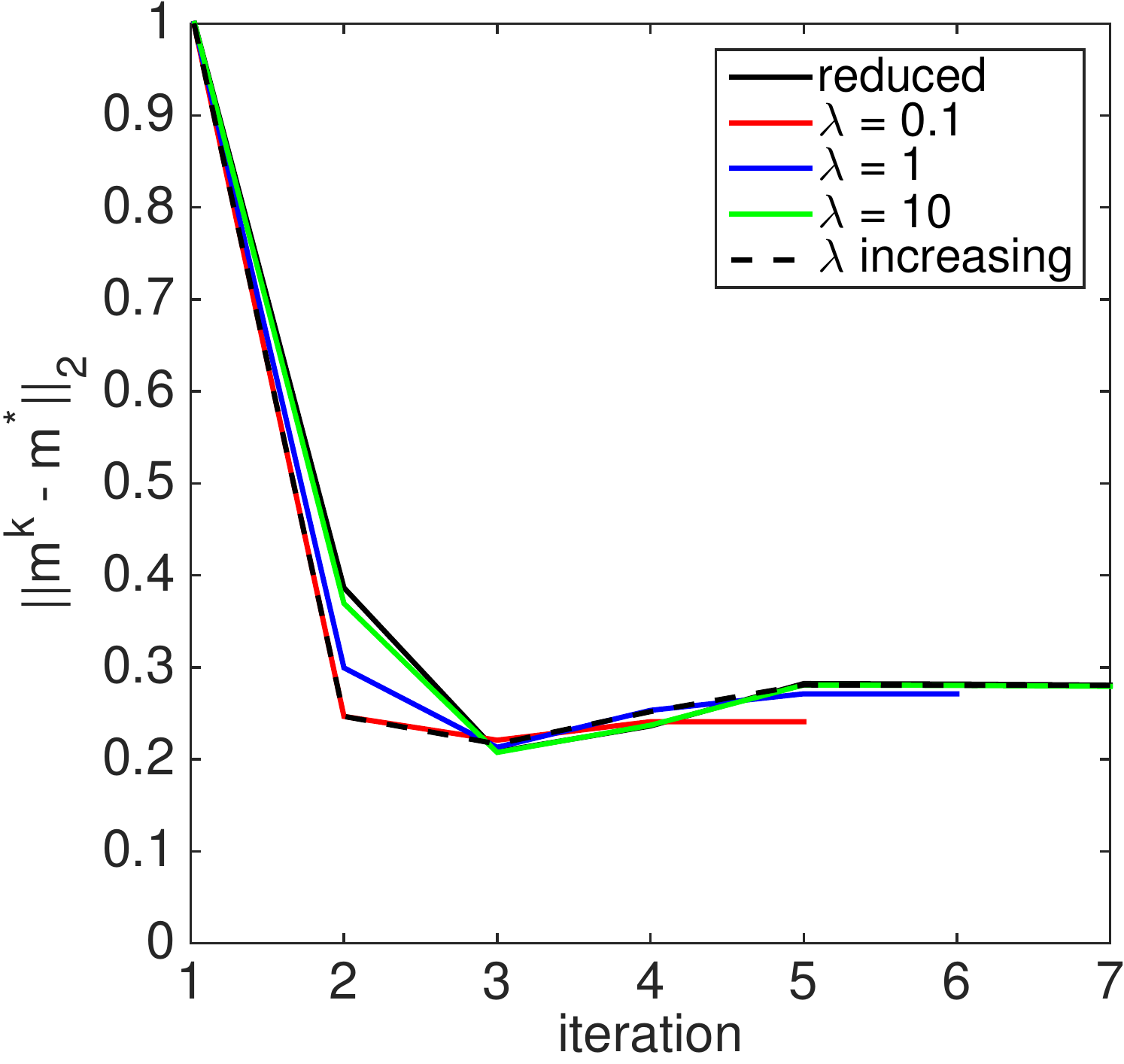}\\
\end{tabular}
\centering
\begin{tabular}{cccc}
\includegraphics[scale=.2]{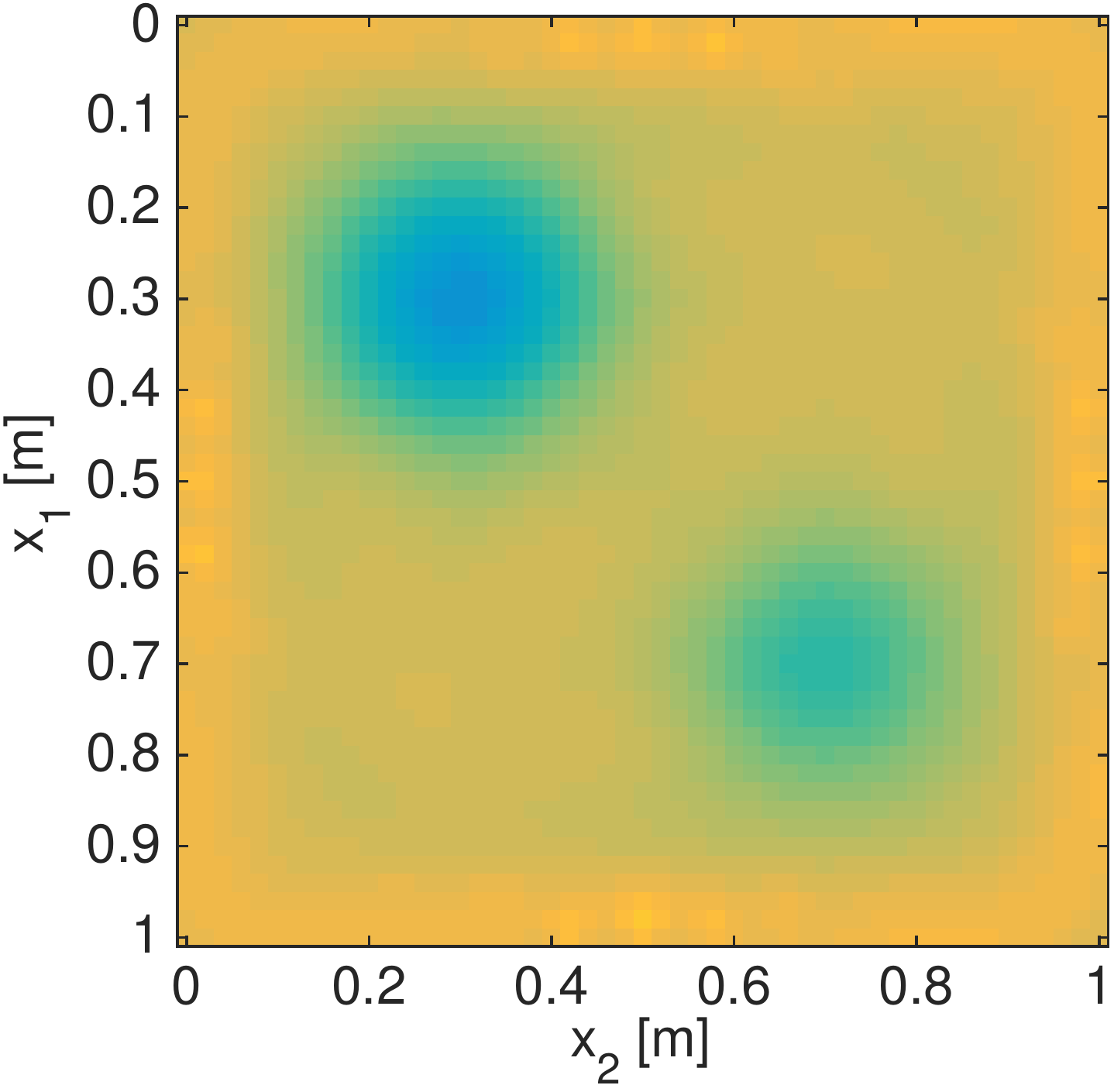}&
\includegraphics[scale=.2]{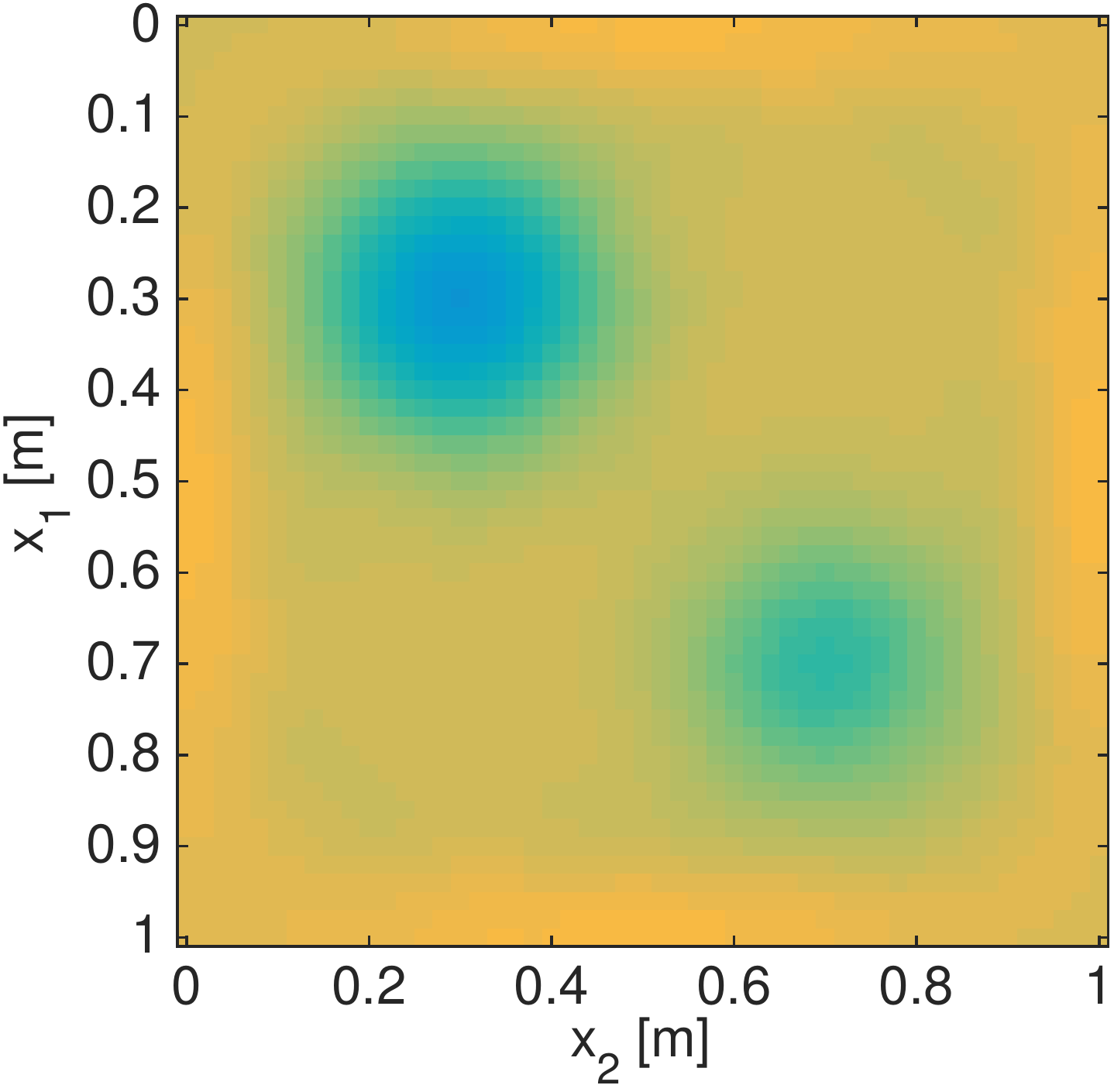}&
\includegraphics[scale=.2]{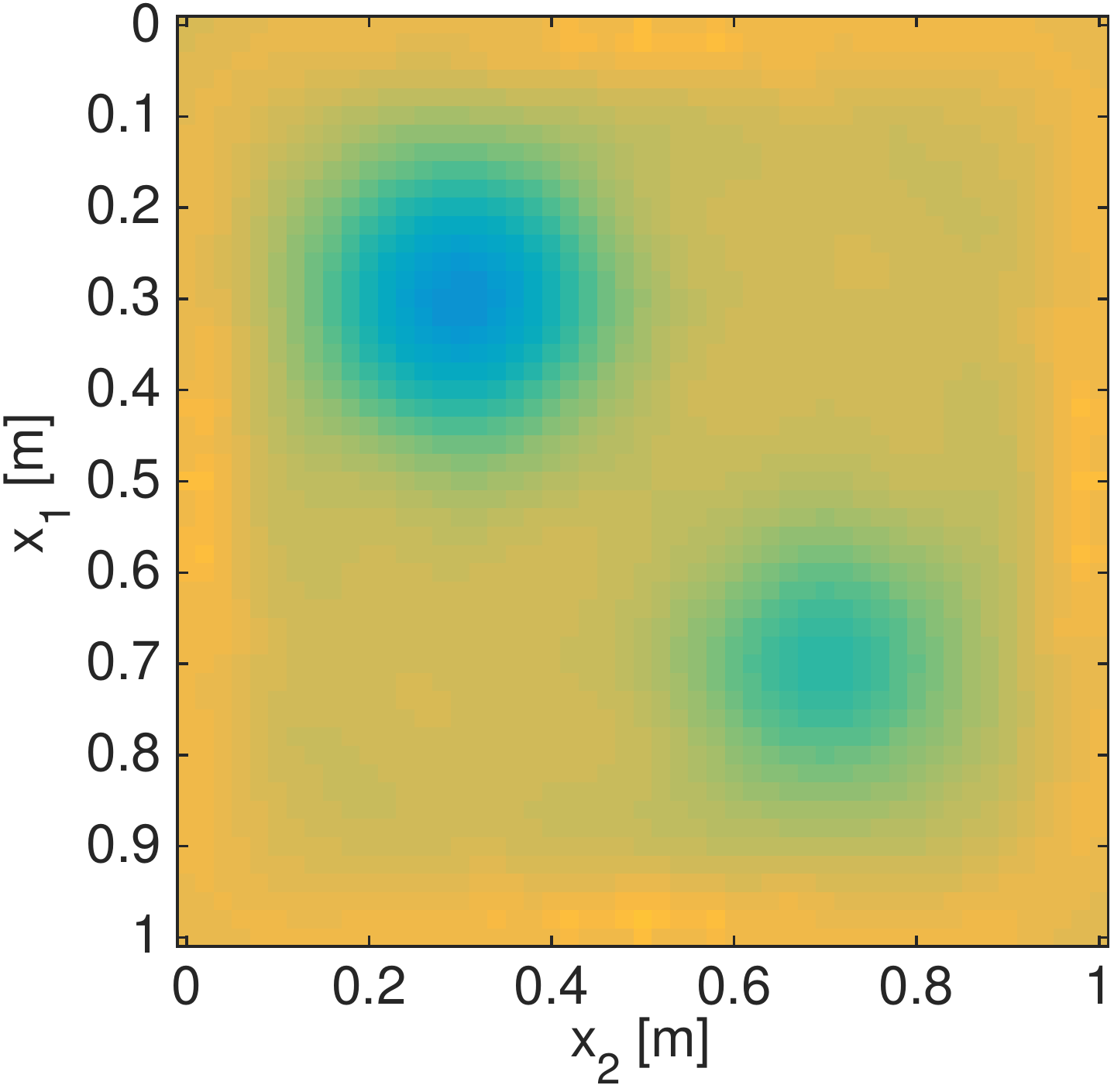}&
\includegraphics[scale=.2]{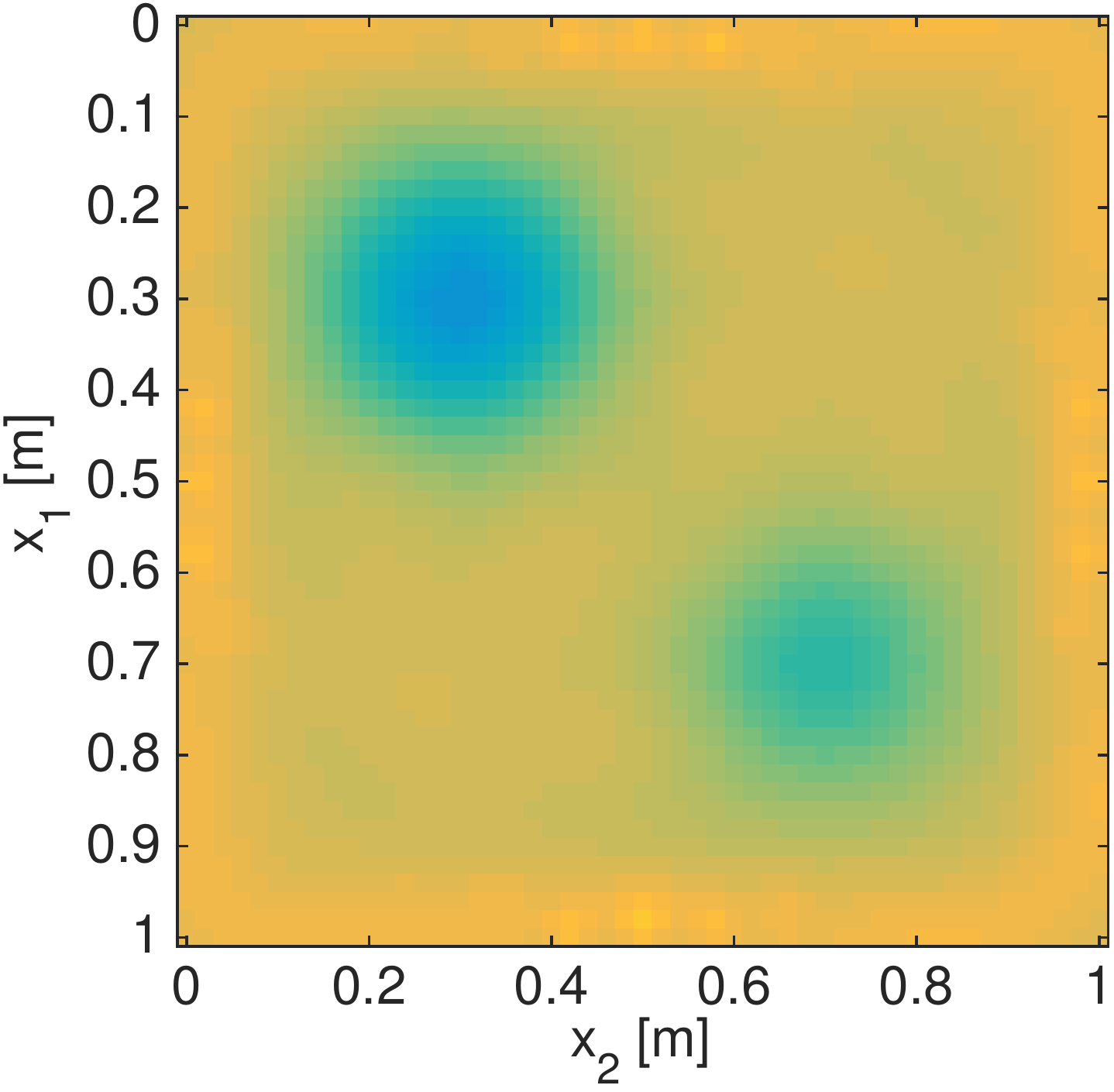}\\
{\small reduced}&{\small $\lambda=0.1$}&{\small $\lambda=1$}&{\small $\lambda=10$}\\
\end{tabular}
\caption{Convergence history, reconstruction error and reconstructions for data without noise using a GN method. Even though the penalty method does not converge to same tolerance as the reduced method in terms of the gradient of the Lagrangian, the resulting parameter estimates are almost the same.}
\label{fig:2D_exp1}
\end{figure}

\begin{figure}
\centering
\begin{tabular}{cc}
\includegraphics[scale=.3]{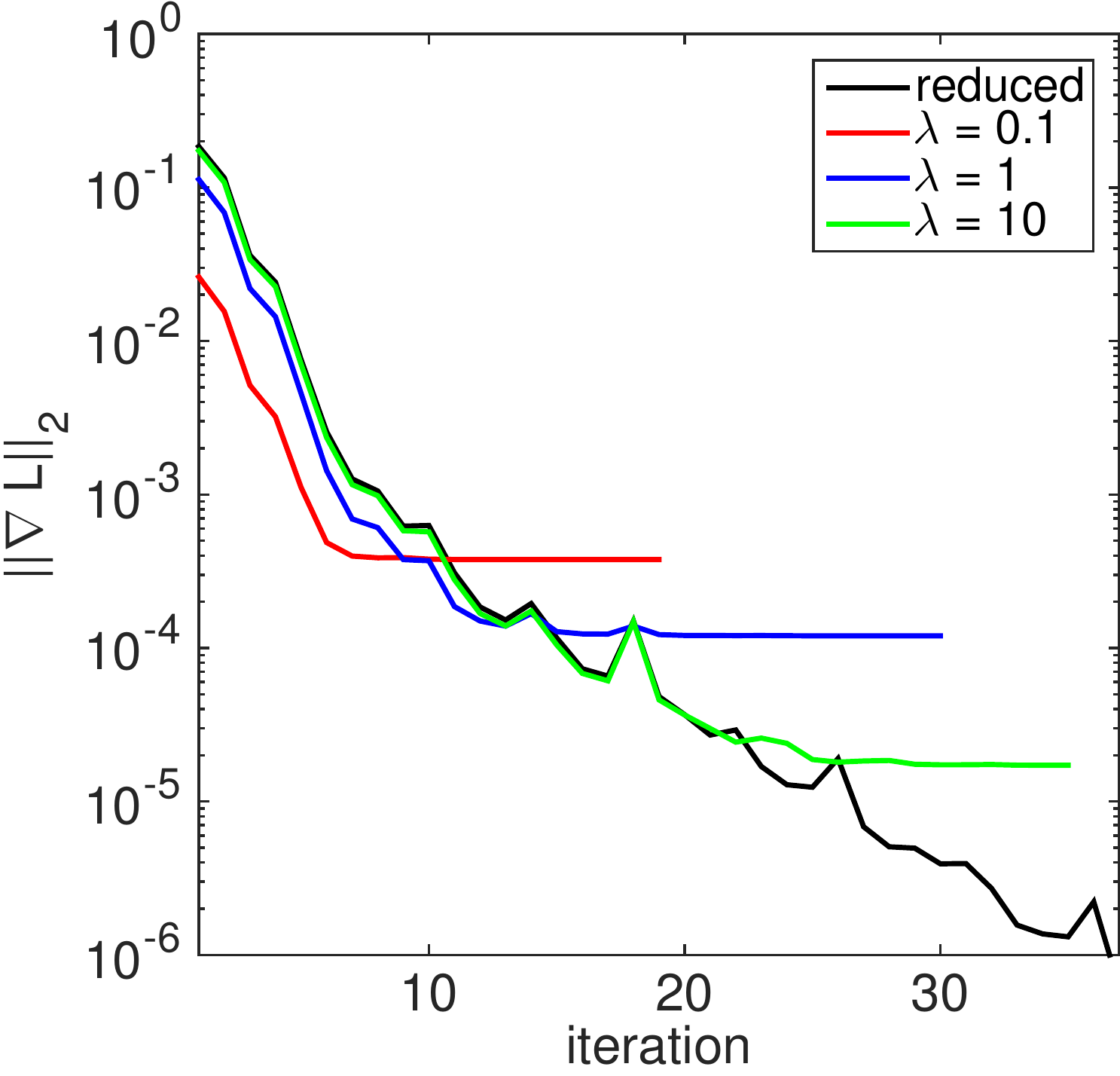}&
\includegraphics[scale=.3]{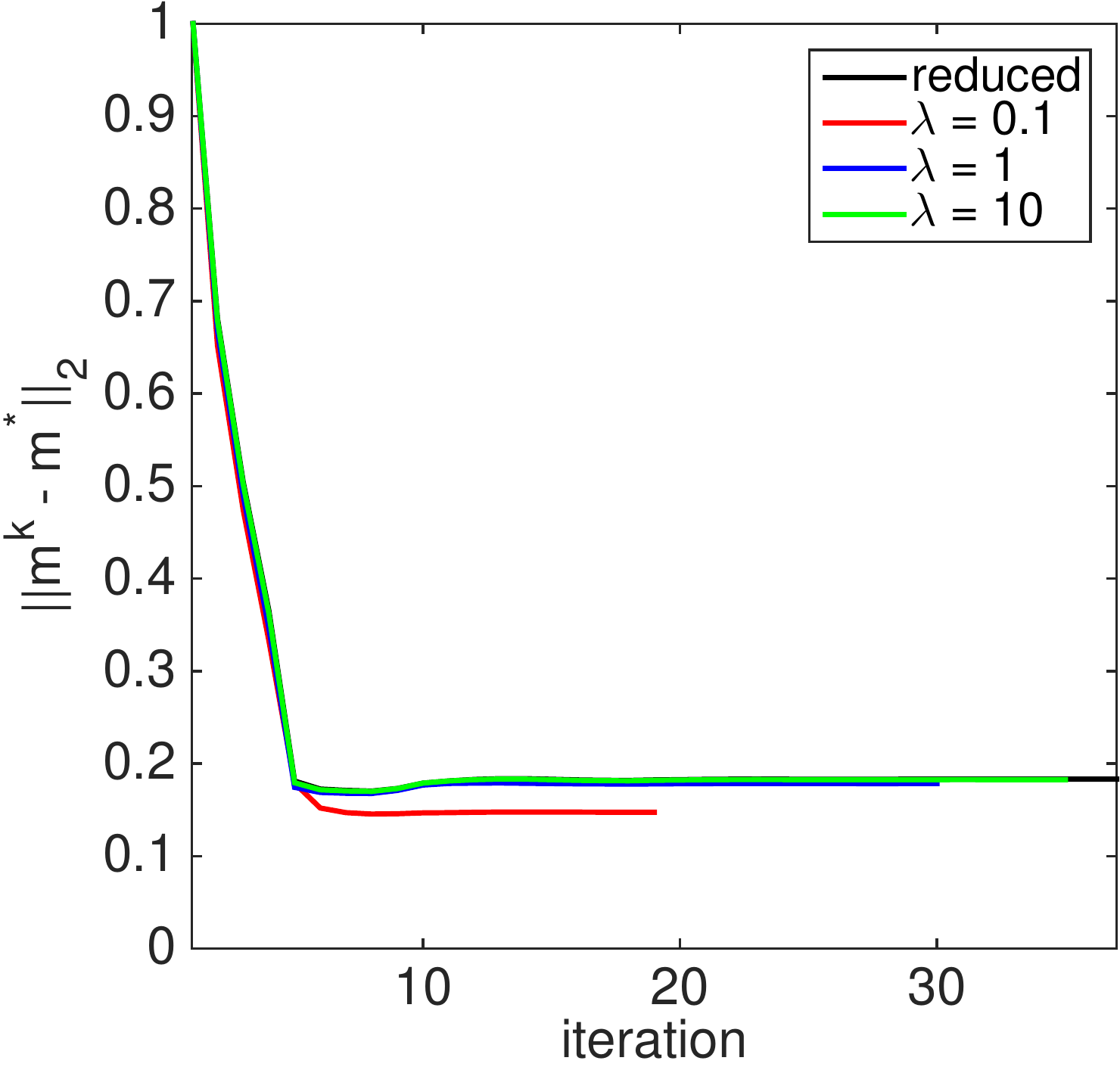}\\
\end{tabular}
\centering
\begin{tabular}{cccc}
\includegraphics[scale=.2]{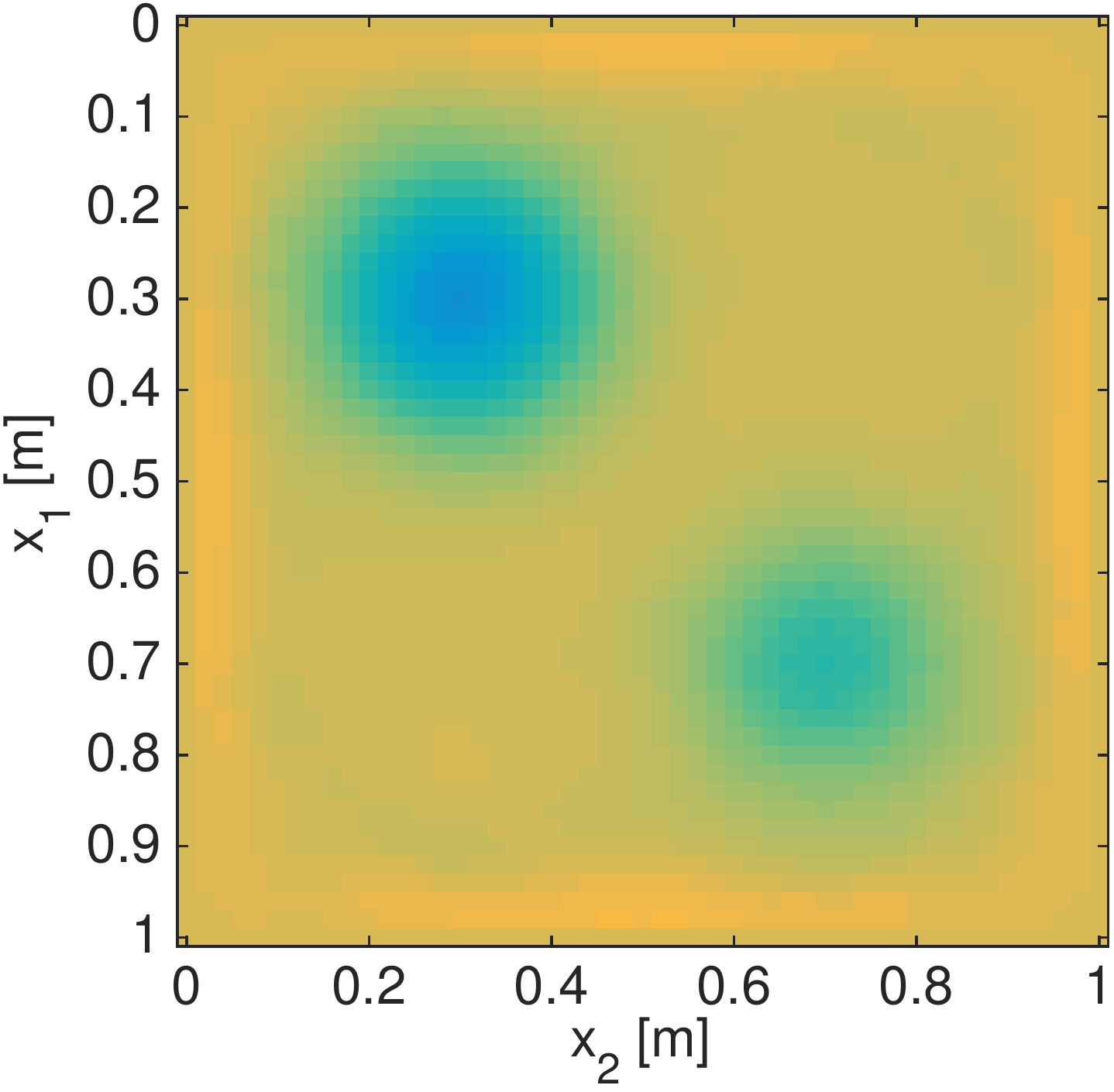}&
\includegraphics[scale=.2]{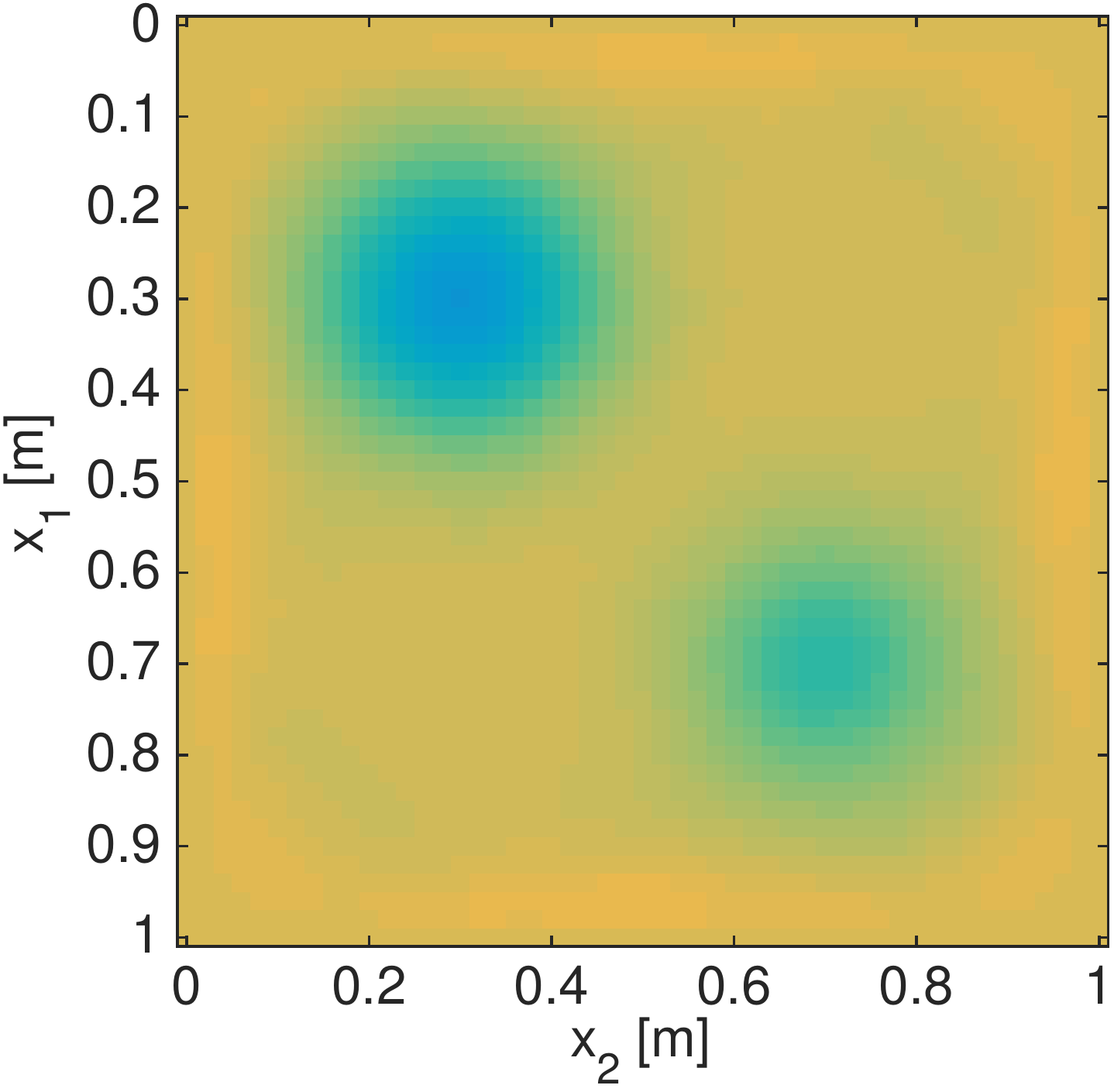}&
\includegraphics[scale=.2]{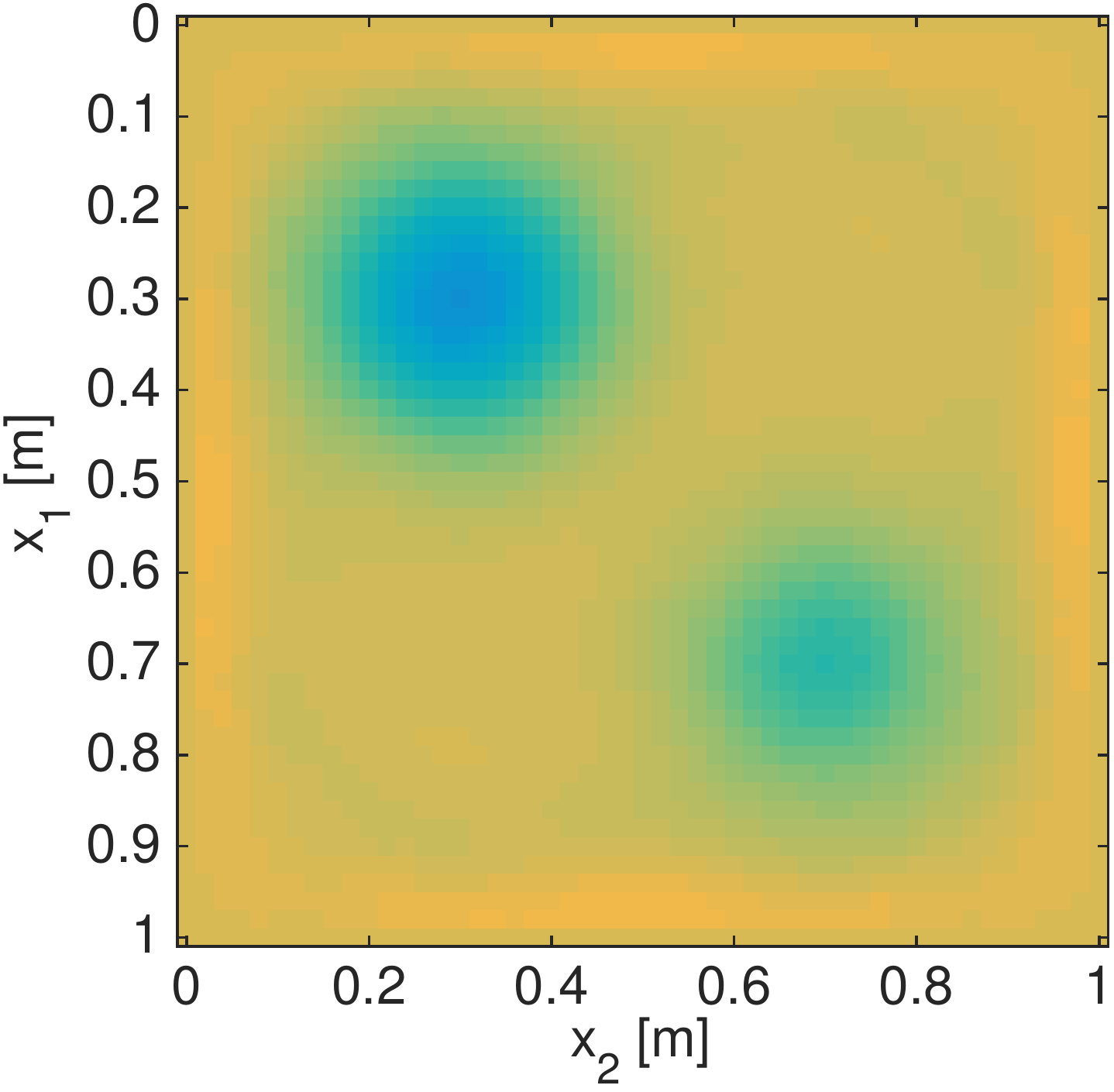}&
\includegraphics[scale=.2]{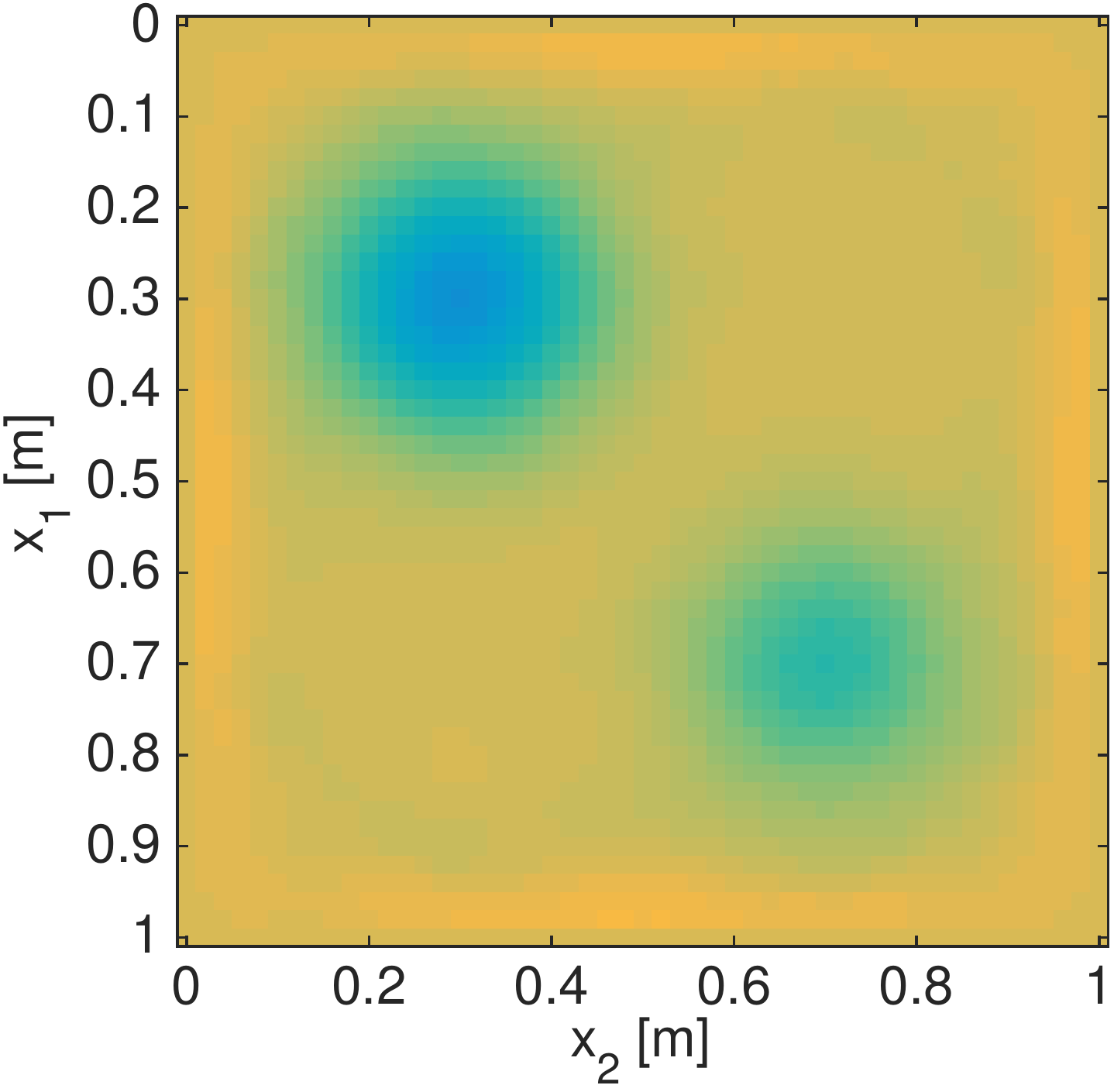}\\
{\small reduced}&{\small $\lambda=0.1$}&{\small $\lambda=1$}&{\small $\lambda=10$}\\
\end{tabular}
\caption{Convergence history, reconstruction error and reconstructions for data without noise using a QN method. Even though the penalty method does not converge to same tolerance as the reduced method in terms of the gradient of the Lagrangian, the resulting parameter estimates are almost the same.}
\label{fig:2D_exp2}
\end{figure}

\begin{figure}
\centering
\begin{tabular}{cc}
\includegraphics[scale=.3]{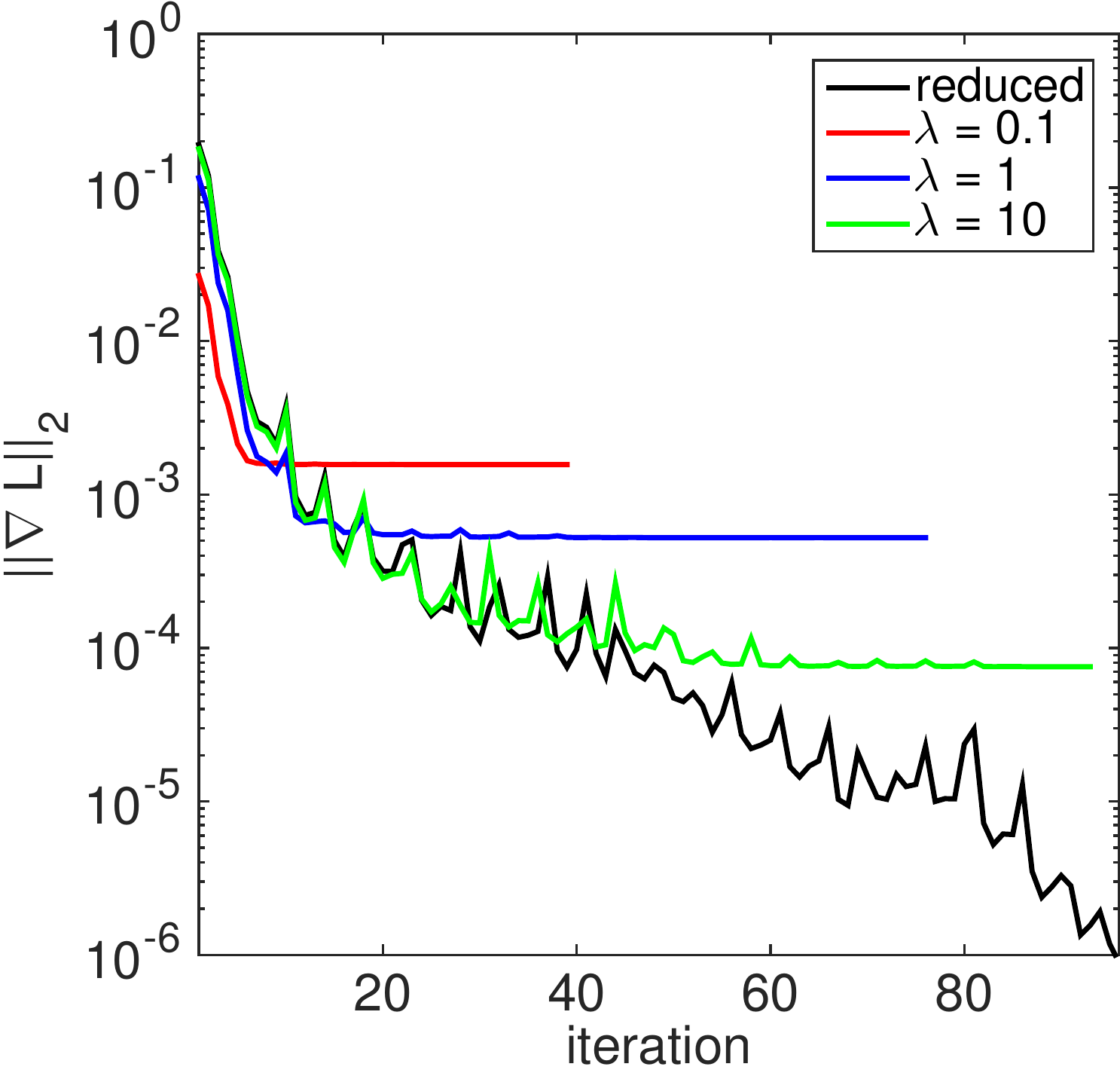}&
\includegraphics[scale=.3]{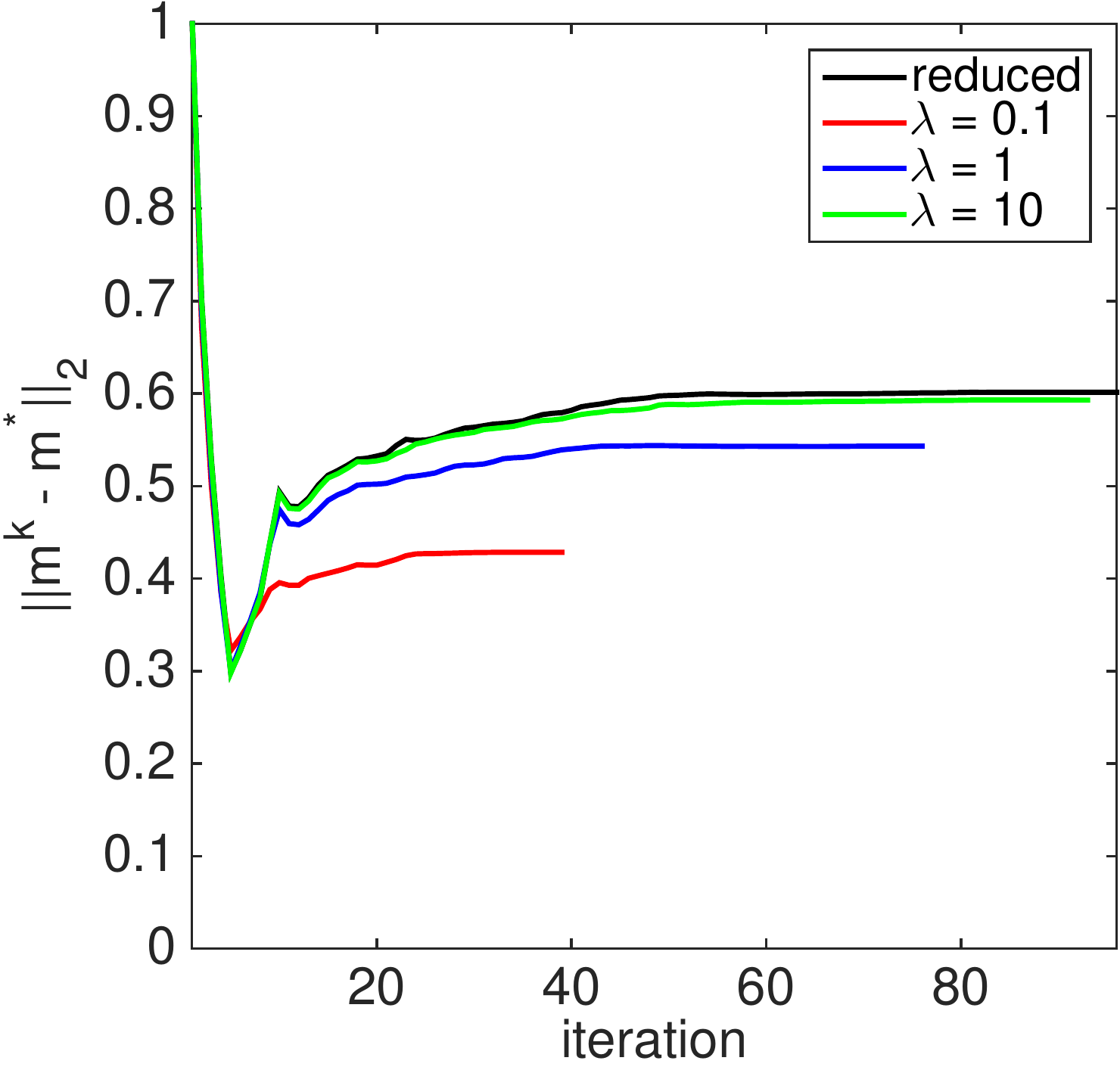}\\
\end{tabular}
\centering
\begin{tabular}{cccc}
\includegraphics[scale=.2]{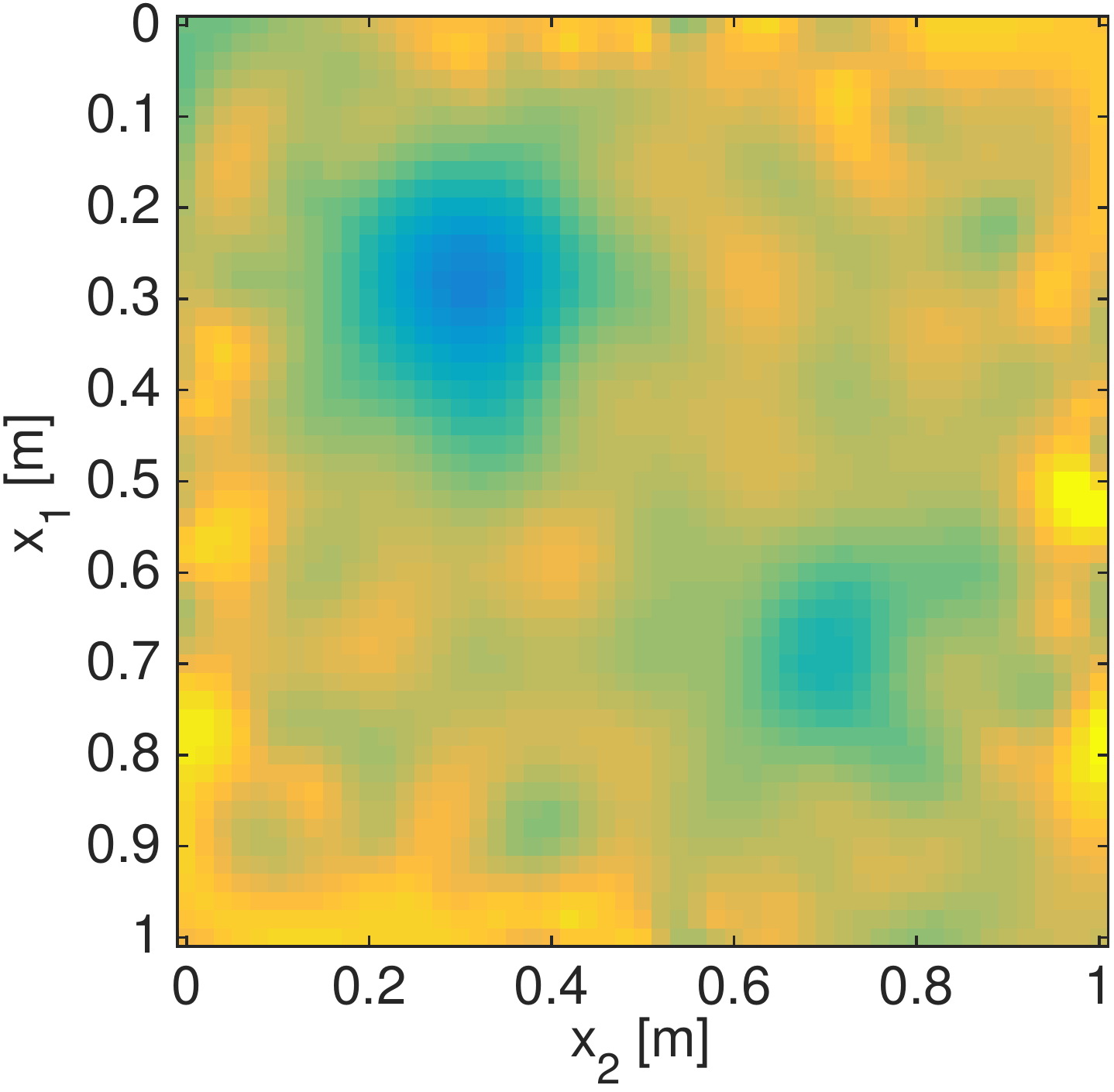}&
\includegraphics[scale=.2]{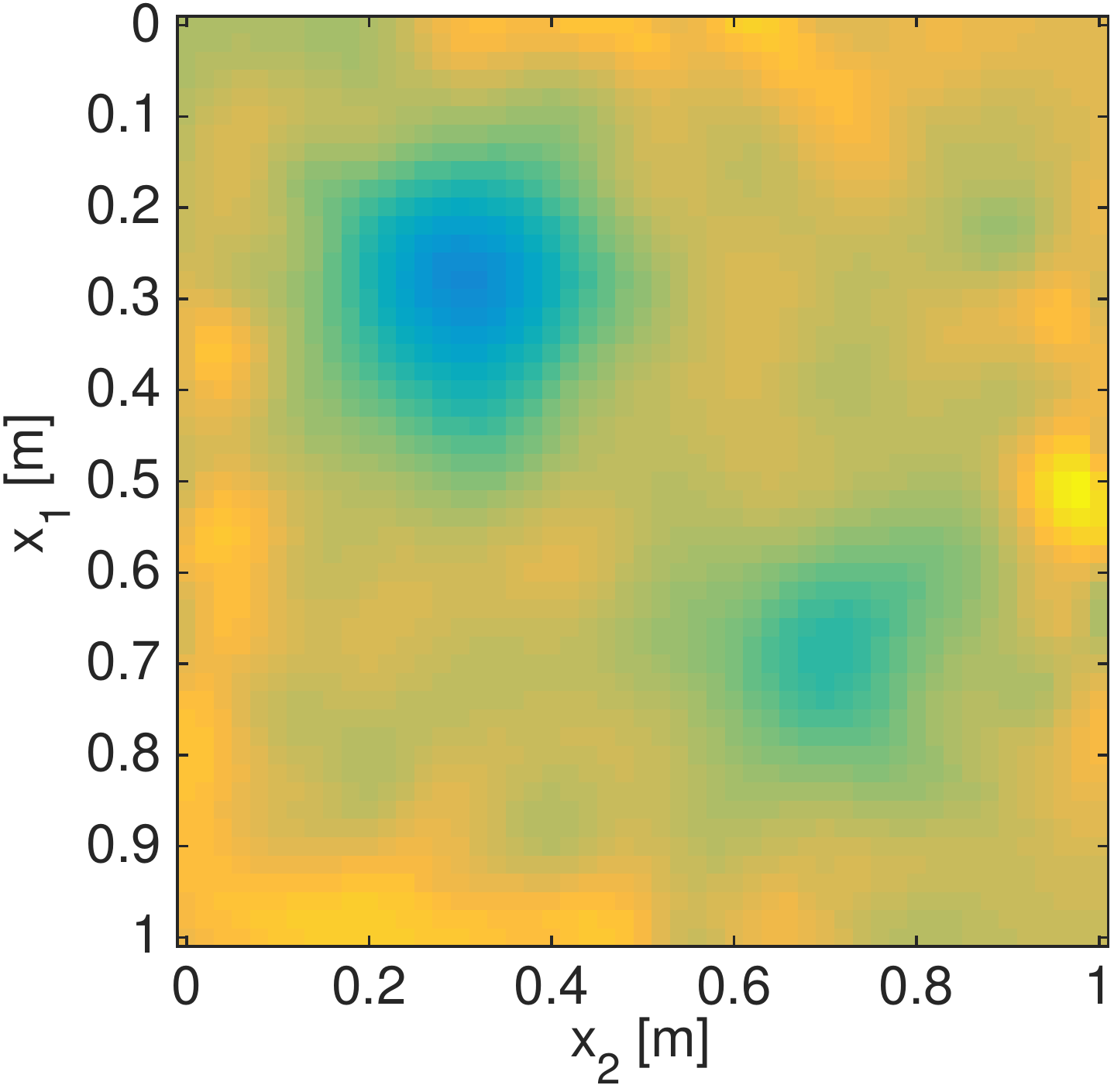}&
\includegraphics[scale=.2]{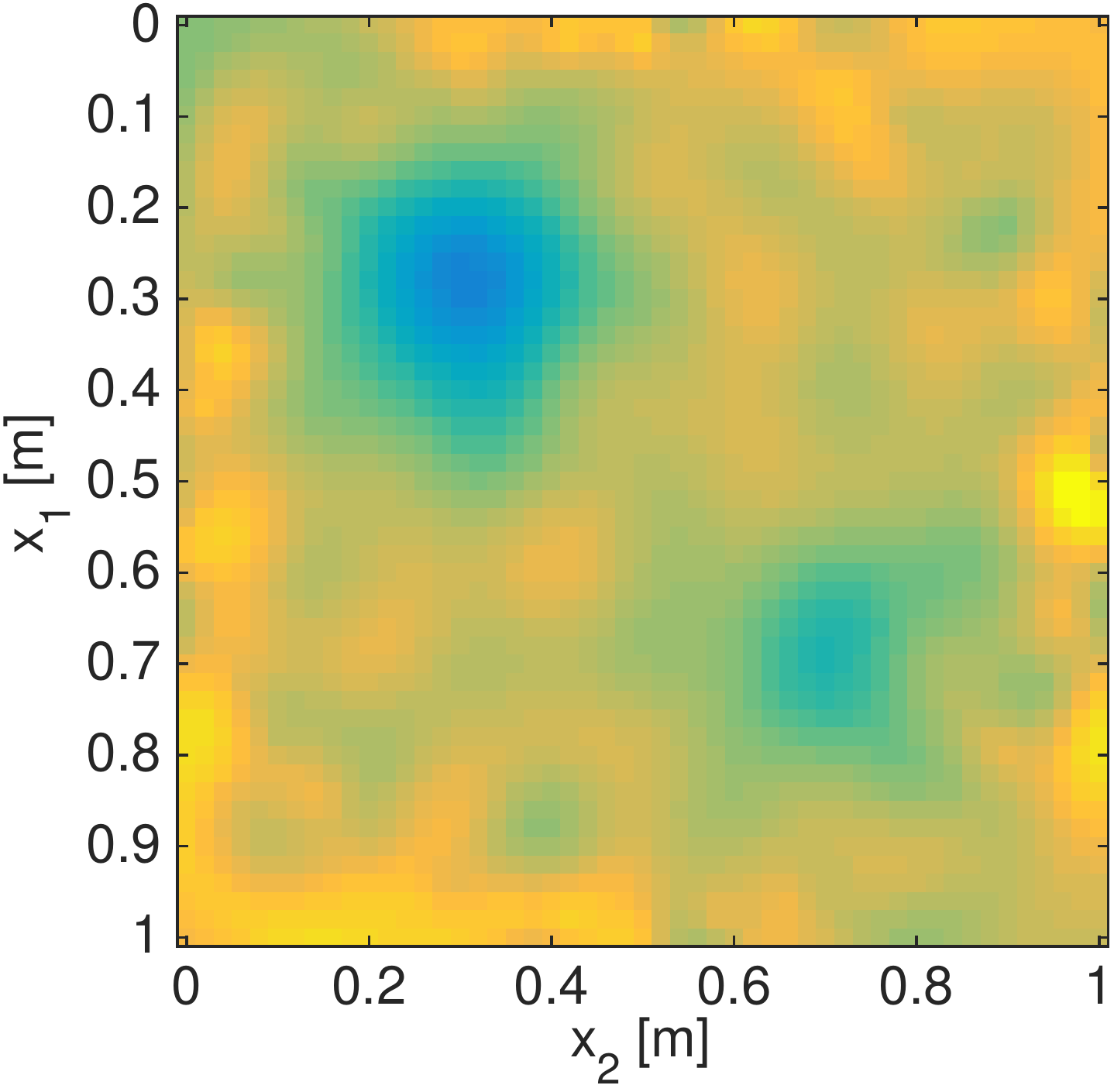}&
\includegraphics[scale=.2]{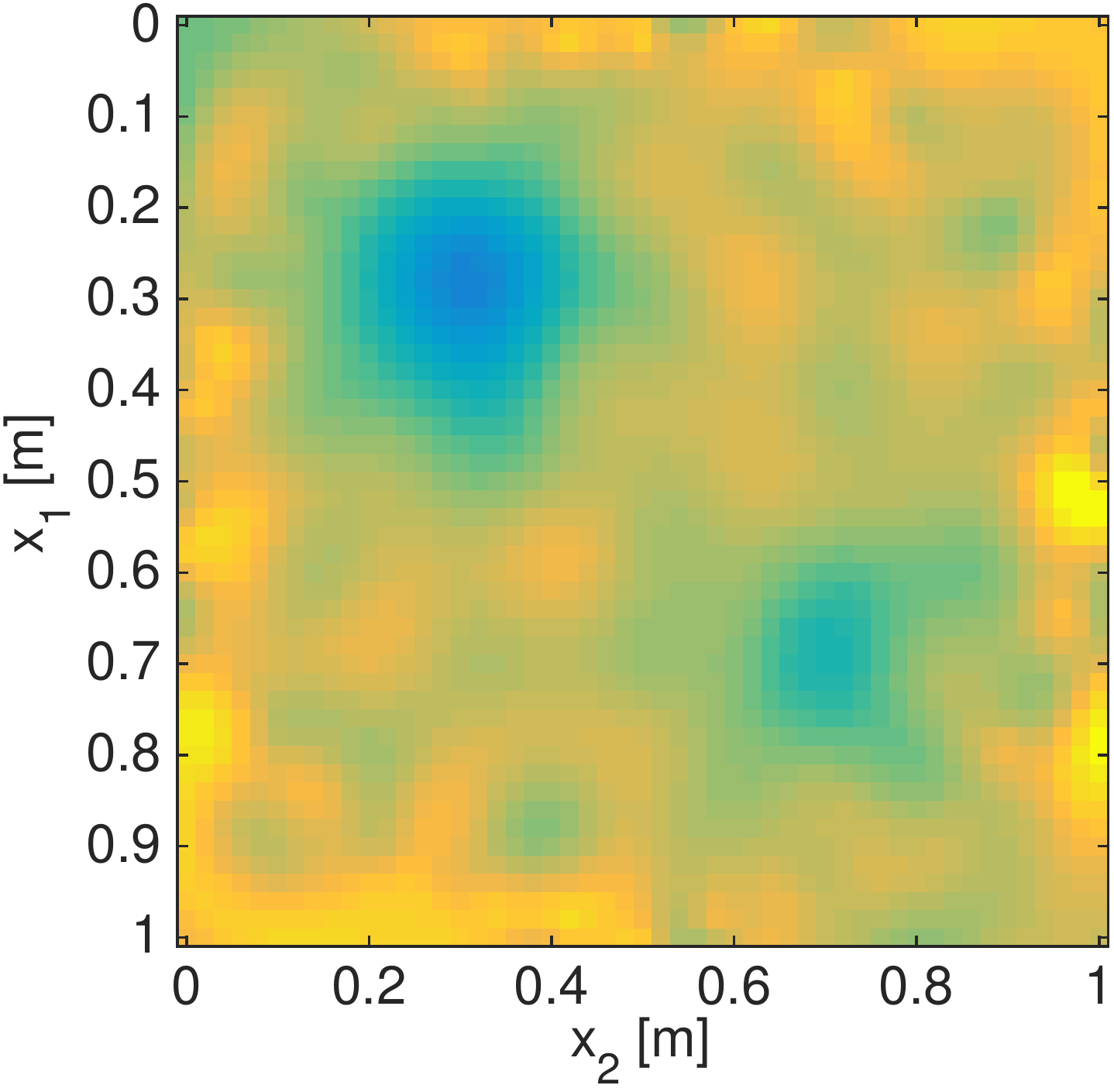}\\
{\small reduced}&{\small $\lambda=0.1$}&{\small $\lambda=1$}&{\small $\lambda=10$}\\
\end{tabular}
\caption{Convergence history, reconstruction error and reconstructions for data with 10\% Gaussian noise using a QN method. Even though the penalty method does not converge to same tolerance as the reduced method in terms of the gradient of the Lagrangian, the resulting parameter estimates are almost the same. In fact, for small $\lambda$, result is even a little better.}
\label{fig:2D_exp3}
\end{figure}

\begin{figure}
\centering
\begin{tabular}{cc}
\includegraphics[scale=.3]{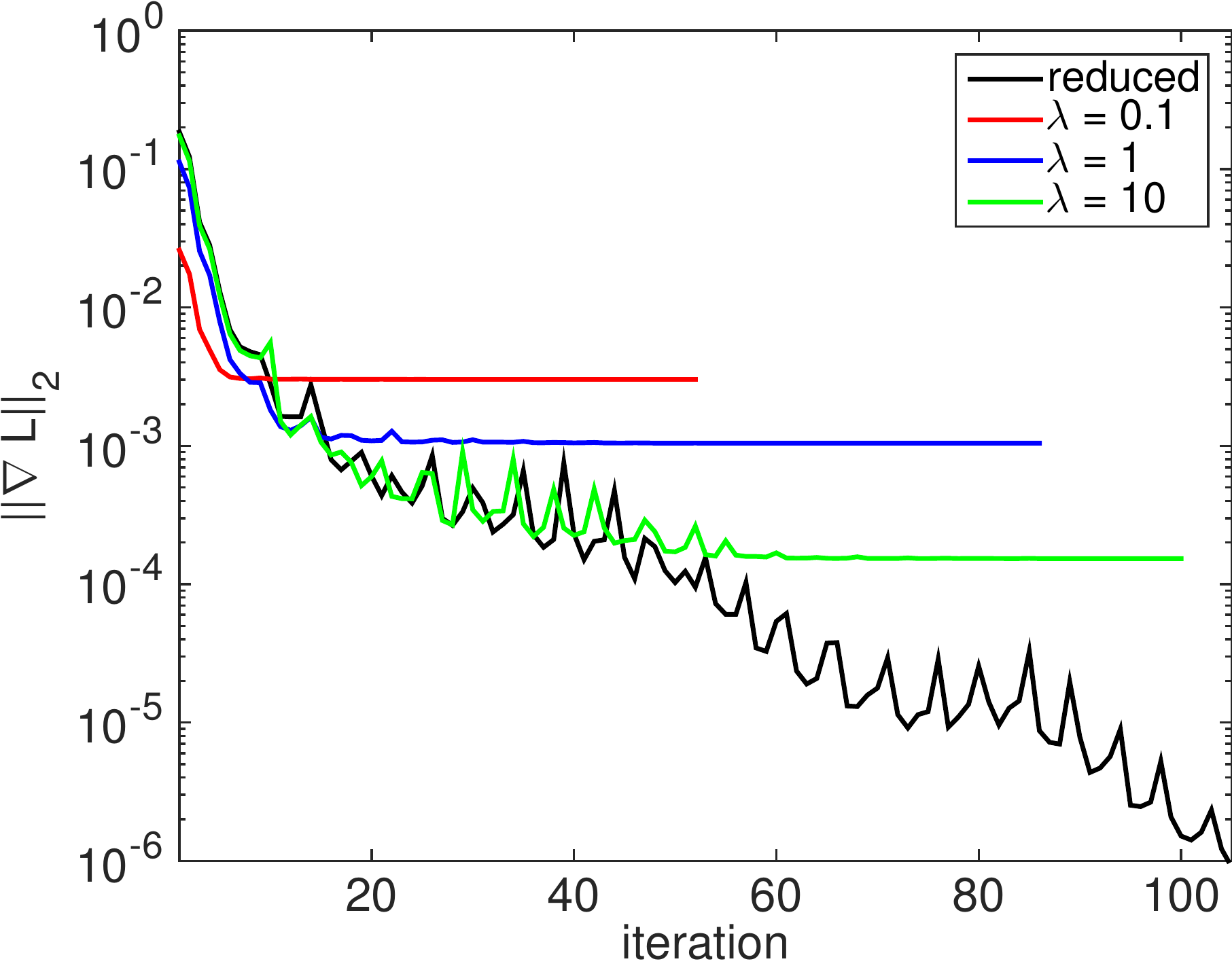}&
\includegraphics[scale=.3]{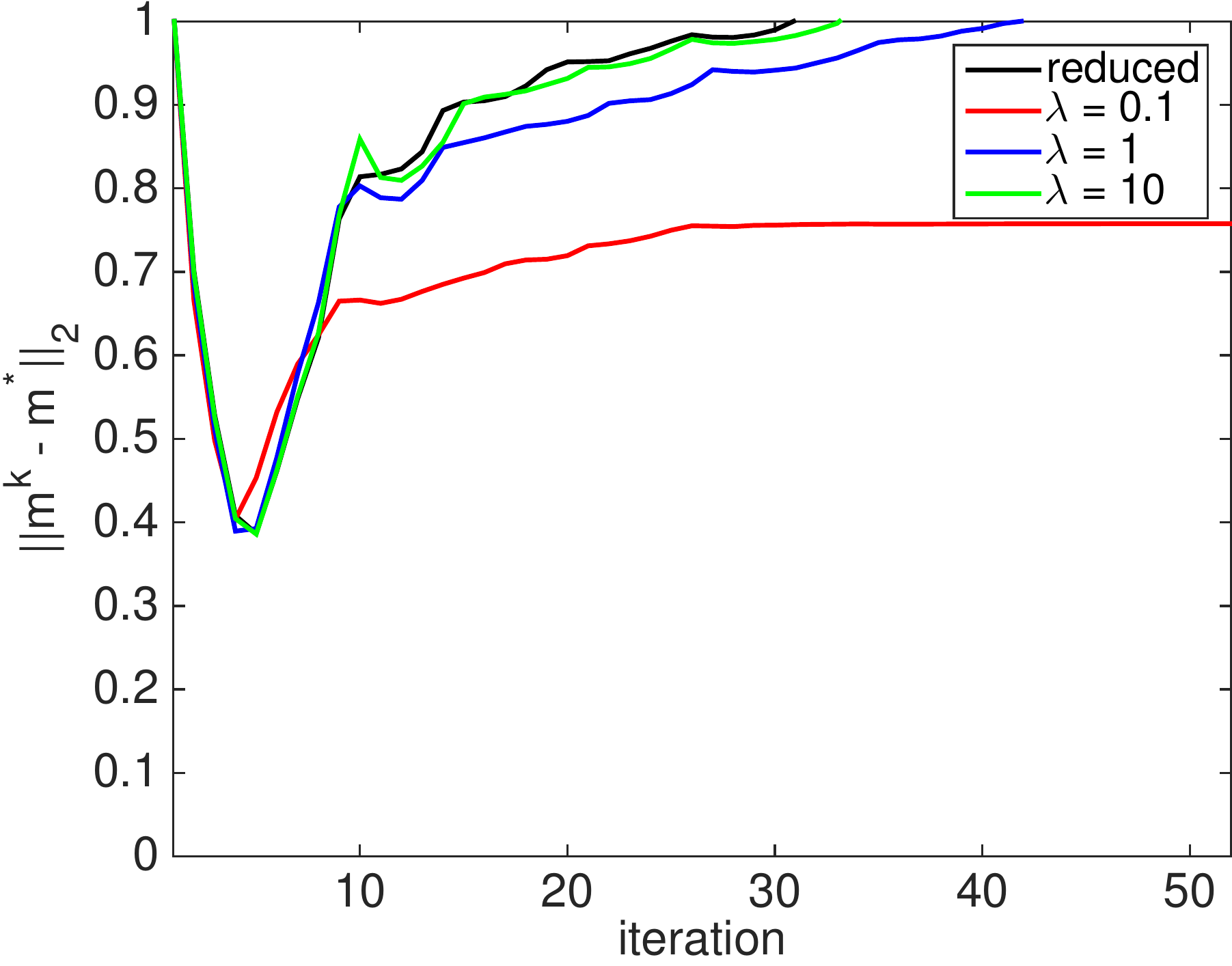}\\
\end{tabular}
\centering
\begin{tabular}{cccc}
\includegraphics[scale=.2]{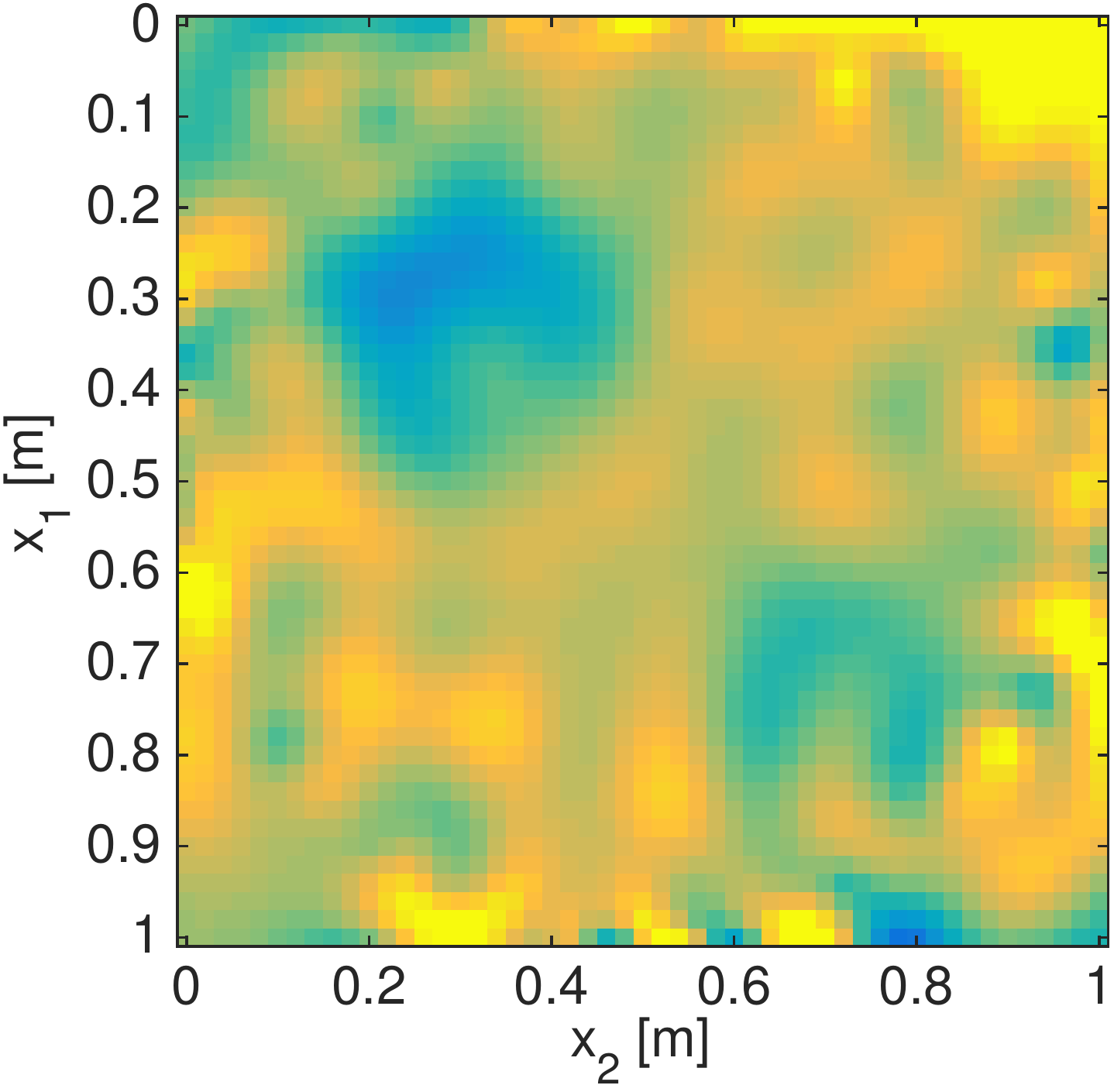}&
\includegraphics[scale=.2]{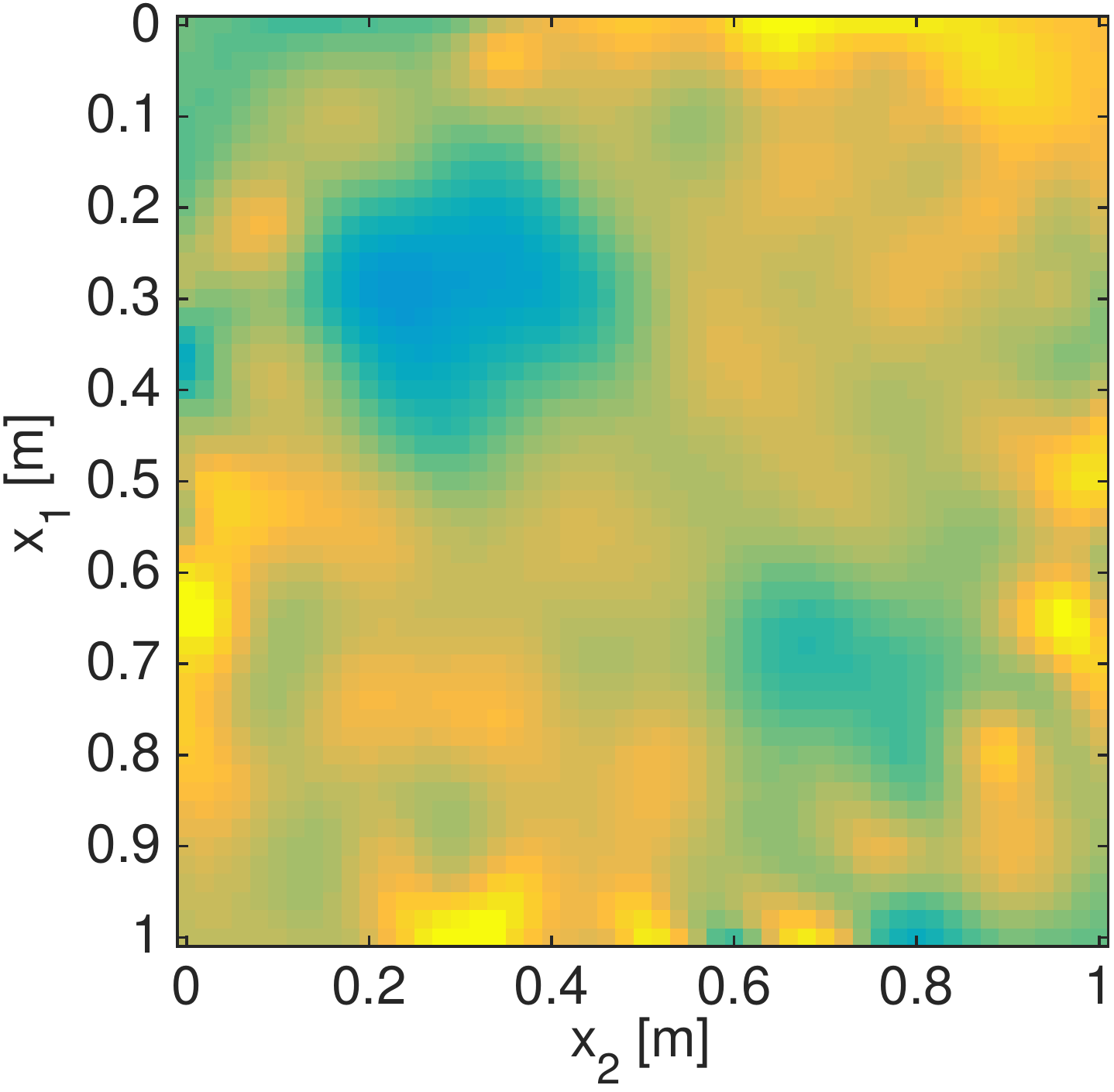}&
\includegraphics[scale=.2]{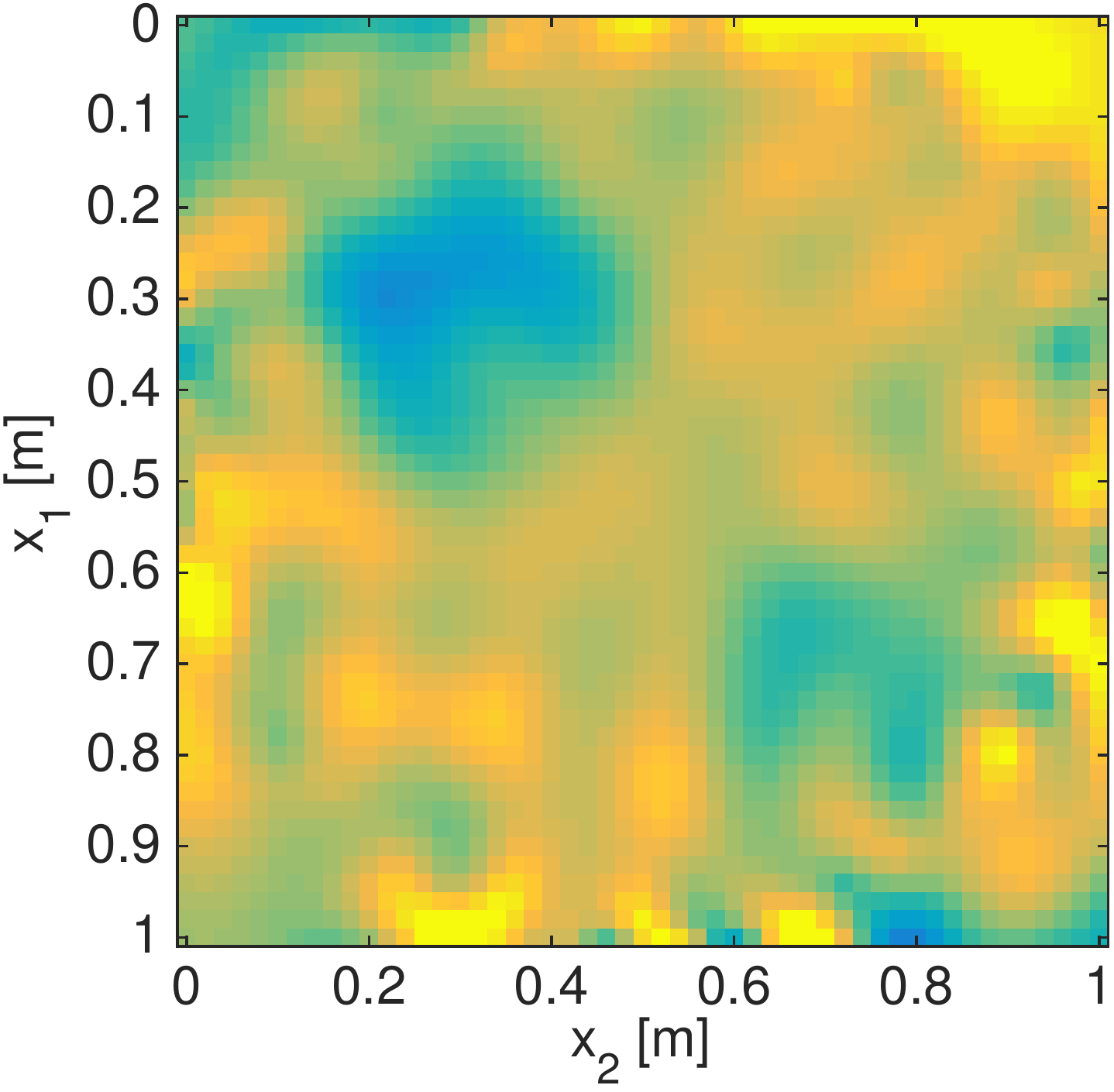}&
\includegraphics[scale=.2]{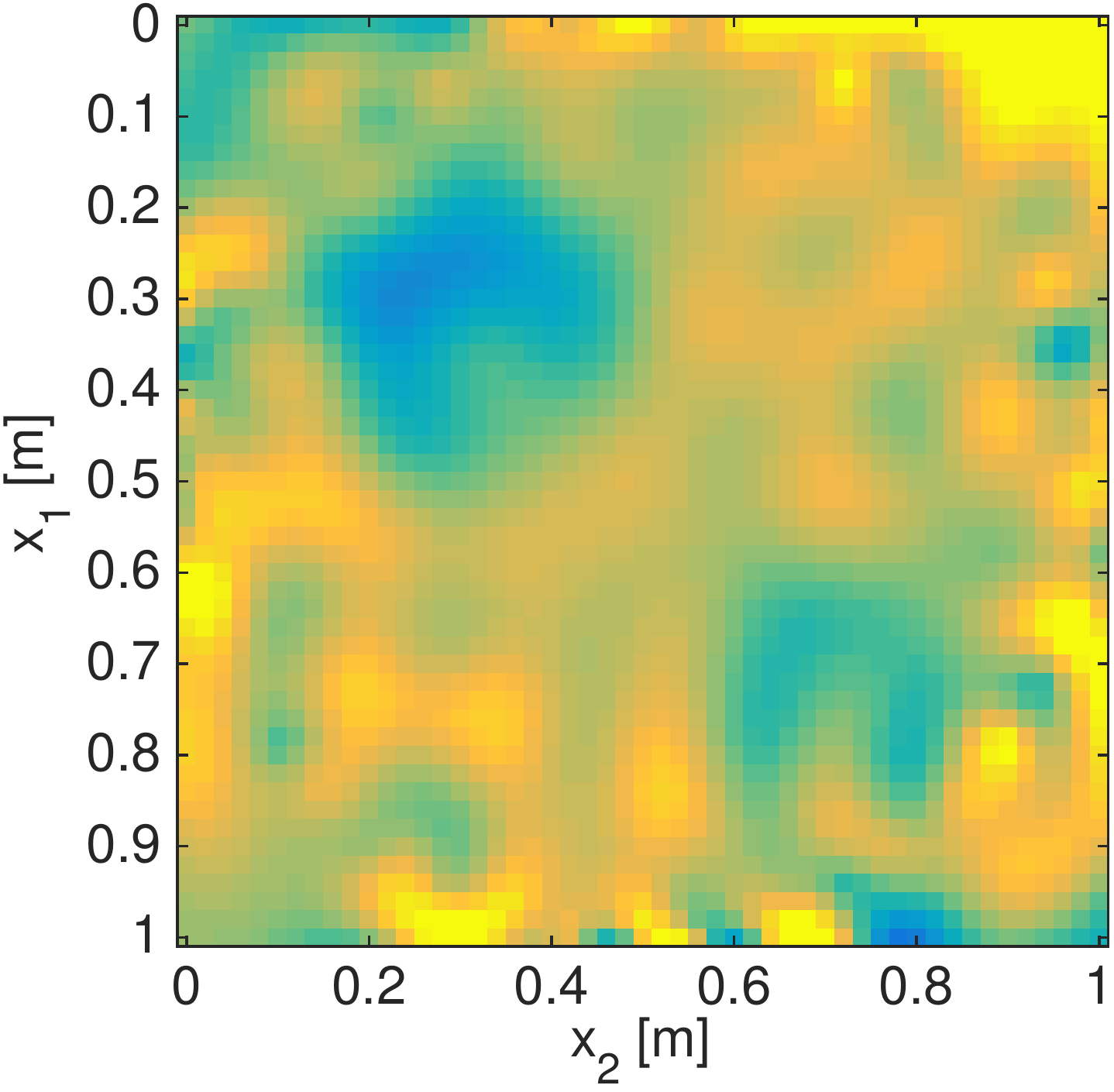}\\
{\small reduced}&{\small $\lambda=0.1$}&{\small $\lambda=1$}&{\small $\lambda=10$}\\
\end{tabular}
\caption{Convergence history, reconstruction error and reconstructions for data with 20\% Gaussian noise using a QN method. Even though the penalty method does not converge to same tolerance as the reduced method in terms of the gradient of the Lagrangian, the resulting parameter estimates are almost the same. In fact, for small $\lambda$, result is even a little better.}
\label{fig:2D_exp4}
\end{figure}

\begin{figure}
\centering
\includegraphics[scale=.7]{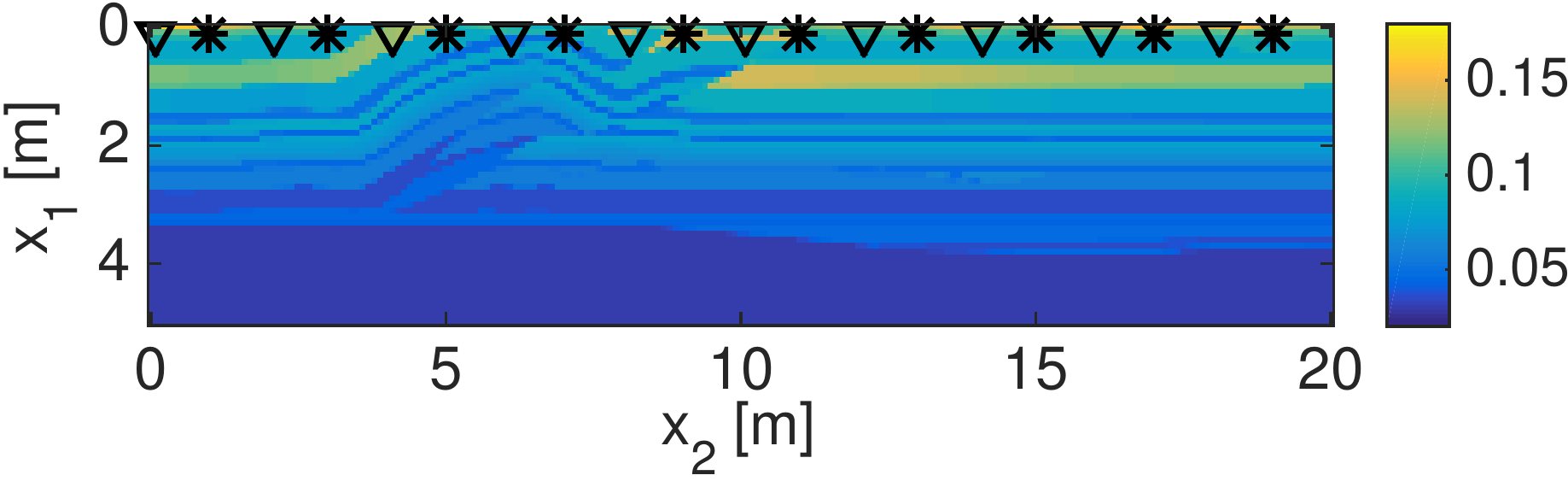}\\
\begin{tabular}{cc}
\includegraphics[scale=.3]{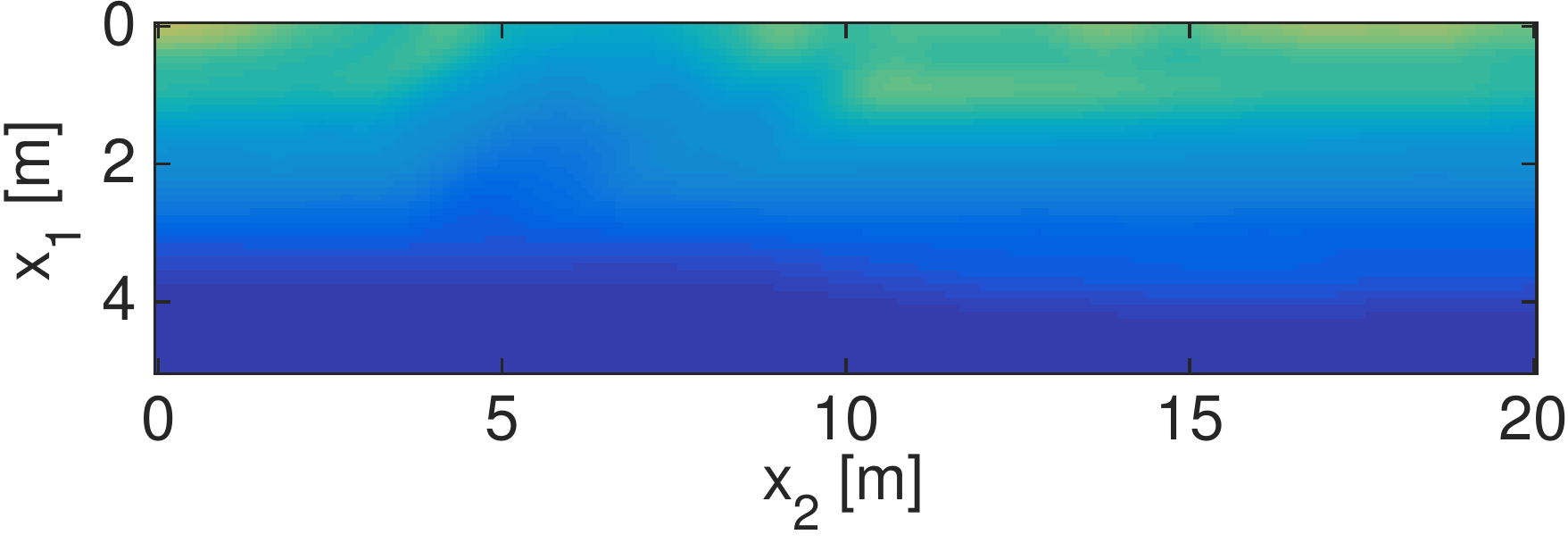}&
\includegraphics[scale=.3]{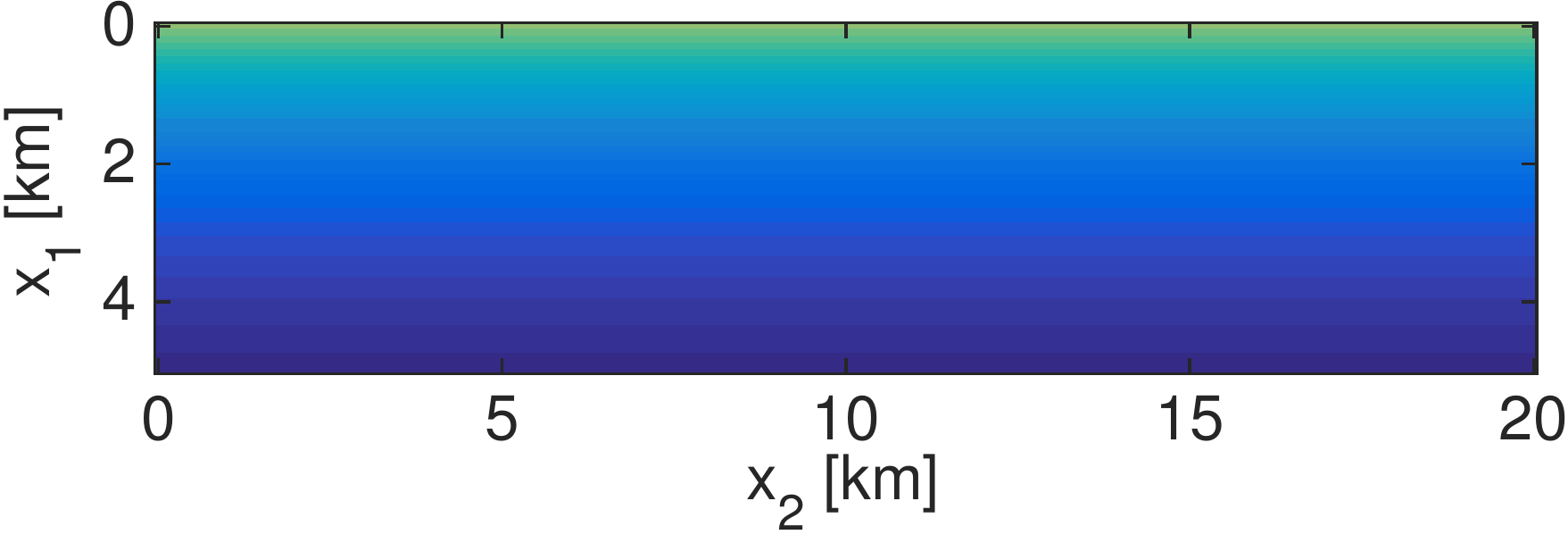}\\
\includegraphics[scale=.3]{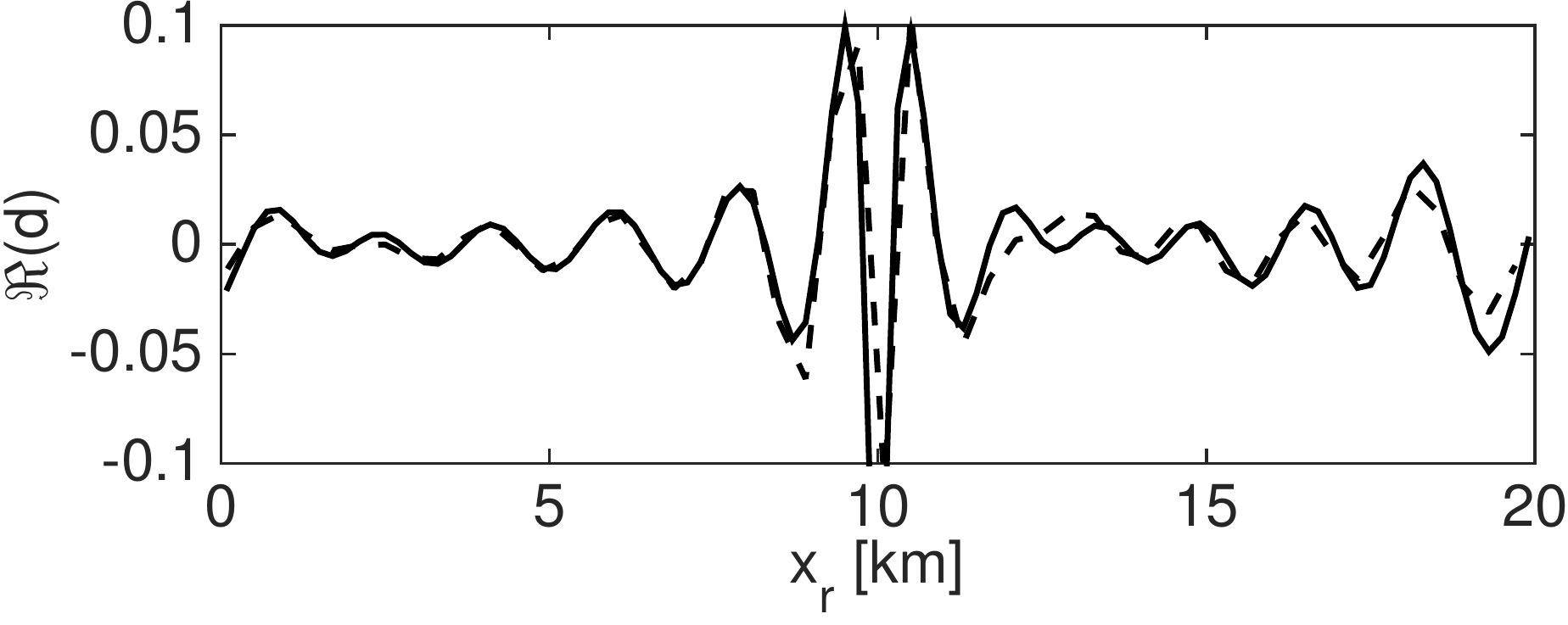}&
\includegraphics[scale=.3]{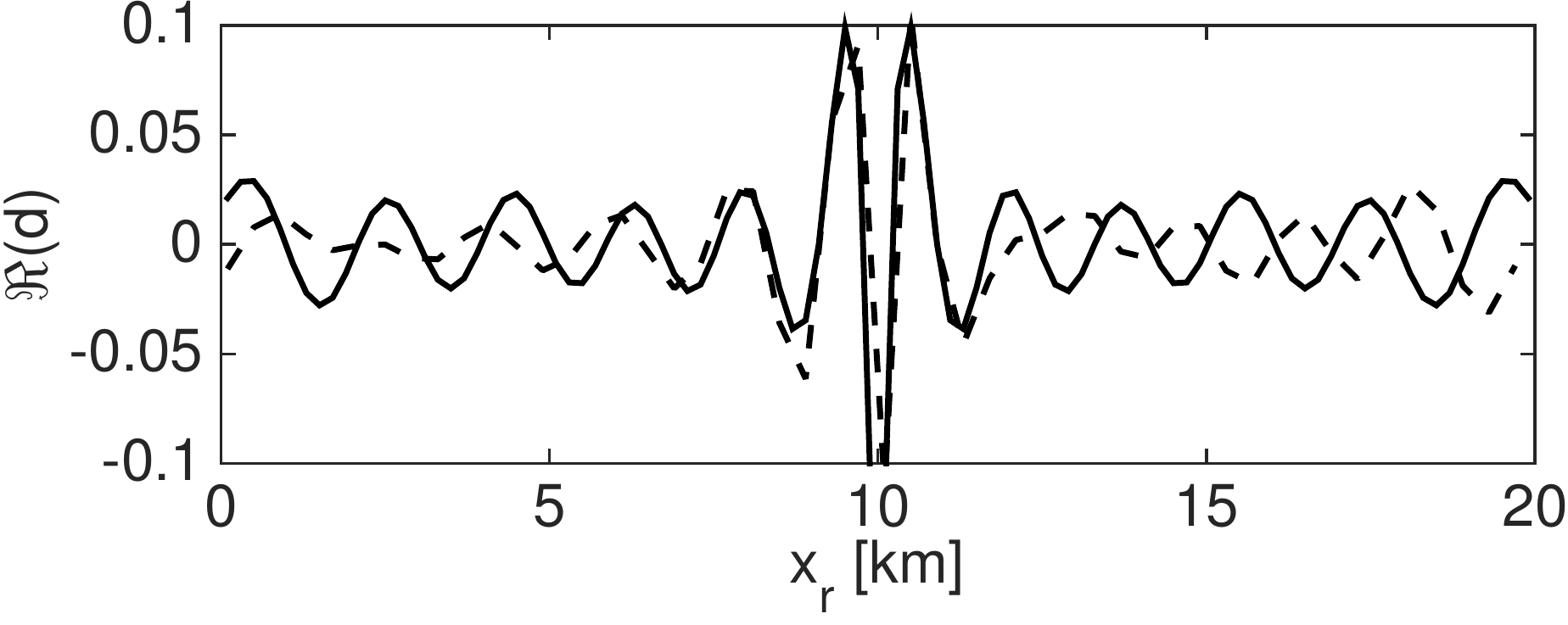}\\
\end{tabular}
\caption{Ground truth ($s^2/km^2$) (top) with locations of the sources ($*$) and receivers ($\bigtriangledown$) and initial iterates I (middle, left) and II (bottom, right). The bottom row shows the data for a source in the centre for the ground truth (dashed line) as well as the data for the two initial iterates. The first initial iterate produces data that differs only slightly from the observed data and inversion is considered to be easy. The second initial iterate produces data that is shifted significantly with respect to the observed data and inversion is considered to be difficult.}
\label{fig:overthrust_model}
\end{figure}

\begin{figure}
\centering
\begin{tabular}{cc}
\includegraphics[scale=.3]{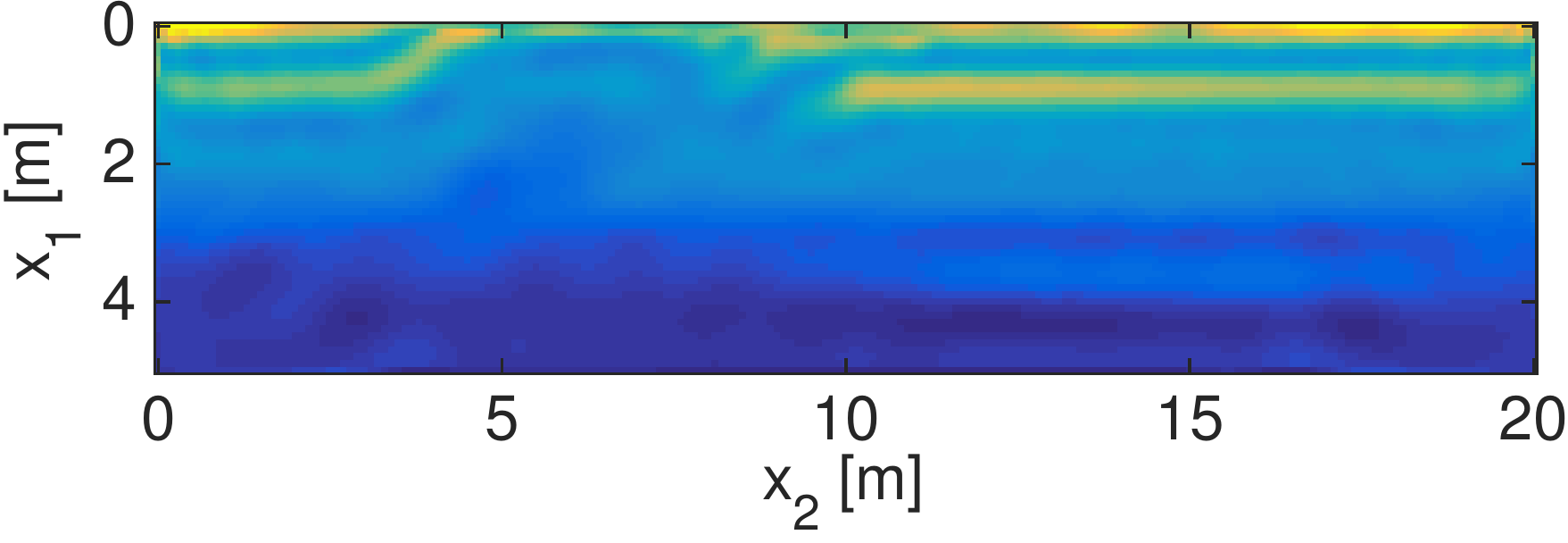}&
\includegraphics[scale=.3]{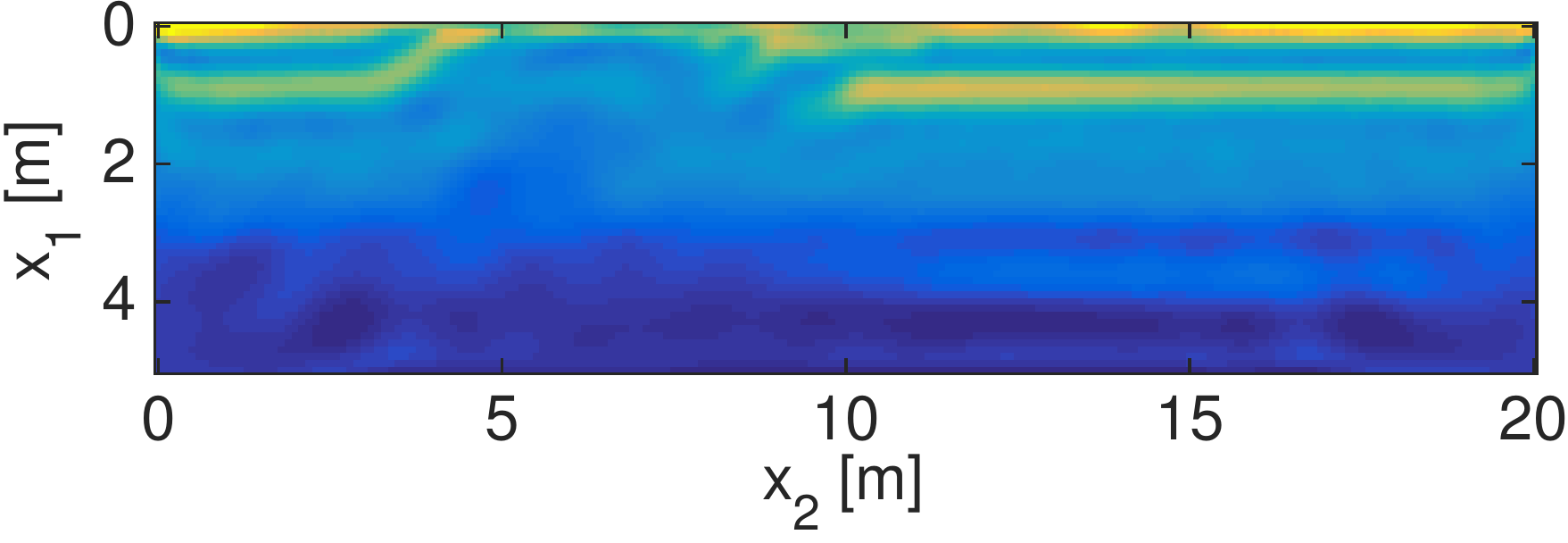}\\
\includegraphics[scale=.3]{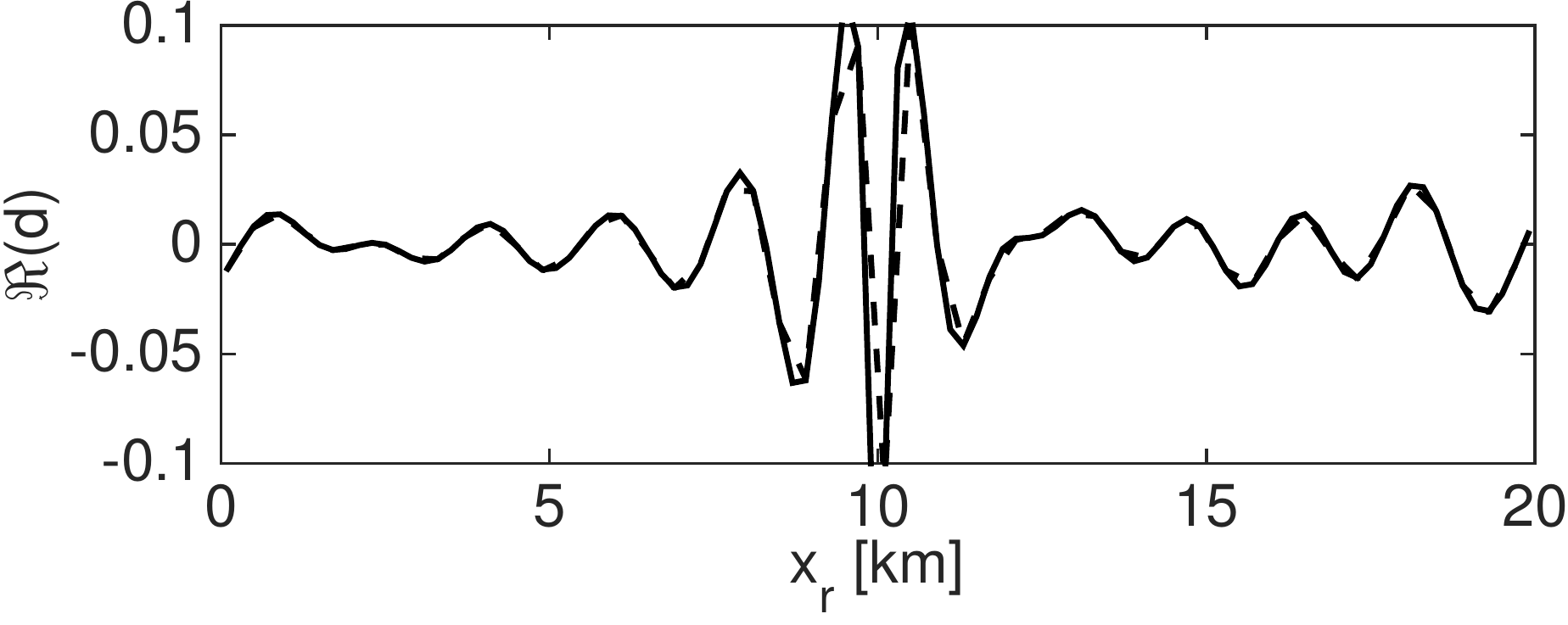}&
\includegraphics[scale=.3]{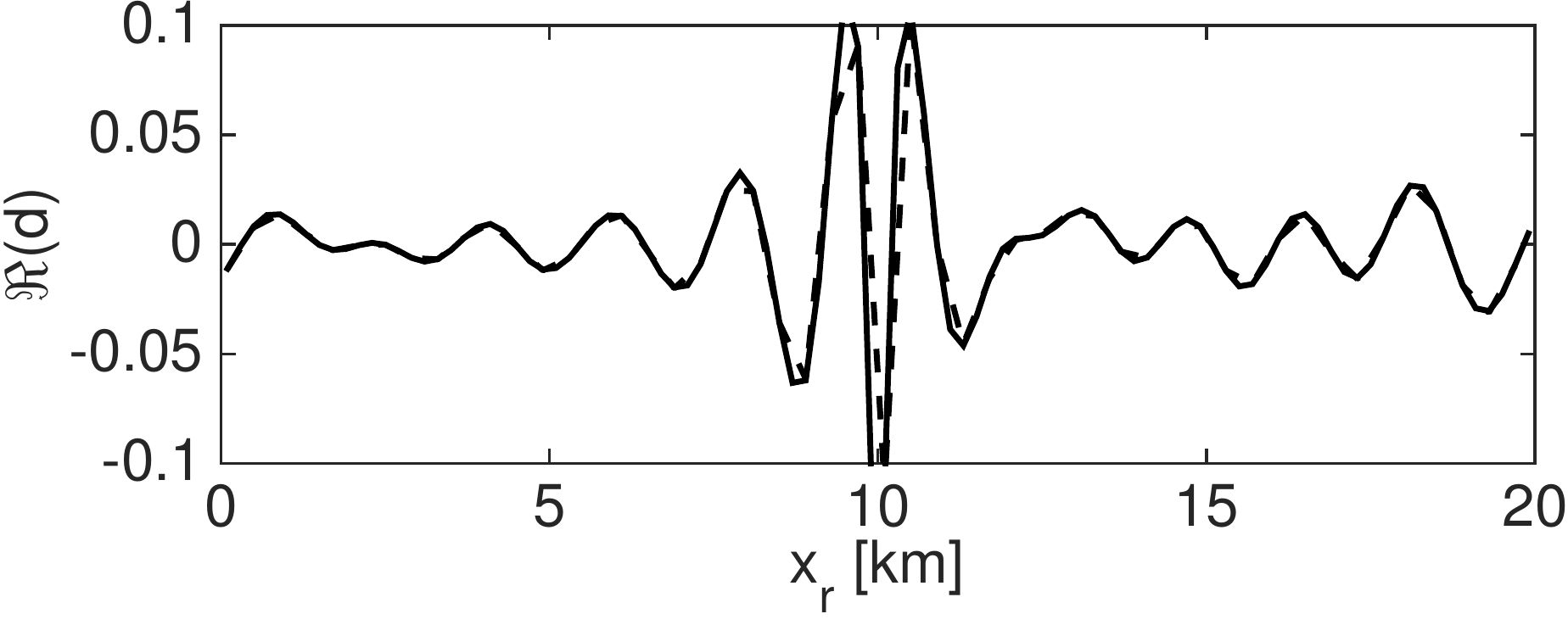}\\
{\small reduced}&
{\small $\lambda=0.1$}\\
\includegraphics[scale=.3]{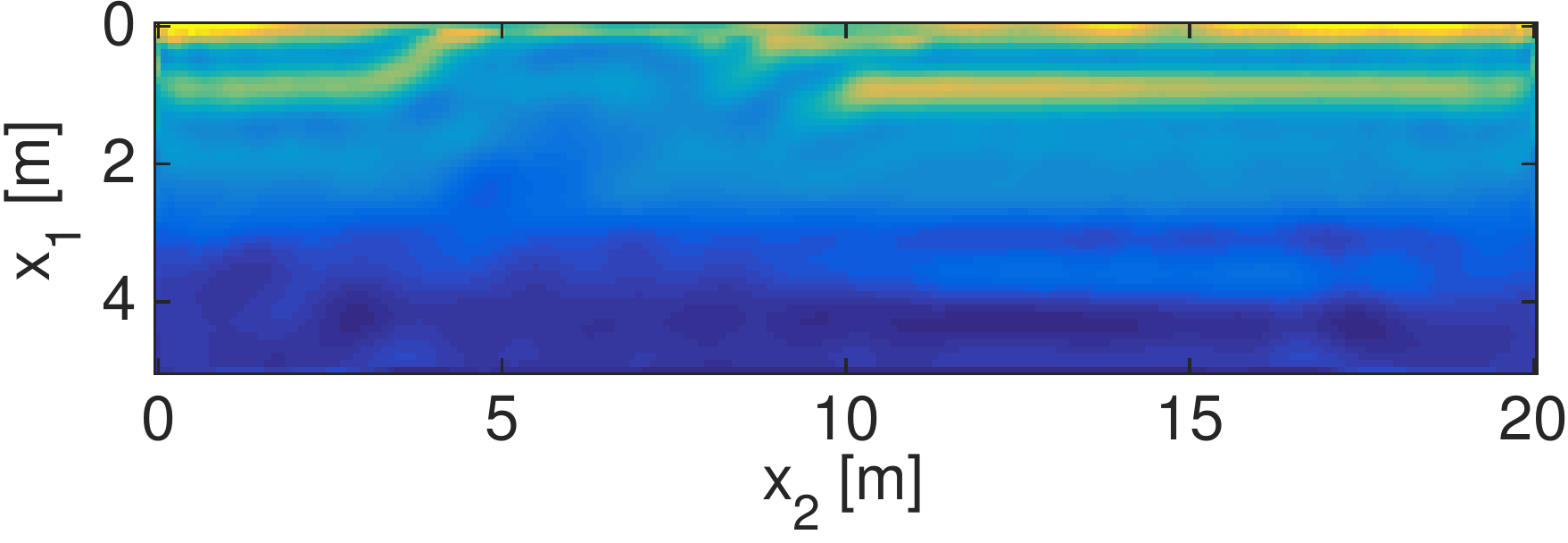}&
\includegraphics[scale=.3]{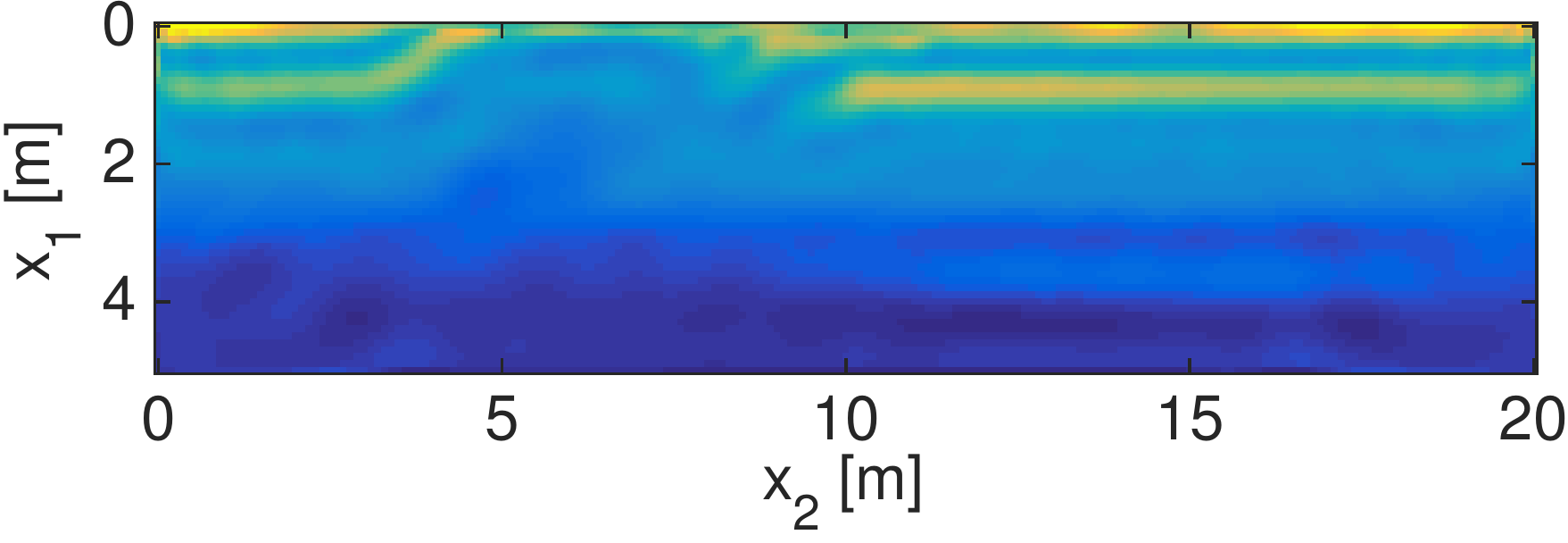}\\
\includegraphics[scale=.3]{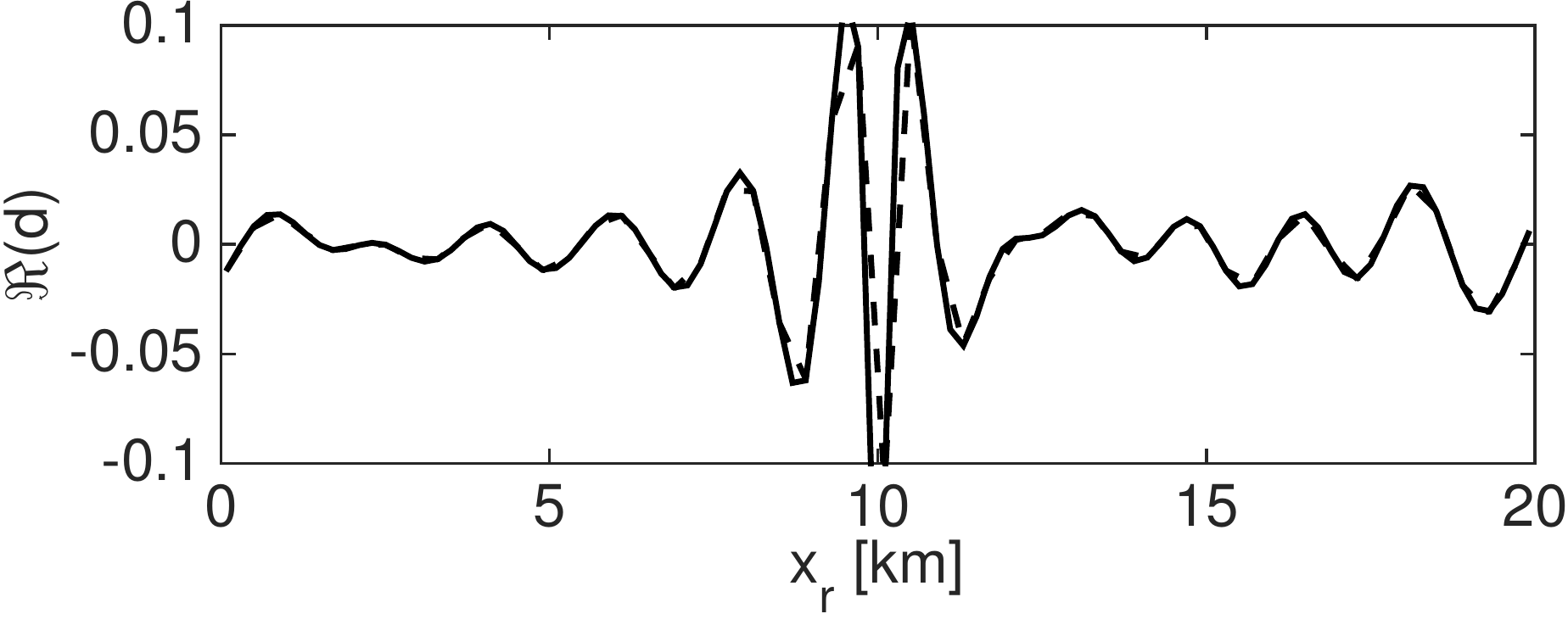}&
\includegraphics[scale=.3]{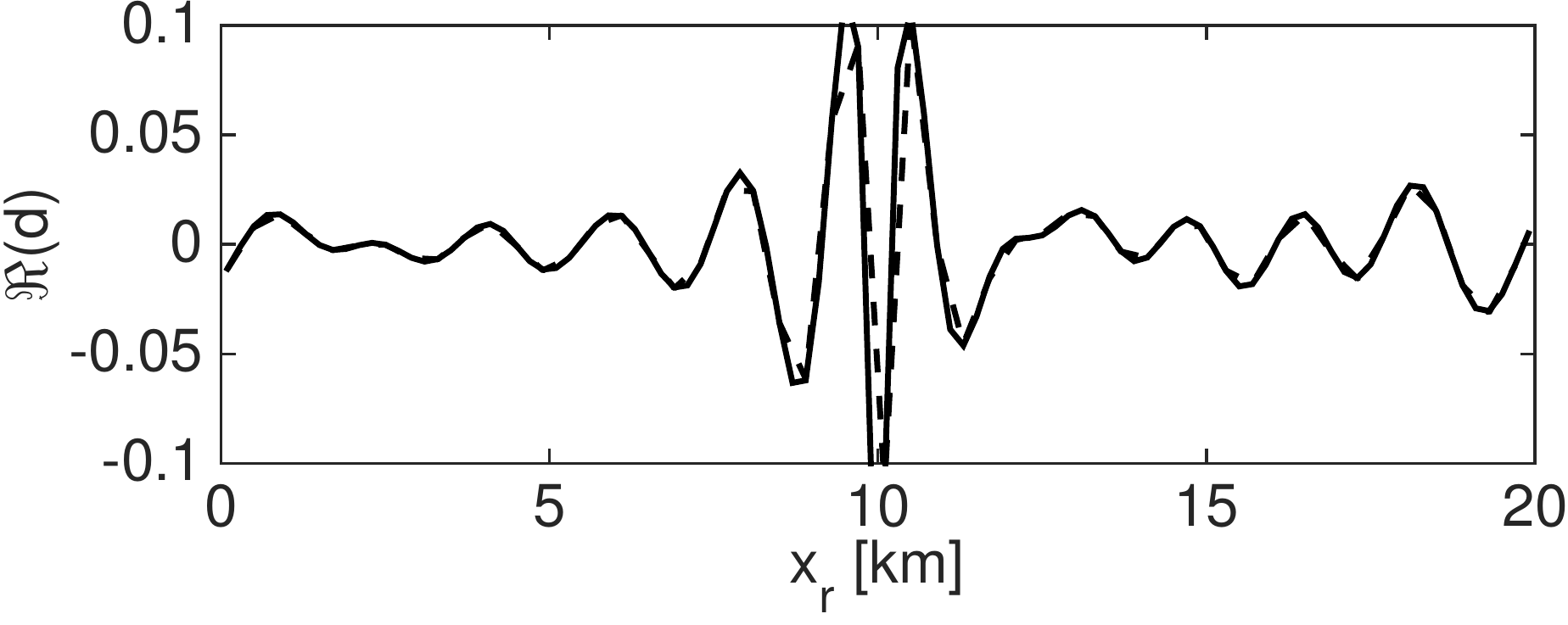}\\
{\small $\lambda=1$}&
{\small $\lambda=10$}\\
\end{tabular}
\caption{QN reconstructions after 50 iterations and corresponding data for a source in the center, starting from the initial iterate I. Both the penalty and reduced methods converge to the same final iterate when starting from this initial guess and are able to fit the data equally well.}
\label{fig:2D_overthrust1}
\end{figure}

\begin{figure}
\centering
\begin{tabular}{cc}
\includegraphics[scale=.3]{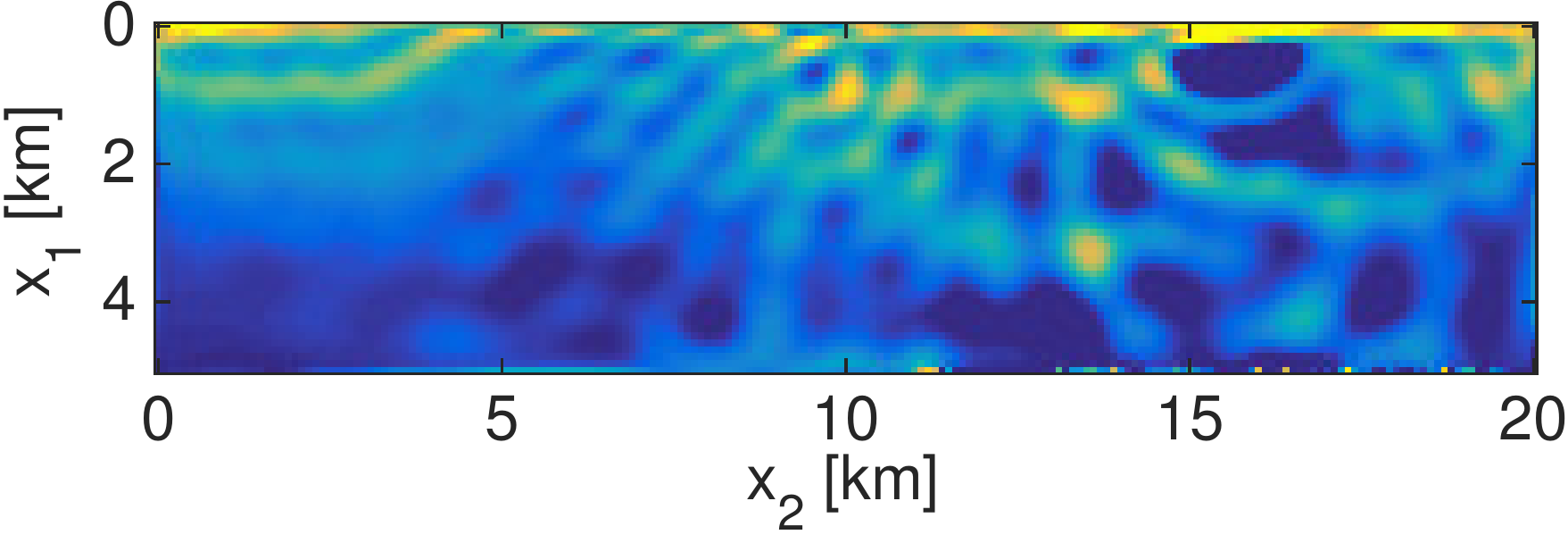}&
\includegraphics[scale=.3]{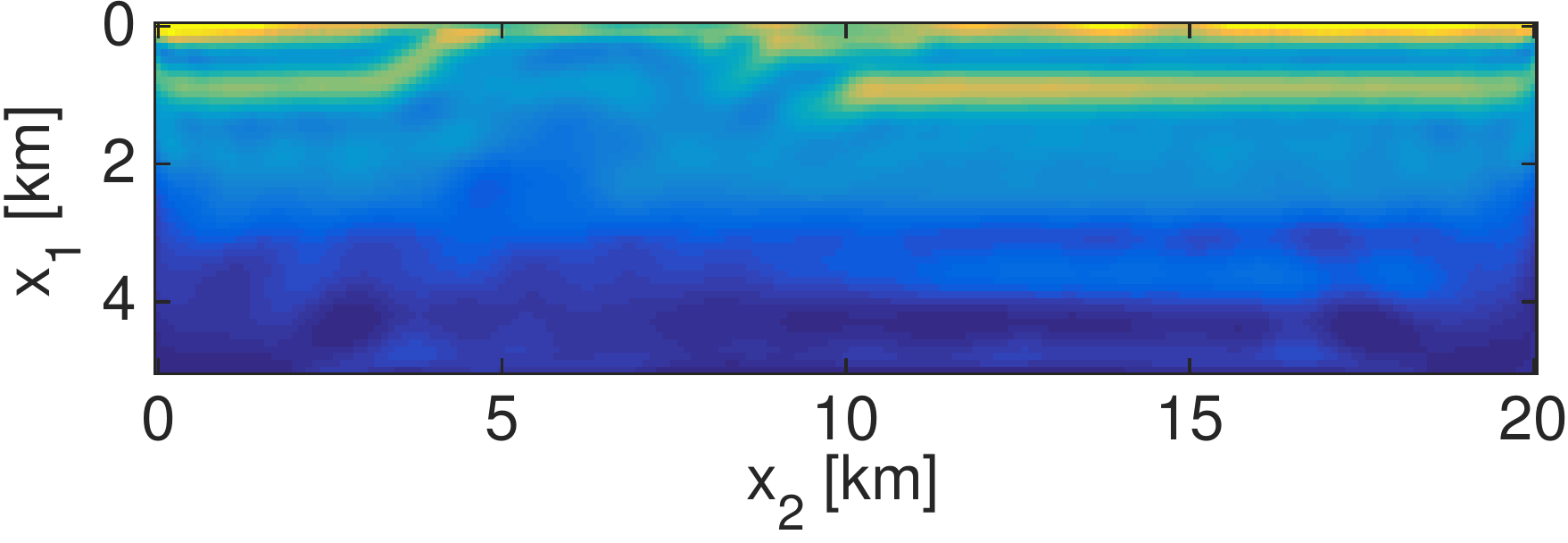}\\
\includegraphics[scale=.3]{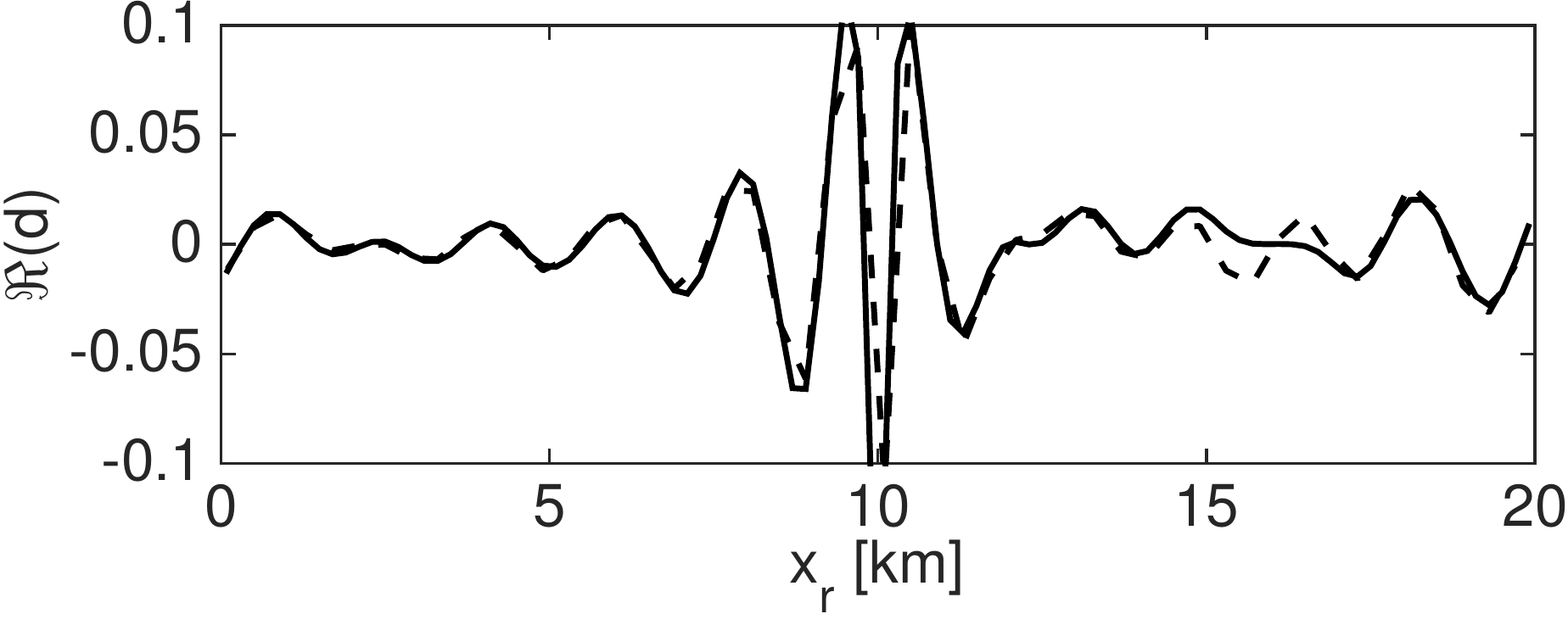}&
\includegraphics[scale=.3]{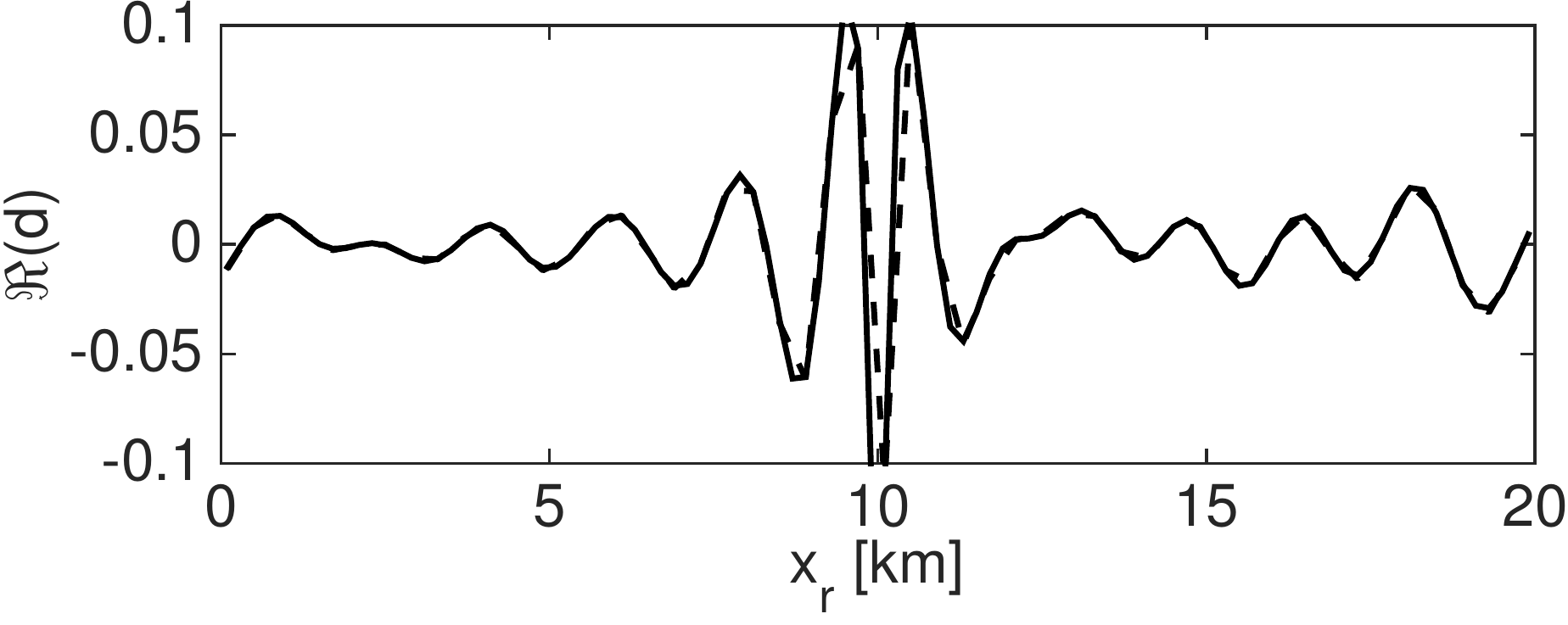}\\
{\small reduced}&
{\small $\lambda=0.1$}\\
\includegraphics[scale=.3]{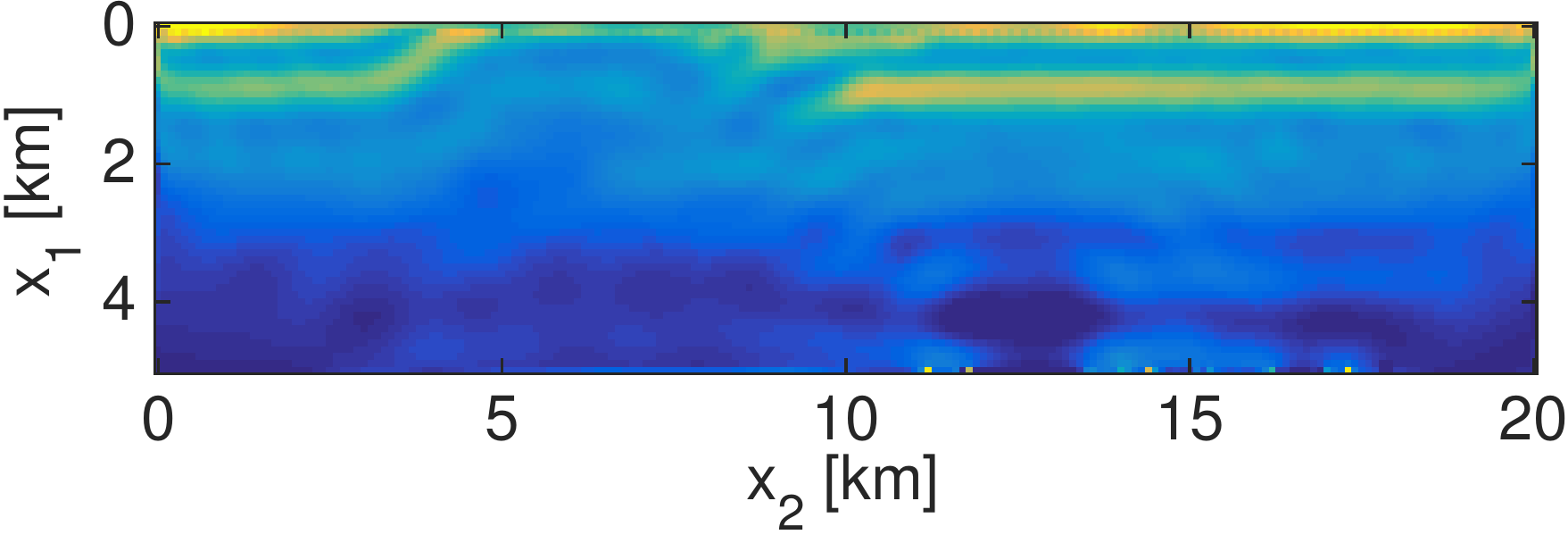}&
\includegraphics[scale=.3]{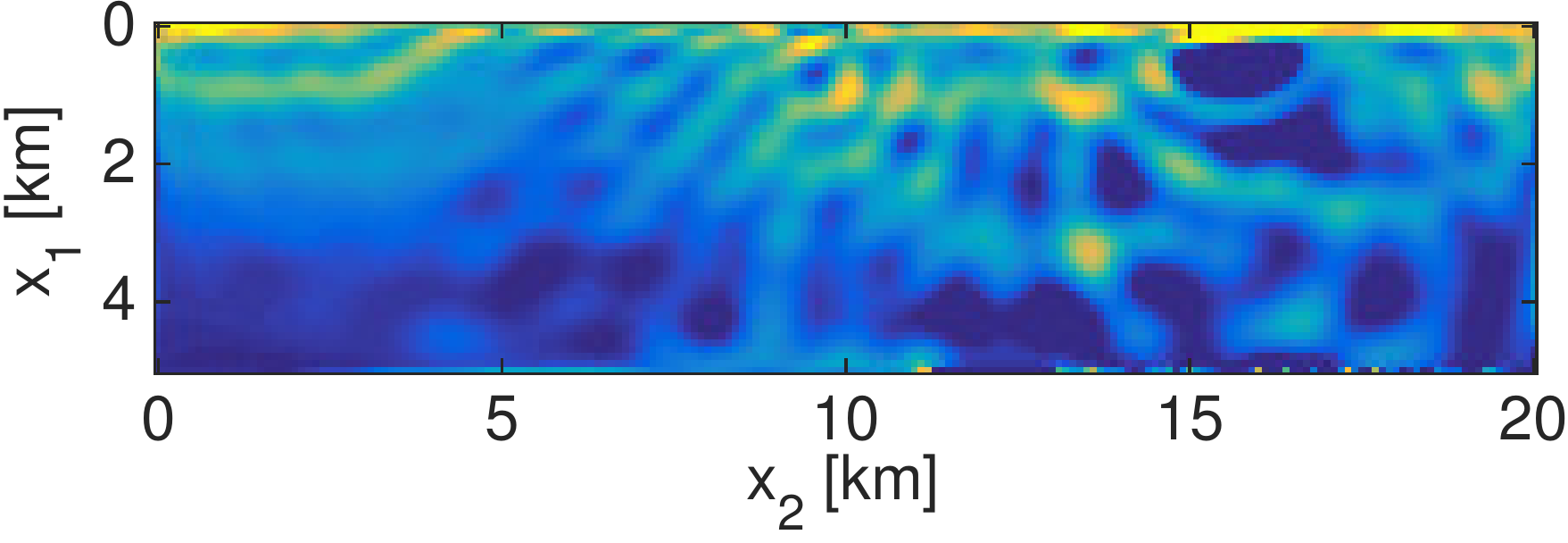}\\
\includegraphics[scale=.3]{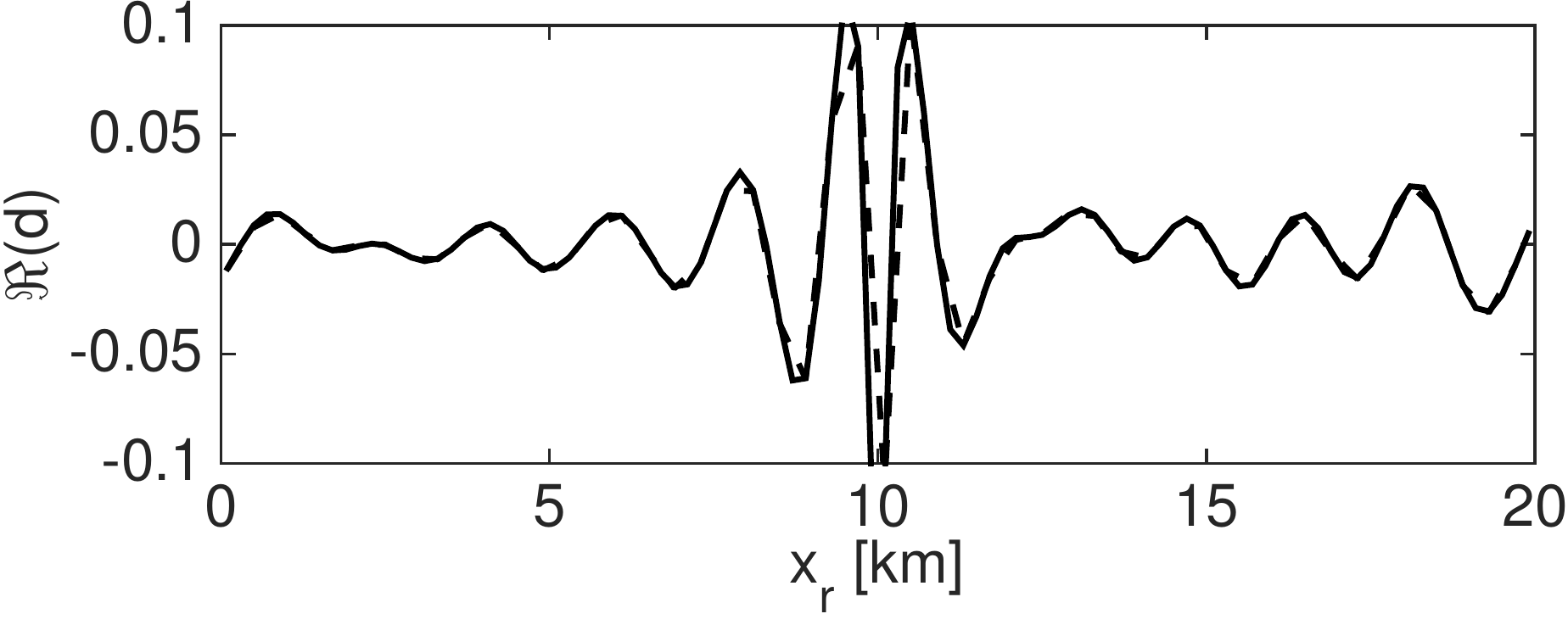}&
\includegraphics[scale=.3]{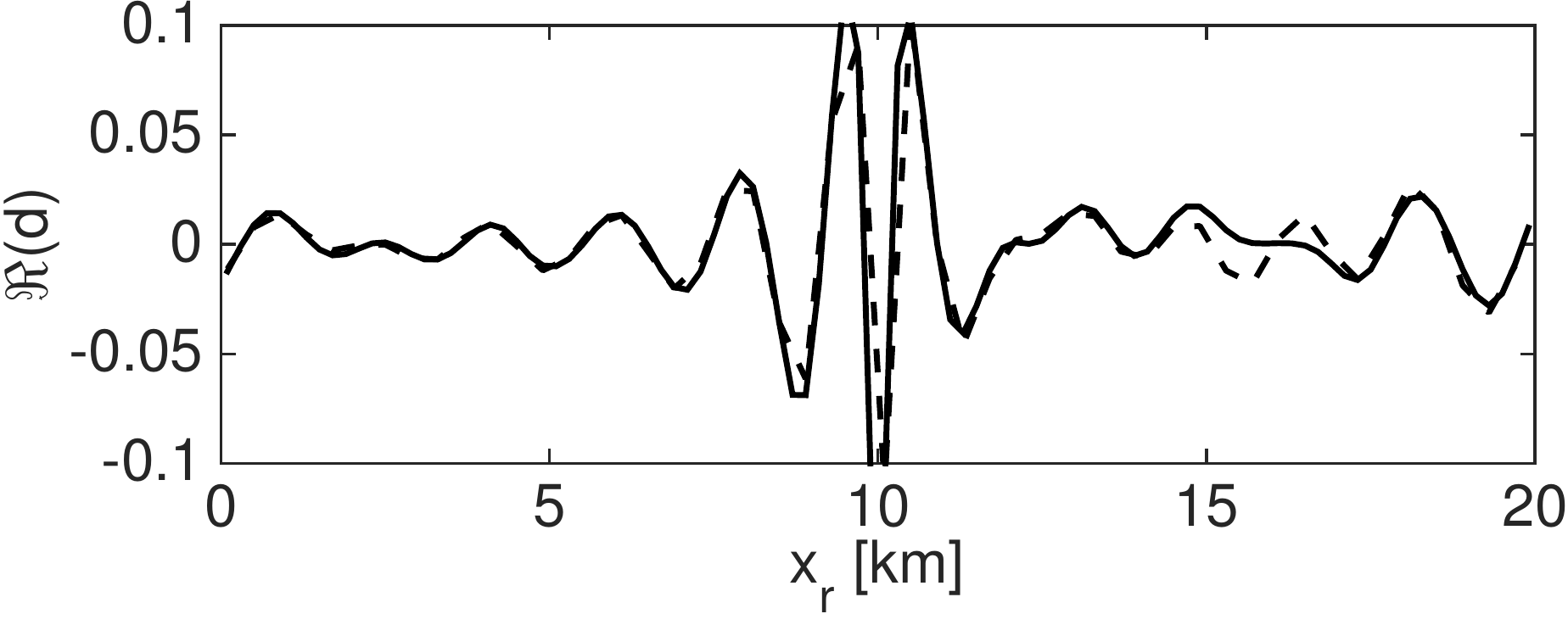}\\
{\small $\lambda=1$}&
{\small $\lambda=10$}\\
\end{tabular}
\caption{QN reconstructions after 50 iterations and corresponding data for a source in the center, starting from the initial iterate II. For small $\lambda$, the penalty method converges to the same final iterate as when starting from initial guess II, showing stability against changes in the initial guess. The reduced method converges to a completely different model, suggesting that the optimization method is stuck in a local minimum. This is confirmed when looking at the data-fit.}
\label{fig:2D_overthrust2}
\end{figure}

\begin{figure}
\centering
\begin{tabular}{cc}
\includegraphics[scale=.3]{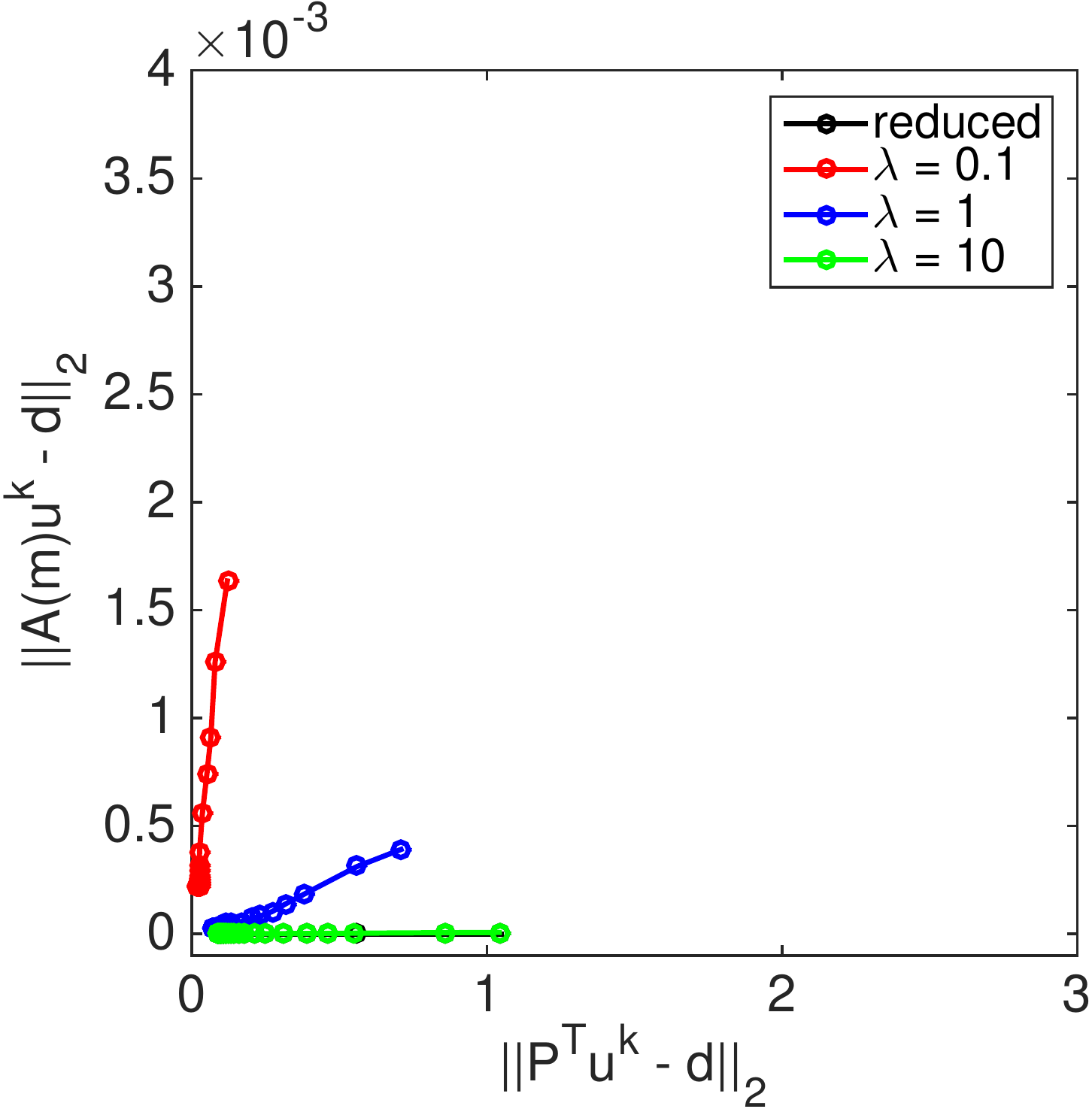}&
\includegraphics[scale=.3]{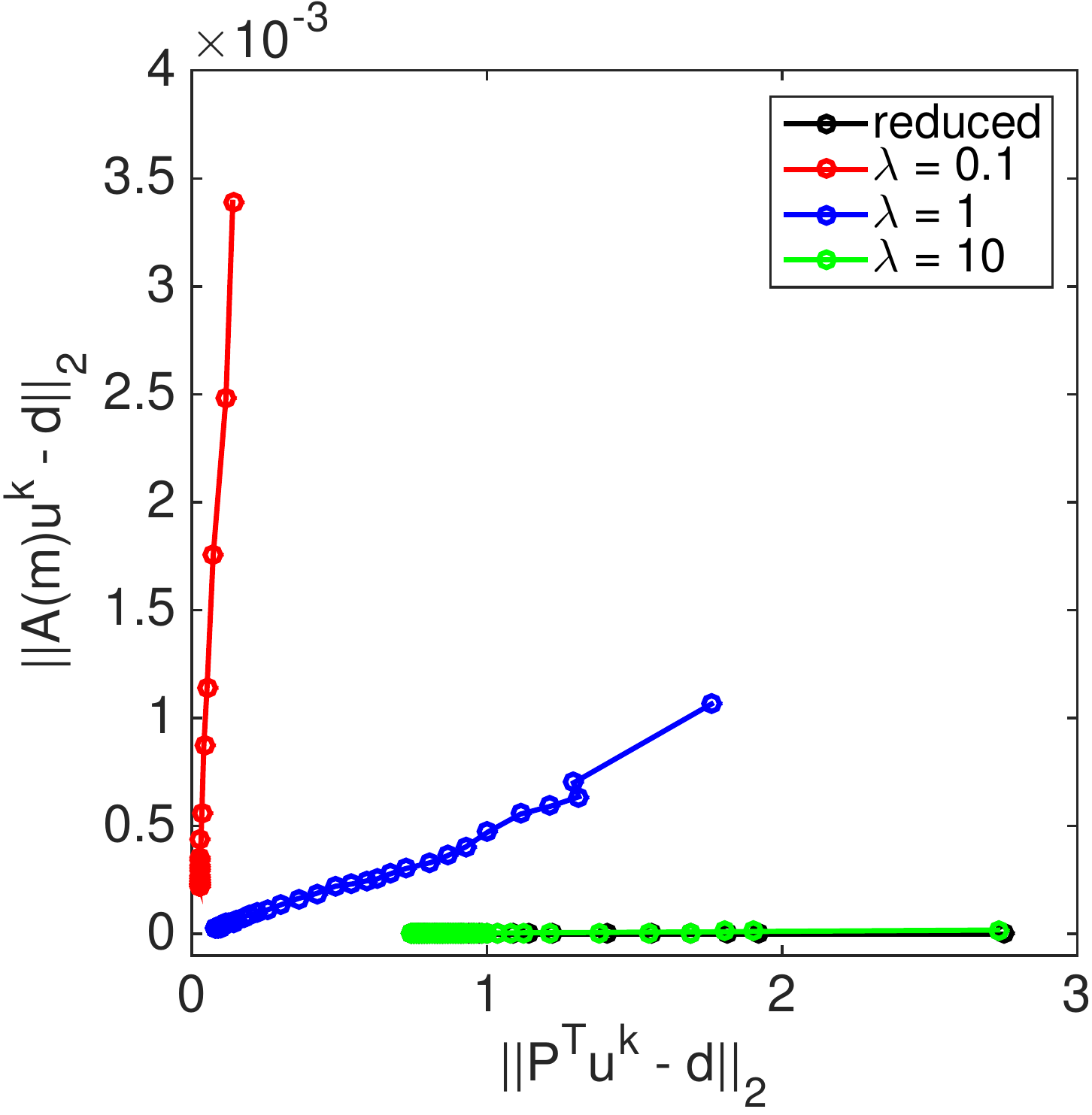}\\
\end{tabular}
\caption{Convergence history in terms of the data-fit and distance to the constraint, starting from initial iterate I (left) and starting from initial iterate II (right). These plots show that, for small $\lambda$, the penalty method is able to reduce both the data and PDE misfit to the same level when starting from either initial guess. The reduced method, however, cannot reduce the data misfit to the same level when starting from initial guess II, suggesting that it got stuck in a local minimum.}
\label{fig:2D_overthrust3}
\end{figure}


\clearpage

\bibliographystyle{unsrt}
\bibliography{mybib}

\end{document}